\documentclass[11pt]{article}
\usepackage[utf8]{inputenc}
\usepackage[style=alphabetic, 
maxbibnames=99, backend=biber, 
doi=false, isbn=false, url=false]{biblatex}

\usepackage{libertine}
\usepackage{setspace}
\usepackage{amsmath}
\usepackage{amssymb}
\usepackage{amsthm}

\usepackage{hyperref}

\usepackage{tikz-cd}
\usepackage{tikz}
 	 \usetikzlibrary{arrows,backgrounds}
 	 \newcommand{\tikzmath}[2][]
     {\vcenter{\hbox{\begin{tikzpicture}[#1]#2
                     \end{tikzpicture}}}
     }
\usepackage{import}

\newcommand{\figurerescalefactor}{1}

\newtheorem{theorem}{Theorem}

\newtheorem{corollary}[theorem]{Corollary}

\newtheorem{thm}{Theorem}[section]
\newtheorem*{thm*}{Theorem}
\newtheorem{prop}[thm]{Proposition}
\newtheorem{lem}[thm]{Lemma}
\newtheorem{cor}[thm]{Corollary}

\newtheorem*{que*}{Question}

\theoremstyle{definition}
\newtheorem{defn}[thm]{Definition}
\newtheorem{rem}[thm]{Remark}
\newtheorem{obs}[thm]{Observation}
\newtheorem{ex}[thm]{Example}

\usepackage{geometry}
\newcommand{\mc}[1]{\mathcal{#1}}
\newcommand{\mcA}{\mathcal{A}}
\newcommand{\mcB}{\mathcal{B}}
\newcommand{\mcC}{\mathcal{C}}
\newcommand{\mcD}{\mathcal{D}}
\newcommand{\mcE}{\mathcal{E}}
\newcommand{\mcF}{\mathcal{F}}

\newcommand{\mcI}{\mathcal{I}}

\newcommand{\mcP}{\mathcal{P}}
\newcommand{\mcS}{\mathcal{S}}
\newcommand{\mcR}{\mathcal{R}}

\newcommand{\mbf}[1]{\mathbf{#1}}

\newcommand{\bbt}{\mathbf{t}}

\newcommand{\IN}{\mathbb{N}}
\newcommand{\IZ}{\mathbb{Z}}
\newcommand{\IQ}{\mathbb{Q}}
\newcommand{\IR}{\mathbb{R}}

\newcommand{\mrm}[1]{\mathrm{#1}}

\newcommand{\Hom}{\operatorname{Hom}}
\newcommand{\Id}{\operatorname{Id}}
\newcommand{\id}{\mathrm{id}}
\newcommand{\Fun}{\operatorname{Fun}}

\newcommand{\Mon}{\operatorname{Mon}}
\newcommand{\Mor}{\operatorname{Mor}}

\newcommand{\Obj}{\operatorname{Obj}}
\newcommand{\Aut}{\operatorname{Aut}}
\DeclareMathOperator*{\colim}{colim}
\DeclareMathOperator*{\hocolim}{hocolim}

\newcommand{\pr}{\operatorname{pr}}

\newcommand{\Fin}{\operatorname{Fin}}
\newcommand{\Map}{\mathrm{Map}}

\newcommand{\SP}{\mathrm{SP}}

\newcommand{\Cob}{h\mathrm{Bord}}
\newcommand{\Cobtwo}{h\mathrm{Bord}_2}
\newcommand{\tCob}{\mathrm{Bord}}
\newcommand{\ICob}{\underline{h}\mathrm{Bord}}

\newcommand{\chilez}{{\chi\le0}}

\newcommand{\Cobn}{h\mathrm{Bord}_2^{\chilez}}
\newcommand{\Cobnpb}{h\mathrm{Bord}_2^{\chilez, \pb}}
\newcommand{\ICobn}{\ICob_2^{\chilez}}
\newcommand{\Csp}{{h\mathrm{Csp}}}
\newcommand{\ICsp}{\mathrm{Csp}}

\newcommand{\pb}{{\partial_+}}

\newcommand{\nc}{{\rm nc}}

\newcommand{\Cut}{\mathrm{Cut}}
\newcommand{\Cc}{\overline{\mathrm{Cut}}^{\con}}
\newcommand{\red}{{\rm red}}
\newcommand{\cl}{{\rm cl}}
\newcommand{\con}{{\rm con}}

\newcommand{\Sub}{\mathrm{Sub}}
\newcommand{\Fil}{\mathrm{Fil}}
\newcommand{\Diff}{\mathrm{Diff}}

\newcommand{\sSet}{\mathrm{sSet}}

\newcommand{\J}{\mathcal{J}}
\newcommand{\SubJ}{\smallint_{\J}\mathrm{Sub}}
\newcommand{\Subpi}{\mathrm{Sub}}
\newcommand{\Subtop}{\mathrm{Sub}^{\rm top}}

\newcommand{\Homeo}{\mathrm{Homeo}}
\newcommand{\Ar}{\mathrm{Ar}}

\newcommand{\amalglim}{\operatornamewithlimits{\amalg}}
\newcommand{\rmF}{\mathrm{F}}

\newcommand{\Top}{\mathrm{Top}}
\newcommand{\Set}{\mathrm{Set}}
\newcommand{\Cat}{\mathrm{Cat}}
\newcommand{\Gpd}{\mathrm{Gpd}}
\newcommand{\Kan}{\mathrm{Kan}}

\newcommand{\Ch}{\mathrm{Ch}}

\newcommand{\val}{\operatorname{val}}

\newcommand{\ol}[1]{\overline{#1}}
\newcommand{\UL}[1]{\underline{#1}}
\newcommand{\gle}[1]{\langle #1 \rangle}
\newcommand{\blank}{\underline{\ \ }}
\newcommand{\qand}{\quad \text{and} \quad}

\newcommand{\cd}{\bullet}
\newcommand{\oh}{\tfrac{1}{2}}
\newcommand{\oth}{\tfrac{1}{3}}
\newcommand{\tth}{\tfrac{2}{3}}

\newcommand{\GC}{\mathsf{GC}}
\newcommand{\fGC}{\mathsf{fGC}}

\newcommand{\ga}{\alpha}
\newcommand{\gb}{\beta}
\newcommand{\gc}{\gamma}
\newcommand{\gC}{\Gamma}
\newcommand{\gd}{\delta}
\newcommand{\gD}{\Delta}

\newcommand{\gi}{\iota}

\newcommand{\gl}{\lambda}
\newcommand{\gL}{\Lambda}
\newcommand{\gs}{\sigma}
\newcommand{\gS}{\Sigma}
\newcommand{\gt}{\theta}
\newcommand{\go}{\omega}
\newcommand{\gO}{\Omega}
\newcommand{\gp}{\varphi}

\newcommand{\ot}{\otimes}
\newcommand{\wh}[1]{\widehat{#1}}

\newcommand{\op}{{\rm op}}

\newcommand{\oneloop}{\tikzmath{
    \draw (0,0) arc (-90:-450:.13cm);
    \filldraw (0,0) circle (.03cm);
    }}

\title{The surface category and tropical curves}

\author{Jan Steinebrunner}

\begin{document}

\maketitle

\begin{abstract}
    We compute the classifying space of the surface category $\Cobtwo$ whose objects are closed oriented $1$-manifolds and whose morphisms are diffeomorphism classes of oriented surface bordisms, and show that it is rationally equivalent to a circle.
    It is hence much smaller than the classifying space of the topologically enriched surface category $\tCob_2$ studied by Galatius-Madsen-Tillmann-Weiss.
    
    However, we also show that for the wide subcategory $\Cobn \subset \Cobtwo$ 
    that contains all morphisms without disks or spheres, 
    the classifying space $B\Cobn$ is surprisingly large.
    Its rational homotopy groups contain the homology of 
    all moduli spaces of tropical curves $\gD_g$ as a summand.
    
    The technical key result shows that a version of positive boundary surgery 
    applies to a large class of discrete symmetric monoidal categories, 
    which we call \emph{labelled cospan categories}.
    We also use this to show that the $(2,1)$-category of cospans of finite sets 
    has a contractible classifying space.
\end{abstract}

\section{Introduction}

The main object of study of this paper is the surface category $\Cobtwo$.
Objects in $\Cobtwo$ are closed oriented $1$-manifolds and 
morphisms are compact oriented surface cobordisms, up to boundary-preserving diffeomorphisms.
This category is best known as the 
source of $2$-dimensional topological field theories.

Motivated by Segal's work on conformal field theories \cite{Seg04},
Madsen and Tillmann \cite{MT01} define a topological category $\mcC_2$, 
where the space of morphisms $M \to N$ is the moduli space of surfaces with boundary $M^- \amalg N$.
This category has proven to be a useful tool for studying the moduli spaces $B\Diff(\Sigma_g)$ in the large genus limit $g \to \infty$.
In their seminal paper \cite{GMTW06} Galatius, Madsen, Tillmann, and Weiss
compute the classifying space of $\tCob_2$ as $B(\tCob_2) \simeq \Omega^{\infty-1}\mathit{MTSO}_2$,
and deduce an alternative proof of the Mumford conjecture, which was originally proven in \cite{MW07}.

In this paper we compute the classifying space $B(\Cobtwo)$ and some closely related spaces.
We will refer to $\Cobtwo$ as the \emph{truncated} surface category in order to distinguish
it from the topological surface category $\tCob_2$. 
Indeed, the functor $\tCob_2 \to \Cobtwo$ that sends a surface cobordism to its diffeomorphism class
identifies $\Cobtwo$ as the homotopy-category of $\tCob_2$.
However, this functor is very far from being an equivalence:
it collapses each of the moduli spaces $B\Diff(\Sigma_{g,k})$ to a point.
The category $\Cobtwo$ 
does not ``know'' about diffeomorphisms, 
but only about the combinatorics of how surfaces can be glued.

The above might suggest that the classifying space of the truncated surface category $\Cobtwo$ 
is much simpler than its topological analogue $\tCob_2$.
Indeed, our Theorem \ref{theorem:ICob2} implies that $B(\Cobtwo)$ is rationally equivalent to a circle.
However, we will see that this is only the case because the disk morphism $D^2:\emptyset \to S^1$ 
yields a trivial nullhomotopy of a certain obstruction category $\mcF_g$ described below.
We hence restrict our attention to the wide subcategory $\Cobn \subset \Cobtwo$
where cobordisms are not allowed to have connected components that are disks or spheres.
For this subcategory we show:

\begin{theorem}\label{theorem:rational-homotopy}
    The rational homotopy groups of $B(\Cobn)$ are
    \[
        \pi_*^\IQ B(\Cobn) \cong 
        \IQ\gle{\ga} \oplus \IQ\gle{\rho_1, \rho_2, \dots }
        \oplus \bigoplus_{g \ge 2} 
        H_*( \gS^2 \gD_{g}; \IQ )
    \]
    where $|\ga| = 1$, $|\rho_i| = 4i+2$, and $\gD_g$ is the moduli space of tropical curves of genus $g$ and volume $1$.
\end{theorem}

One dimension below, the homotopy type of $B(\Cob_1)$ was determined in \cite[Theorem A]{Stb21} as 
\[
    B(\Cob_1) \simeq \Omega^{\infty-2}(\mathit{MTSO}_2 \to H\mathbb{Z}) ,
\]
which is also larger than the corresponding $B(\tCob_1) \simeq \Omega^\infty \mathbb{S}$ that we know from \cite{GMTW06}.
In higher dimension, $d\ge 3$ the author does not expect $B(\Cob_d)$ to have a meaningful description:
for example, the main theorem of \cite{Juhsz2018} suggests that already $B(\Cob_3)$ is extremely complicated.

\subsection*{The classifying space of the surface category}
Since $\Cobtwo$ is a symmetric monoidal category its classifying space $B\Cobtwo$ 
admits the structure of an infinite loop space. 
It will be useful to keep track of this structure and all our computations respect it.
The only previously known computation regarding $B(\Cobtwo)$ is a theorem
of Tillmann \cite{Til96}, which shows 
that there is an equivalence of infinite loop spaces $B(\Cobtwo) \simeq S^1 \times X$, 
where $X$ is some simply connected infinite loop space.
It was conjectured \cite[Conjecture 5.3]{JT13} that $X$ is contractible,
and more generally that the classifying spaces
of $\Cobtwo$ and several related categories should be $1$-types.
One of our main theorems states that this is \emph{almost} true:
\begin{theorem}\label{theorem:ICob2}
    There is an equivalence of infinite loop spaces $B(\ICob_2) \simeq S^1$.
\end{theorem}
Here $\ICob_2$ is a $(2,1)$-category that we will think of as a refinement
of the truncated cobordism category in which closed components are counted ``properly''.
In general it is defined as follows:
\begin{defn}
    For all $d\ge 0$ the $2$-category $\ICob_d$ has 
    as objects closed oriented $(d-1)$-manifolds and
    morphisms $W:M \to N$ are compact oriented cobordisms from $M$ to $N$.
    A $2$-morphism $\ga:W \Rightarrow W'$ is a bijection $\ga:\pi_0(W) \cong \pi_0(W')$
    such that there exists a diffeomorphism $\gp:W \cong W'$
    that is the identity on the boundary and satisfies $\pi_0(\gp) = \ga$.
\end{defn}
It is important to note that the diffeomorphism $\gp$ is not part of 
the data of the $2$-morphism $\ga:W \Rightarrow W'$.
There are always at most finitely many $2$-morphisms $W \Rightarrow W'$
and if $W$ has no closed components, then there is at most one.
From this it follows that the quotient map 
$B(\ICob_d) \to B(\Cob_d)$ is a rational equivalence.
For $d=2$ we construct an equivalence of infinite loop spaces:
    \[
    B(\Cobtwo) \simeq B(\ICob_2) \times \tau_{\ge 3}Q\left(\bigvee\nolimits_{g \ge 0}S^2\right).
    \]
Here $\tau_{\ge 3} Q(\bigvee_{g\ge0} S^2)$ denotes the $2$-connected cover of the free infinite loop space on $\bigvee_{g\ge0} S^2$.
Together with Theorem \ref{theorem:ICob2} this identifies the mysterious infinite loop spaces $X$ in Tillmann's theorem.
\begin{corollary}
    There is an equivalence of infinite loop spaces 
    \[B(\Cob_2) \simeq S^1 \times
     \tau_{\ge 3}Q\left(\bigvee\nolimits_{g \ge 0}S^2\right).
    \]
\end{corollary}
In particular, $X$ is rationally trivial and can be thought of as an ``error-term'' coming from the fact that in $\Cobtwo$ we are not counting the closed components ``properly''.

\subsection*{The surface category without disks and the moduli space of tropical curves}
The above result seems to suggest that
-- as long as one looks at the slightly refined $\ICob_2$ --
the classifying space of the surface category is not very interesting.
However, in the proof of Theorem \ref{theorem:ICob2} we will see 
that this is only the case because several interesting subspaces of $B(\ICob_2)$
happen to admit a null-homotopy. 
This uses the fact that for any surface cobordism $W:M \to N$ 
and point $p \in W$ one may decompose $W = W' \cup_{S^1} D^2$ where $W'$ 
is obtained from $W$ by removing a ball around $p$.
If one restricts to the following subcategory,
this is no longer possible and the homotopy type becomes more interesting.

\begin{defn}
    The subcategory $\Cobn \subset \Cobtwo$ is defined to contain all objects,
    but only those morphisms $W:M \to N$ where
    no connected component of $W$ is a disk or a sphere.
\end{defn}

Equivalently, one can require that every component $V \subset W$
has non-positive Euler characteristic, i.e.\ $\chi(V) \le 0$.
Using this reformulation one checks that 
$\Cobn$ is closed under both composition and disjoint union.%
\footnote{
    One could also consider another, even larger subcategory 
    $\Cobn \subset \Cobtwo^{\rm no\ disks} \subset \Cobtwo$ where spheres are allowed.
    This simply splits as a product $\Cobtwo^{\rm no\ disks} \cong \Cobn \times \IN$,
    where the second factor is $\IN$ thought of as a category with one object, counting the number of spheres.
}
By passing to this subcategory we uncover the rich homotopical structure 
that was lurking behind the seemingly simple statement of Theorem \ref{theorem:ICob2}:

\begin{theorem}\label{theorem:BICobn}
    There is an equivalence of infinite loop spaces
    \[
        B(\ICobn) 
        \simeq S^1 \times Q(\gS^2 BO(2))
        \times Q\left(\bigvee\nolimits_{g \ge 2} \gS^2 B\J_g\right).
    \]
\end{theorem}

The spaces $B\J_g$ are the classifying spaces of certain finite categories 
of stable graphs, which have appeared in the study of tropical curves.
In fact, we construct a rational homology isomorphism $B\J_g \to \Delta_g$,
where $\Delta_g$ is the moduli space of tropical curves of genus $g$ and volume $1$.
The rational homology of $\gD_g$ was described by \cite{CGP18}
in terms of Kontsevich's commutative graph complex. 
Following \cite{Wil15} we let $\fGC_{0,\rm conn}$ denote the graph complex
of connected graphs without tadpoles, where the degree of a graph
is given by its number of edges. 
Combining these results we obtain:
\begin{corollary}\label{corollary:pi_*ICobn}
    The rational homotopy groups of $B(\ICobn)$, or equivalently of $B(\Cobn)$, are:
    \begin{align*}
        \pi_*^\IQ B(\Cobn) 
        &\cong  \IQ\gle{\alpha} \oplus H_{*-1}( \fGC_{0, \rm conn} ).
    \end{align*}
\end{corollary}
Here $\alpha$ is a generator of $\pi_1 B(\Cobn) \cong \IZ$.
The identification of the rational homotopy of $\Cobn$ 
with the homology of the graph complex leads to exponential lower bounds on its dimension:
as observed in \cite{CGP18} it follows from Willwacher's computation 
$H^0(\GC_2) = \mathfrak{grt}_1$ in \cite{Wil15} that 
$\dim_\IQ H_{2g}(\fGC_{0,\rm conn}) \ge \dim_\IQ H_0(\GC_2^{g\rm-loop}) > 1.3^g$ 
for $g\gg 0$.
To better understand the connection between $\Cobn$ and the $\Delta_g$,
we will construct an explicit comparison map,
illustrated in figure \ref{fig:mu}.
Let $\SP^\infty(X)$ denote the infinite symmetric power on a based space $X$; 
this is the free topological commutative monoid on $X$.
By \cite{DT58} we have $\pi_k(\SP^\infty(X)) \cong \tilde{H}_k(X; \IZ)$.
\begin{theorem}\label{theorem:the-map}
    There is an explicit continuous map of partially defined commutative topological monoids
    \[
        \mu: B(\Cobn) \longrightarrow
        \SP^\infty\left(\Sigma^2 \gD_2 \vee \Sigma^2 \gD_3 \vee \Sigma^2 \gD_4 \vee \dots \right)
    \]
    such that up to rational homotopy equivalence 
    it corresponds to the projection to the third factor in Theorem \ref{theorem:BICobn}.
    In particular, the map on homotopy groups 
    $\pi_* B(\Cobn) \to \bigoplus_{g \ge 2} \tilde{H}_{*-2}(\gD_g)$
    rationally agrees with the map from Theorem \ref{theorem:rational-homotopy}.
\end{theorem}

\begin{figure}[ht]
    \centering
    \def\svgwidth{.9\linewidth}
    \small
    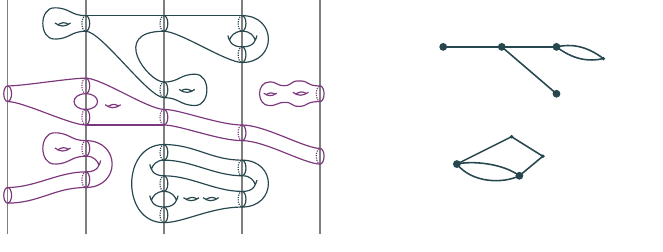
    \caption{An example of how the map $\mu$ (defined in \ref{defn:mu})
    can be evaluated on a $4$-simplex in $B(\Cobn)$.
    The $4$-simplex is parametrised by $(t_0, t_1, t_2, t_3, t_4) \in [0,1]^5$ with $\sum t_i = 1$.
    The double-suspension $\Sigma^2\gD_g$ is given by triples $[(G,w,d), a, b]$
    where $a,b \in [0,1]$ with $a+b \le 1$ 
    and $(G,w,d)$ a stable metric graph of genus $g$ and volume $1-a-b$.
    This is identified with the base-point if $a=0$ or $b=0$.
    To evaluate $\mu$ we sum over closed components $V$ of the diagram,
    discarding components with boundary. 
    To each $V$ we assign a stable metric graph 
    with an edge of length $t_i/k_i$ for every time $V$ intersects the $i$th vertical line,
    where $V$ intersects the $i$th line $k_i$ times.
    To be precise, valence $2$ genus $0$ vertices should be deleted and the length
    of their adjacent edges added.
    The coordinates $(a,b)$ in the suspension $\Sigma^2\gD_g$ are given 
    by the sum of the $t_i$ ``before'' and ``after'' $V$, respectively.}
    \label{fig:mu}
\end{figure}

\subsection*{Labelled cospan categories}

The techniques used to prove the main theorems stated so far apply 
in much greater generality than just to $\Cobtwo$ and its subcategories.
In fact, the main property of $\Cobtwo$ that we are using is
that we can talk about connected components of objects and morphisms.
One way of formalising this is to say that $\pi_0$ defines 
a symmetric monoidal functor
\[
    \pi_0:\Cobtwo \to \Csp, \quad M \mapsto \pi_0(M), \quad
    ([W]:M \to N) \mapsto [\pi_0(M) \to \pi_0(W) \leftarrow \pi_0(N)]
\]
from $\Cobtwo$ to the category $\Csp$ of cospans in finite sets.
This category $\Csp$ has as objects finite sets $A$ and 
as morphisms $A \to B$ it has isomorphism classes of cospans $[A \to X \leftarrow B]$.
Two cospans are isomorphic if there is a bijection $\gs:X \cong X'$
compatible with the maps from $A$ and $B$.
We equip $\Csp$ with the symmetric monoidal structure coming from disjoint union.

Since connected surfaces are uniquely determined by their genus
and number of boundary components,
giving a morphism $M \to N$ in $\Cobtwo$ amounts to the same data 
as giving a cospan $[\pi_0(M) \to X \leftarrow \pi_0(N)]$
and a labelling $g:X \to \IN$ that encodes the genus of each component.
It will be a very useful perspective to think of $\Cobtwo$ 
as a category of cospans labelled by natural numbers 
as indicated in figure \ref{fig:Cob2=N-labelled-Csp}.
\begin{figure}[ht]
    \centering
    \small
    \def\svgwidth{.55\linewidth}
\begingroup%
  \makeatletter%
  \providecommand\color[2][]{%
    \errmessage{(Inkscape) Color is used for the text in Inkscape, but the package 'color.sty' is not loaded}%
    \renewcommand\color[2][]{}%
  }%
  \providecommand\transparent[1]{%
    \errmessage{(Inkscape) Transparency is used (non-zero) for the text in Inkscape, but the package 'transparent.sty' is not loaded}%
    \renewcommand\transparent[1]{}%
  }%
  \providecommand\rotatebox[2]{#2}%
  \newcommand*\fsize{\dimexpr\f@size pt\relax}%
  \newcommand*\lineheight[1]{\fontsize{\fsize}{#1\fsize}\selectfont}%
  \ifx\svgwidth\undefined%
    \setlength{\unitlength}{96.75028763bp}%
    \ifx\svgscale\undefined%
      \relax%
    \else%
      \setlength{\unitlength}{\unitlength * \real{\svgscale}}%
    \fi%
  \else%
    \setlength{\unitlength}{\svgwidth}%
  \fi%
  \global\let\svgwidth\undefined%
  \global\let\svgscale\undefined%
  \makeatother%
  \begin{picture}(1,0.52946192)%
    \lineheight{1}%
    \setlength\tabcolsep{0pt}%
    \put(0,0){\includegraphics[width=\unitlength,page=1]{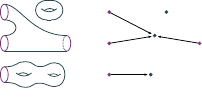}}%
    \put(0.76322178,0.12844344){\color[rgb]{0,0,0}\makebox(0,0)[lt]{\lineheight{1.25}\smash{\begin{tabular}[t]{l}$2$\end{tabular}}}}%
    \put(0.83901669,0.44084482){\color[rgb]{0,0,0}\makebox(0,0)[lt]{\lineheight{1.25}\smash{\begin{tabular}[t]{l}$1$\end{tabular}}}}%
    \put(0.00194096,0.00426962){\color[rgb]{0,0,0}\makebox(0,0)[lt]{\lineheight{1.25}\smash{\begin{tabular}[t]{l}$M \xrightarrow{\qquad\ W\ \qquad} N$ \end{tabular}}}}%
    \put(0.50868034,0.00426962){\color[rgb]{0,0,0}\makebox(0,0)[lt]{\lineheight{1.25}\smash{\begin{tabular}[t]{l}$\pi_0(M) \xrightarrow{\quad} \pi_0(W) \xleftarrow{\quad} \pi_0(N)$ \end{tabular}}}}%
  \end{picture}%
\endgroup%

    \caption{A morphism in $\Cobtwo$ can be thought of as a cospan of finite sets labelled in $\IN$.}
    \label{fig:Cob2=N-labelled-Csp}
\end{figure}
Formalising this leads to the definition of labelled cospan categories.
The reader is referred to section \ref{sec:labelled-cospans}, 
which serves as an introduction to this paper from the perspective of labelled cospan categories.
\begin{defn}
    A \emph{labelled cospan category}%
    \footnote{
        It turns out that labelled cospan categories encode the same data as coloured properads, see theorem \ref{conj:labelled-Csp=properads}, which was conjecture in the first version of this article and then proven in \cite{BH22, BS22}.
    }
    is a symmetric monoidal category $\mcC$
    together with a symmetric monoidal functor $\pi:\mcC \to \Csp$ 
    satisfying four axioms, which ensure that every object 
    and morphism in $\mcC$ uniquely decomposes as the disjoint union 
    of its connected components. (See definition \ref{defn:labelled-cospans} for details.)
\end{defn}

In the study of topologically enriched cobordism categories 
a key theorem is that the inclusion of the positive boundary
subcategory $\tCob_d^\pb \subset \tCob_d$ induces an equivalence on classifying spaces.
(See \cite[Theorem 6.1]{GMTW06} and \cite[Theorem 3.1]{GRW14}.)
Here the positive boundary subcategory contains all objects
and all those cobordisms $W:M \to N$ where each component $V \subset W$ 
intersects $N$ non-trivially, i.e.\ has positive boundary.
We can define a positive boundary subcategory $\mcC^\pb \subset \mcC$
for any labelled cospan category $(\mcC \to \Csp)$ by allowing those
morphisms $f:x \to y$ where the right-hand arrow in the cospan
$[\pi(x) \to \pi(f) \leftarrow \pi(y)]$ is surjective.

For truncated cobordism categories the positive boundary category 
is often much easier to compute.
For example we have $B(\Cob_1^\pb) \simeq QS^0$ 
by \cite[Corollary 4.3.2]{Rap14},
$B(\Cobtwo^\pb)  \simeq S^1$ by \cite{Til96},
and $B(\Cobtwo^{\chi\le0, \pb}) \simeq S^1$ and $B(\Csp^\pb) \simeq *$ 
by proposition \ref{prop:computing-pb}.
However, we cannot generally expect the inclusion 
of the positive boundary category to be an equivalence.
After all $B(\Cob_1)$ is more complicated than $QS^0$ by \cite{Stb21}
and $B(\Cobn)$ is more complicated than $S^1$ by Theorem \ref{theorem:BICobn}.

The key theorem about the classifying spaces of labelled cospan categories
will describe the failure of $B(\mcC^\pb) \to B(\mcC)$ to be an equivalence.  
For this theorem to be stated in its most natural form we 
replace $\mcC$ by a $(2,1)$-category $\ICsp(\mcC)$,
in analogy with how we replaced $\Cobtwo$ by $\ICob_2$.
This recovers the previous construction as $\ICob_2 \simeq \ICsp(\Cobtwo)$
and just like before it does not change the classifying space 
up to rational equivalence: $B(\ICsp(\mcC)) \simeq_\IQ B(\mcC)$.

\begin{theorem}[Decomposition and Surgery Theorem]\label{theorem:decomposition-and-surgery}
    Let $(\mcC \to \Csp)$ be a labelled cospan category
    that admits surgery (definition \ref{defn:admits-surgery})
    and assume that $B(\mcC)$ is group-complete.
    Then there is a fiber sequence of infinite loop spaces:
    \[
        B(\mcC^\pb) \longrightarrow B(\ICsp(\mcC)) \longrightarrow
        Q\left(\bigvee\nolimits_g S^2 (B \mcF_g(\mcC)) \right).
    \]
    Here the wedge runs over all connected morphisms $g:1_\mcC \to 1_\mcC$,
    $\mcF_g(\mcC)$ denotes the category of non-trivial factorisations
    $(g:1_\mcC \to x \to 1_\mcC)$ from definition \ref{defn:F_g},
    and $S^2$ denotes the unreduced double-suspension.
\end{theorem}

The condition of admitting surgery is satisfied by all truncated cobordism categories $\Cob_d$ with $d \ge 2$
as well as by $\Csp$ and various labelled cospan categories 
that one can construct by labelling cospans in an abelian monoid.
The proof of Theorem \ref{theorem:decomposition-and-surgery} takes up the majority of this paper
(section \ref{sec:decomposition} and \ref{sec:surgery}).
Once Theorem \ref{theorem:decomposition-and-surgery} is established
most other computations amount to determining the homotopy types
of the factorisation categories $\mcF_g(\mcC)$.
In the case of $\mcC = \Csp$ we show that $B\mcF(\Csp)$ is contractible.
This concludes the computation started in \cite{Stb19}.

\begin{theorem}\label{theorem:BICsp}
    The classifying space of the $(2,1)$-category of cospans $\ICsp$ 
    is contractible and the classifying space of its homotopy category
    is $B(\Csp) \simeq \tau_{\ge 3} Q(S^2)$.
\end{theorem}

We further show that $B\mcF_{[W]}(\Cob_d)$ is contractible for any closed $d$-manifold $W$,
$d \ge 2$.
This makes use of the disk morphisms $D^d: \emptyset \to S^{d-1}$,
which were not available in $\Cobn$.
This implies:
\begin{theorem}\label{theorem:BCobd}
    For any dimension $d\ge2$ 
    the classifying space of $\ICob_d$ is equivalent to the classifying space 
    of its positive boundary subcategory:
    $
        B(\Cob_d^\pb) \simeq B(\ICob_d)
    $.
\end{theorem}

The classifying spaces of factorisation categories are not always trivial, though.
\cite[Section 8]{Stb21} implies that $\mcF_{S^1}(\Cob_1)$ 
is equivalent to Connes' cyclic category $\gL$.
Even though $\Cob_1$ does not admit surgery in the sense 
of definition \ref{defn:admits-surgery},
the conclusion of Theorem \ref{theorem:decomposition-and-surgery} 
still seems to apply and the map $B\ICsp(\Cob_1) \to Q(\gS^2 (B\mcF_{S^1}(\Cob_1)))$
corresponds to the continuation of the reduction fiber sequence 
that we constructed in \cite[Section 7]{Stb21}.

The subcategory $\Cobn \subset \Cobtwo$ is another example where 
the arguments for the contractibility of $\mcF_g$ fail
as we cannot use the disk morphism $D^2:\emptyset \to S^1$.
In section \ref{sec:J=gD} we prove Theorem \ref{theorem:BICobn} 
by showing that there is a zig-zag of functors inducing an equivalence 
of classifying spaces $B(\mcF_g(\Cobn)) \simeq B(\J_g)$ for all $g \ge 2$.

\subsection*{Outline}
This paper is split into two parts.
In sections \ref{sec:labelled-cospans}, \ref{sec:decomposition}, and \ref{sec:surgery} 
we study the general theory of labelled cospan categories,
and in sections \ref{sec:mcF-and-mcJ} and \ref{sec:J=gD} we focus on the case of $\Cobn$ 
and its relation to $\gD_g$.

We begin in section \ref{sec:labelled-cospans} by defining labelled cospan categories and giving general constructions such as the refinement $\ICsp(\mcC)$.
Section \ref{sec:decomposition} proves the decomposition theorem, which forms the first part of theorem \ref{theorem:decomposition-and-surgery}.
First we establish a general fiber sequence criterion
and a base-change theorem that will be useful throughout,
these are then applied to various spaces of cuts similar to \cite[Section 7]{Stb21}.
In the end of the section we also show that several factorisation categories 
are contractible.
Section \ref{sec:surgery} proves the surgery theorem,
completing Theorem \ref{theorem:decomposition-and-surgery},
by applying ideas of \cite{Gal11} and \cite{GRW14} to labelled cospan categories.
We also compute the positive boundary category of any weighted cospan category by generalizing \cite{Til96}.
At this point Theorems \ref{theorem:ICob2}, \ref{theorem:BICsp}, 
and \ref{theorem:BCobd} follow.

Section \ref{sec:mcF-and-mcJ} relates the factorisation category $\mcF_g(\Cobn)$ to the category of graphs $\J_g$ by constructing a category of filtered graphs
$\int_\J\Sub'$ that admits functors to both categoires and show that these functors induce equivalences on classifying spaces.
Together with the surgery and decomposition theorem this implies Theorem \ref{theorem:BICobn} and \ref{theorem:rational-homotopy}.
Section \ref{sec:J=gD} recalls the moduli space $\gD_g$ and
constructs a rational homology equivalence $B\J_g \to \gD_g$.
Moreover, we prove Theorem \ref{theorem:the-map} by constructing the rational splitting map 
$\mu:B\Cobn \to \SP^\infty(\gS^2\gD_2 \vee \dots)$.

\subsection*{Acknowledgements}
This paper was written as a part of my PhD thesis under the supervision of Ulrike Tillmann,
who I would like to thank for many useful conversations.
I am grateful to Oscar Randal-Williams for giving me very detailed feedback on my thesis 
and I thank Dhruv Ranganathan for his comments.
I would like to thank the referee for two very detailed reports and for their help in improving this paper.

My thesis was supported by St.~John’s College, Oxford through the \emph{Ioan and Rosemary James Scholarship} 
and the EPSRC grant no.~1941474.

\setcounter{tocdepth}{1}
\tableofcontents

\section{Labelled cospan categories}\label{sec:labelled-cospans}

In this section we introduce the framework of labelled cospan categories,
give several basic constructions that will be needed later,
and make a conjecture about how they relate to properads.

\subsection{Labelled cospans}

\subsubsection{Definition}
\begin{defn}
    The cospan category $\Csp$ has as objects finite sets $A$ 
    and as morphisms $A \to B$ isomorphism classes $[A \to X \leftarrow B]$
    of cospans of finite sets. Here two cospans $f:A \to X \leftarrow B:g$
    and $f':A \to X' \leftarrow B:g'$ are called isomorphic if there 
    is a bijection $\gp:X \cong X'$ such that $f' = \gp \circ f$ and $g' = \gp \circ g'$.
    The composite of two morphisms $[A \to X \leftarrow B]$
    and $[B \to Y \leftarrow C]$ is defined by taking the pushout
    $[A \to X \cup_B Y \leftarrow C]$.
    This is a symmetric monoidal category under disjoint union.
\end{defn}

\begin{ex}
    The key to the following definitions is that any cobordism $W:M \to N$
    yields a cospan of finite sets $\pi_0(M) \to \pi_0(W) \leftarrow \pi_0(N)$
    in such a way that composition of cobordisms corresponds to compositon 
    of cospans. In other words, taking connected components defines
    a symmetric monoidal functor $\pi_0:\Cob_d \to \Csp$.
\end{ex}

\begin{defn}
    Consider a symmetric monoidal category $(\mcC, \ot, 1_\mcC)$ 
    together with a symmetric monoidal functor $\pi: \mcC \to \Csp$.
    We say that an object $M \in \mcC$ is \emph{connected} with respect to $\pi$
    if $\pi(M)$ is a set with one element. 
    Similarly, we say that a morphism $W:M \to N \in \mcC$ is \emph{connected}
    with respect to $\pi$ if the set $\pi(W)$ in the cospan 
    $[\pi(M) \to \pi(W) \leftarrow \pi(N)]$ has a single element.
    We say that $W$ is \emph{reduced} if the cospan $[\pi(M)  \to \pi(W) \leftarrow \pi(N)]$
    has the property that $\pi(M) \amalg \pi(N) \to \pi(W)$ is surjective.
    Let $\Hom_\mcC^\con(M, N) \subset \Hom_\mcC(M, N)$
    and $\Hom_\mcC^\red(M, N) \subset \Hom_\mcC(M, N)$
    denote the subsets of connected and reduced morphisms, respectively.
    We further let $\Obj^\con(\mcC)$ and $\Mor^\con(\mcC)$ denote the sets of connected objects and morphisms of $\mcC$, respectively.
\end{defn}

The following definition encodes the idea that in $\Cob_d$ any object and cobordism
canonically decomposes into its set of connected components.
\begin{defn}\label{defn:labelled-cospans}
    A \emph{labelled cospan category} is a symmetric monoidal category
    $(\mcC, \ot, 1_\mcC)$ together with a symmetric monoidal functor 
    $\pi: \mcC \to \Csp$ satisfying the following axioms:
    \begin{itemize}
        \item[(i)]
        For any object $M \in \mcC$ such that $\pi(M)$ has $n$ elements
        we can find connected objects $M_1, \dots, M_n \in \mcC$ such that
        $M \cong M_1 \ot \dots \ot M_n$. If $n=0$, we require $M \cong 1_\mcC$.
        \item[(ii)]
        The abelian monoid $\Hom_\mcC(1_\mcC, 1_\mcC)$
        is freely generated by the subset of
        connected morphisms.
        \item[(iii)]
        For any two objects $M, N \in \mcC$ the following map is a bijection:
        \[
            \ot: \Hom_\mcC^\red(M, N) \times \Hom_\mcC(1_\mcC, 1_\mcC) 
            \to \Hom_\mcC(M, N).
        \]
        \item[(iv)]
        For any four objects $M_1,M_2, N_1,N_2 \in \mcC$ 
        the following diagram is a pullback square:
    \end{itemize}
        \[
            \begin{tikzcd}[column sep = 1pc]
                \Hom_\mcC^\red(M_1, N_1) \times \Hom_\mcC^\red(M_2, N_2) 
                \ar[r, "\ot"] \ar[d, "\pi"] & 
                \Hom_\mcC^\red(M_1 \ot M_2, N_1 \ot N_2) \ar[d, "\pi"] \\
                \Hom_\Csp^\red(\pi(M_1), \pi(N_1)) \times \Hom_\Csp^\red(\pi(M_2), \pi(N_2)) 
                \ar[r, "\amalg"] &
                \Hom_\Csp^\red(\pi(M_1) \amalg \pi(M_2), \pi(N_1) \amalg \pi(N_2)) .
            \end{tikzcd}
        \]
    A functor of labelled cospan categories 
    $F:(\pi_\mcC:\mcC \to \Csp) \to (\pi_\mcD:\mcD \to \Csp)$ 
    is a symmetric monoidal functor $F:\mcC \to \mcD$ together with a 
    symmetric monoidal natural isomorphism $\ga: \pi_\mcD \circ F \cong \pi_\mcC$.
\end{defn}

\begin{rem}
    We will show in lemma \ref{lem:hCsp(C)=C} that any labelled cospan category
    $(\mcC \to \Csp)$ is equivalent to another category $h\Csp(\mcC)$
    where morphisms are indeed cospans of finite sets labelled 
    by connected morphisms in $\mcC$. This justifies the name,
    but the list of axioms might still seem a little arbitrary.
    In definition \ref{defn:labelled-infinity-cospan} we provide a much simpler
    definition in the world of $\infty$-categories, which generalizes
    the above and only imposes a single axiom. 
    In fact one can show that labelled cospan categories are equivalent to properads, 
    see theorem \ref{conj:labelled-Csp=properads} \cite{BH22, BS22}.
    For now we shall motivate the definition by giving examples.
\end{rem}

\begin{ex}
    For any dimension $d$ and tangential structure $\gt$ the truncated
    $\gt$-structured $d$-dimensional cobordism category $\Cob_{d,\gt} = h(\mcC_{d,\gt})$ 
    is a labelled cospan category when equipped with the symmetric monoidal
    functor $\pi_0:\Cob_{d,\gt} \to \Csp$ that sends a closed $(d-1)$-manifold $M$
    to the finite set $\pi_0(M)$ and a $d$-dimensional cobordism $W:M \to N$ 
    to the cospan $[\pi_0(M) \to \pi_0(W) \leftarrow \pi_0(N)]$.
\end{ex}

\begin{ex}\label{ex:sub-cospan-cat}
    Let $(\mcC \to \Csp)$ be some labelled cospan category and let $\mcD \subset \mcC$
    be a symmetric monoidal subcategory such that:
    \begin{itemize}
        \item[(a)] If $M \ot N \in \mcD$ for some $M,N \in \mcC$, then $M,N \in \mcD$.
        \item[(b)] If $f \ot g \in \mcD$ for some $(f:M \to N), (g:M' \to N') \in \mcC$,
        then $f,g \in \mcD$.
    \end{itemize}
    Then $(\mcD \to \Csp)$ is a labelled cospan category.
    To see this we check the four properties:
    \begin{itemize}
        \item[(i)] This follows because we had such a decomposition in $\mcC$ and 
        the connected objects appearing in the decomposition of an object of $\mcD$
        also need to lie in $\mcD$ by (a).
        \item[(ii)] This follows because if a submonoid $\Hom_\mcD(1_\mcC, 1_\mcC)$
        of a free monoid $\Hom_\mcC(1_\mcC, 1_\mcC)$ satisfies the cancellation
        property (b), then it is necessarily freely generated on a subset
        of the generators of the original monoid.
        \item[(iii)] The map is injective because it is the restriction
        of an injection and it is surjective by (b).
        \item[(iv)] The map from the top-left corner to the pullback is
        injective because it is the restriction of an injection 
        and it is surjective by (b).
    \end{itemize}
    Combining this with the previous example we see that $(\pi_0:\Cobn \to \Csp)$
    is a labelled cospan category.
\end{ex}

\subsubsection{Decomposing labelled cospans}

The following lemma explains how we can think of any morphism $W:M \to N$
as the cospan $\pi(M) \to \pi(W) \leftarrow \pi(N)$ together with a labelling
of the components of $\pi(W)$ by connected morphisms in $\mcC$.
\begin{lem}\label{lem:decomposing-morphisms}
    Consider two objects $M, N \in \mcC$ and decompose them into connected objects
    $M = M_1 \ot \dots \ot M_m$ and $N = N_1 \ot \dots \ot N_n$.
    The the set of morphisms $M \to N$ can be described as:
    \[
        \Hom_\mcC(M, N) \cong 
        \Hom(1_\mcC, 1_\mcC) \times 
        \coprod_{[X]:\pi(M) \to \pi(N)}\  \prod_{x \in X} 
        \Hom_\mcC^\con\left(\bigotimes_{i \in f^{-1}(x)} M_i, 
        \bigotimes_{j \in g^{-1}(x)} N_j\right)
    \]
    where the coproduct runs over all isomorphism classes of reduced cospans 
    $f:\pi(M) \to X \leftarrow \pi(N):g$.
    This decomposition is natural with respect to functors of labelled cospan categories.
\end{lem}
\begin{proof}
    We have the decompositon 
    $\Hom_\mcC(M,N) \cong \Hom_\mcC(1_\mcC, 1_\mcC) \times \Hom_\mcC^\red(M, N)$
    by condition (iii). The set of reduced morphisms maps to the set of 
    reduced cospans, so we get a disjoint decomposition 
    \[
        \Hom_\mcC^\red(M, N) \cong \coprod_{[X]:\pi(M) \to \pi(N)} 
        \Hom_\mcC^{[X]}(M, N) 
    \]
    where $\Hom_\mcC^{[X]}(M, N) = \{W:M \to N\;|\; \pi(W) = [X]\}$ is the set
    of morphisms with underlying cospan $X$. To understand this set of morphisms
    we decompose $M$ and $N$ according to the cospan $X$:
    \[
        M \cong \bigotimes_{x \in X} M_x, \quad M_x := \bigotimes_{i \in f^{-1}(x)} M_i
        \qand
        N \cong \bigotimes_{x \in X} N_x, \quad N_x := \bigotimes_{j \in g^{-1}(x)} N_j
    \]
    Here we identified $\pi(M)$ with $\{1,\dots,m\}$ according to the 
    decomposition $M= M_1 \ot \dots \ot M_m$ and similarly for $N$.
    We can now apply condition (iv) multiple times and glue the pullback squares
    to obtain the following pullback square:
    \[
        \begin{tikzcd}
            \prod_{x \in X}\Hom_\mcC^\red(M_x, N_x) 
            \ar[r, "\ot"] \ar[d, "\pi"] & 
            \Hom_\mcC^\red(\ot_{x \in X}M_x, \ot_{x \in X}N_x) \ar[d, "\pi"] \\
            \prod_{x \in X}\Hom_\Csp^\red(\pi(M_x), \pi(N_x)) 
            \ar[r, "\amalg"] &
            \Hom_\Csp^\red(\amalg_{x \in X}\pi(M_x), \amalg_{x \in X}\pi(N_x)) .
        \end{tikzcd}
    \]
    By definition of $M_x$ and $N_x$ the cospan $[\pi(M) \to X \leftarrow \pi(N)]$
    corresponds to the 
    element of the product $\prod_{x \in X}\Hom_\Csp^\red(\pi(M_x), \pi(N_x)) $
    that is given by the unique connected cospan 
    $[\pi(M_x) \to \{x\} \leftarrow \pi(N_x)]$ 
    in each entry. 
    We can now compare the fibers of the vertical maps at the points corresponding to $[X]$. 
    Since the diagram is a pullback square the fibers are isomorphic:
    \begin{align*}
        &\prod_{x \in X} \{W_x \in \Hom_\mcC^\red(M_x, N_x)\;|\;\ \pi(W_x) = \{x\}\}\\
        &\cong \{ W \in \Hom_\mcC^\red(\ot_{x \in X} M_x, \ot_{x \in X} N_x) \;|\; \pi(W) = [X]\}.
    \end{align*}
    The left-hand side is the product over $\Hom_\mcC^\con(M_x, N_x)$ 
    and the right-hand side is $\Hom_\mcC^{[X]}(M, N)$.
    Together with the first part this concludes the proof.
\end{proof}

\begin{cor}\label{cor:equiv-of-labelled-cospan-cats}
    A functor $F:(\pi_\mcC:\mcC \to \Csp) \to (\pi_\mcD:\mcD \to \Csp)$ 
    of labelled cospan categories is an equivalence of categories 
    if and only if it satisfies the following two conditions:
    \begin{itemize}
        \item For any connected object $M \in \mcD$ there is a connected object 
        $M' \in \mcC$ such that $F(M') \cong M$.
        \item For any two objects $M, N \in \mcC$ the functor $F$ induces a bijection
        \[
            F: \Hom_\mcC^\con(M, N) \cong \Hom_\mcD^\con(F(M), F(N)).
        \]
    \end{itemize}
\end{cor}
\begin{proof}
    The first condition implies that $F:\mcC \to \mcD$ is essentially surjective
    because by condition (i) every object in $\mcD$ is the product of connected objects.
    The second condition applied to $M=1_\mcC=N$ implies, together with condition (ii),
    that $\Hom_\mcC(1_\mcC, 1_\mcC) \to \Hom_\mcD(1_\mcD, 1_\mcD)$ is a bijection.
    Then we can use the decomposition from lemma \ref{lem:decomposing-morphisms}
    to see that being a bijection on connected morphisms and endomorphisms of $1_\mcC$
    implies that $F$ is a bijection on all homs.
\end{proof}

For any labelled cospan category we can define a reduced category in which 
all closed components are deleted.
\begin{defn}\label{defn:reduced-category}
    The reduced category $\mcC^\red$ has the same objects as $\mcC$ 
    and the morphisms are equivalence classes of morphisms in $\mcC$
    under the action of the abelian monoid $\Hom_\mcC(1_\mcC, 1_\mcC)$.
    This indeed defines a category as for any $W:1_\mcC \to 1_\mcC$
    the action $(W \otimes \blank):\Hom_\mcC(M, N) \to \Hom_\mcC(M, N)$.
    commutes with pre- and post-composition with any morphism.
\end{defn}

Note that it follows from the axioms for labelled cospan categories
that every morphism $[W]:M \to N$ in $\mcC^\red$ has a unique
representative $W \in \Hom_\mcC(M \to N)$ that is reduced.
In other words, the composite
\[
    \Hom_\mcC^\red(M, N) \hookrightarrow \Hom_\mcC(M, N) 
    \to \Hom_{\mcC^\red}(M, N)
\]
is a bijection.
We could therefore equivalently define the category $\mcC^\red$
to have the same objects as $\mcC$ and to have as morphisms the reduced morphisms in $\mcC$.
Composition would then be defined by $W \circ_{\mcC^\red} V = U$
where $U$ is the unique reduced morphism such that there is $Q:1_\mcC \to 1_\mcC$
with $W \circ_{\mcC} V = U \otimes Q$.

\begin{ex}
    There is a canonical bijection between 
    the set of equivalence relations on $A \amalg B$ and $\Hom_\Csp^\red(A, B)$ 
    defined by sending an equivalence relation $R \subset (A \amalg B)^2$
    to the cospan $[A \to (A\amalg B)/\!R \leftarrow B]$.
    This can be used to give an alternative description of $\Csp^\red$
    where morphisms $A \to B$ are equivalence relations on $A \amalg B$.
    In this category morphisms $R \subset (A \amalg B)^2$ and $S \subset (B \amalg C)^2$
    are composed by restricting the equivalence relation on $A \amalg B \amalg C$
    that is generated by $R$ and $S$ to the subset $A \amalg C$.
\end{ex}

\subsubsection{Weighted cospans}

Another class of examples can be constructed by labelling cospans with elements
of an abelian monoid.
This includes $\Csp$, $\Cob_2$, and $\Cobn$, and will provide a useful point of view
in section \ref{sec:mcF-and-mcJ}.
In principle it should be possible to label cospans with operations 
of any modular operad, or even a coloured properad 
(see conjecture \ref{conj:labelled-Csp=properads}).
For the purpose of this paper it will suffice to consider the case
where the operations in the modular operad are elements of an abelian monoid $A$
and pairing two inputs is given by addition with a fixed element $\ga$.
We also allow to specify a subset $A_1 \subset A$, which can be thought of as a stability condition in the sense of stable graphs, see definition \ref{defn:stable-graph}.

\begin{defn}\label{defn:weighting-monoid}
    A \emph{weighting monoid} is a triple $(A, A_1, \ga)$ where $A$ is an abelian monoid, which we will denote additively,
    $A_1 \subset A$ is a subset satisfying $A_1 + A \subset A_1$, 
    and $\ga \in A_1$ is some element.
\end{defn}
The most important weighting monoid is $(\IN, \IN_{\ge 1}, 1)$, 
and we will usually think about this case.

\begin{defn}\label{defn:weighted-cospans}
    For a monoid $(A, +)$ with a preferred element $\ga \in A$ 
    we define the $(A, \ga)$-weighted cospan category $\Csp(A, \ga)$ as follows. 
    Objects are finite sets. 
    Morphisms are equivalence classes of labelled cospans 
    $[i:M \to X \leftarrow N:j, l:X \to A]$.
    Two labelled cospans are equivalent if there is a bijection $\gp:X \cong X'$
    satisfying $l' \circ \gp = l$, $\gp \circ i = i'$, and $\gp \circ j = j'$.
    The composition of two morphisms $(i:M \to X \leftarrow N:j,l)$ and 
    $(i':N \to Y \leftarrow L:j',l')$ is defined by composing the cospans as usual
    and equipping the pushout with the labelling:
    \[
        (l * l')(u \in X \amalg_N Y) 
        =
        {b_1(u; X \leftarrow N \to Y)} \cdot \ga
        + \sum_{x \in (X \to X \amalg_N Y)^{-1}(u)} l(x) 
        + \sum_{y \in (Y \to X \amalg_N Y)^{-1}(u)} l'(y).
    \]
    To define $b_1(u; X \leftarrow N \to Y) \in \IN$ we build a bipartite graph $G$ with vertices $X \amalg Y$ and edges $N$ where each $n \in N$ is an edge from $j(n) \in X$ to $i'(n) \in Y$.
    The set of path components $\pi_0(G)$ of this graph is the pushout $X \amalg_N Y$ and $b_1(u; X, Y)$ is defined to be the first Betti number of the path component of $G$ that corresponds to $u$.
    In formulas this gives
    \[
        b_1(u; X \leftarrow N \to Y) := |(N \to X \amalg_N Y)^{-1}(u)|
        - |(X \amalg Y \to X \amalg_N Y)^{-1}(u)| + 1.
    \]
\end{defn}

\begin{lem}\label{lem:weighted-Csp}
    Definition \ref{defn:weighted-cospans} yields a well-defined category $\Csp(A, \ga)$
    and disjoint union defines a symmetric monoidal structure on this category.
    Moreover, the forgetful functor $\Csp(A, \ga) \to \Csp$ is symmetric monoidal
    and makes $\Csp(A, \ga)$ into a labelled cospan category.
\end{lem}
\begin{proof}
    We begin by checking that composition is associative.
    Consider three composable morphisms $(M \to X \leftarrow N,l)$,
    $(N \to Y \leftarrow L,l')$, and $(L \to Z \leftarrow O, l'')$.
    Both ways of composing these morphisms result in the cospan $(M \to T \leftarrow O)$ where $T = X \amalg_N Y \amalg_L Z$ and the labelling sends $u \in T$ to
    \begin{align*}
        &\left(|(N \amalg L \to T)^{-1}(u)| 
        - |(X \amalg Y \amalg Z \to T)^{-1}(u)| + 1 \right) \cdot \alpha \\
        &+\sum_{x \in (X \to T)^{-1}(u)} l(x)
        +\sum_{y \in (Y \to T)^{-1}(u)} l'(y)
        +\sum_{z \in (Z \to T)^{-1}(u)} l''(z)
    \end{align*}
    
    The composition of $(A,\ga)$-labelled cospans is defined component-wise
    and hence disjoint union $\amalg: \Csp(A,\ga) \times \Csp(A, \ga) \to \Csp(A, \ga)$
    is a functor. 
    This defines a symmetric monoidal structure with unit the empty set $\emptyset$ and the same structure isomorphisms as for $\Csp$ (now labelled by $0$).
    By construction the forgetful functor $\Csp(A, \ga) \to \Csp$ is symmetric monoidal.
    We leave the check that this satisfies the conditions of definition \ref{defn:labelled-cospans} to the reader.
\end{proof}

\begin{defn}\label{defn:restricted-weighted-cospans}
    For weighting monoid $(A, A_1, \alpha)$ as in definition \ref{defn:weighting-monoid} we define
    $\Csp(A, A_1, \ga) \subset \Csp(A, \ga)$
    as the subcategory that contains all objects and those morphisms
    $(M \to X \leftarrow N, l)$ where for each $x \in X$
    with $|(M \amalg N \to X)^{-1}(x)| \le 1$ we require $l(x) \in A_1$.
\end{defn}

We briefly check that  $\Csp(A, A_1, \ga) \subset \Csp(A, \ga)$ is indeed closed under composition.
If $(M \to X \leftarrow N, l)$ and $(N \to Y \leftarrow L, l')$ are composable morphisms and $u \in X \amalg_N Y$ is such that $(M \amalg L \to X \amalg_N Y)^{-1}(u)$ has at most one element, then we need to check that $(l * l')(u) \in A_1$.
Consider the bipartite graph $G$ from definition \ref{defn:weighted-cospans} and let $G_u \subset G$ be the component corresponding to $u$. 
If $G_u$ has positive first Betti number, then $(l * l')(u) = b_1(u)\cdot \alpha + \dots$ has $\alpha$ as a summand and thus lies in $A_1$.
Otherwise, we can always find a vertex $v \in X \amalg Y$ of $G_u$ such that
$((M \amalg N) \amalg (N \amalg L) \to X \amalg Y)^{-1}(v)$ has at most one element. 
(If $G_u$ is a single vertex, take that vertex. If it is a tree with at least one edge, at most one leaf is hit by $(M \amalg L \to X \amalg_N Y)^{-1}(u)$ -- take a leaf that is not hit.)
If $v$ is in $X$ this means that $(M \amalg N \to X)^{-1}(v)$ has at most one element so $l(v) \in A_1$, and if $v$ is in $Y$ this means that $(N \amalg L \to Y)^{-1}(v)$ has at most one element so $l'(v) \in A_1$.
In either case $(l*l')(u)$ must be in $A_1$ because one of its summands is in $A_1$.

The subcategory is also closed under disjoint union and further satisfies the condition
of example \ref{ex:sub-cospan-cat}. Hence $\Csp(A, A_1, \ga)$ is a labelled cospan category.

\begin{lem}\label{lem:Cob2=labelled-csp}
    There are equivalences of labelled cospan categories:
    \begin{align*}
        \Csp(0, 0) &\simeq \Csp, &
        \Csp(\IN, 1) &\simeq \Cob_2, &
        \Csp(\IN, \IN_{\ge1}, 1) &\simeq \Cobn.
    \end{align*}
\end{lem}
\begin{proof}
    The first equivalence is clear since a cospan labelled in the trivial monoid $0$    
    is the same data as just a cospan, so the forgetful map $\Csp(0, 0) \to \Csp$
    is an isomorphism of categories.
    
    For the second equivalence we define a functor $F:\Cob_2 \to \Csp(\IN, 1)$
    as the lift of the functor $\pi_0:\Cob_2 \to \Csp$ by labelling the
    cospan $[\pi_0(M) \to \pi_0(W) \leftarrow \pi_0(N)]$ obtained from a morphism $W:M \to N$
    with the function $g: \pi_0(W) \to \IN$ that sends a component to its genus.
    We need to check that $F$ is in fact functorial. This means that we 
    have to verify for any two cobordisms $W:M \to N$ and $V:N \to L$
    that the genus of some component $U \subset W \cup_N V$ is given by:
     \[
        \big(
        |\pi_0(U \cap N)| + 1 - |\pi_0(U \cap W) \amalg \pi_0(U \cap V)|
        \big)
        +
        \sum_{U_0 \in \pi_0(U \cap W)} g(U_0) + \sum_{U_1 \in \pi_0(U \cap V)} g(U_1)
        .
    \] 
    Without loss of generality we may assume that $W \cup_N V$ is connected
    and hence that $U = W \cup_N V$. The Euler characteristic is additive 
    under the composition of surface-bordisms and so we can compute:
    \begin{align*}
        &\chi(W \cup_N V) 
        = \chi(W) + \chi(V) \\
        &= \sum_{U_0 \subset W} (2 - 2g(U_0) - |\pi_0(\partial U_0)|)
        + \sum_{U_1 \subset V} (2 - 2g(U_1) - |\pi_0(\partial U_1)|) \\
        &= 
        2|\pi_0(W) \amalg \pi_0(V)| - 2|\pi_0(N)| - |\pi_0(M)| - |\pi_0(L)|
        -2\left(\sum_{U_0 \subset W} g(U_0) + \sum_{U_1 \subset V} g(U_1)\right)
    \end{align*}
    Combining this with $\chi(W \cup_N V) = 2-2g(W \cup_N V)-|\pi_0(M)|-|\pi_0(L)|$
    yields the desired formula.
    
    The functor $F$ is by construction symmetric monoidal and compatible 
    with the projection to $\Csp$. 
    We may hence use corollary \ref{cor:equiv-of-labelled-cospan-cats}
    to check that it is an equivalence of labelled cospan categories.
    Indeed, $F$ is surjective on connected objects (it hits the one-point set)
    and $F$ is fully faithful on connected morphisms since a connected surface
    with fixed boundary is uniquely determined by its genus.
    
    The third equivalence is obtained by restricting the previous equivalence.
    The subcategory $\Cobn \subset \Cob_2$ corresponds exactly to $\Csp(\IN,\IN_{\ge1},1)$
    since every connected cobordism $W:M \to N$ satisfies:
    \[
        \chi(W) \le 0 
        \Leftrightarrow
        \big(|\pi_0(M) \amalg \pi_0(N)| \le 1 \Rightarrow g(W) \in \IN_{\ge 1}\big).\qedhere
    \]
\end{proof}

\begin{ex}\label{ex:Cob2g<gc}
    For any $\gc \in \IN$ one can define a category $\Cob_2^{g < \gc}$
    as the quotient of $\Cob_2$ by the equivalence relation $\sim_\gc$
    on morphisms that is generated as follows:
    whenever $W:M \to N$ is a cobordism and $V \subset W$ is a component of
    genus at least $\gamma$, then we set $W \sim_\gc W \# (S^1 \times S^1)$
    where the connected sum is taken at $V$.
    Under this equivalence relation two morphisms $W, W': M \to N$
    are identified if they become diffeomorphic after arbitrarily increasing
    the genus of every component that already has genus at least $g$.
    
    This category is equivalent to the weighted cospan category 
    $\Csp(\IN/\gc,1)$ where $\IN/\gc$ is the abelian monoid
    defined by identifying all elements of $\IN$ that lie in $\IN_{\ge\gc}$.
    For $\gc = 0$ we have that $\Cob_2^{g<0} \simeq \Csp$.
    One can also define a subcategory $\Cob_2^{\chi\le0, g<\gc} \subset \Cob_2^{g < \gc}$.
    (This will be the entire category if $\gc = 0$.)
    This category is equivalent to $\Csp(\IN/\gc,\IN_{\ge1}/\gc, 1)$.
\end{ex}

\begin{ex}\label{ex:Cob_d^Sd-1}
    For any $d \ge 2$ consider the full subcategory $\Cob_d^{S^{d-1}} \subset \Cob_d$ 
    of the truncated $d$-dimensional oriented cobordism category on all the objects that are diffeomorphic to $\coprod_{i=1}^k S^{d-1}$ for some $k \ge 0$.  
    Let $\mc{M}_d$ be the commutative monoid of diffeomorphism classes of connected closed oriented $d$-manifolds with addition given by connected sum.
    For $d=2$ we have $\mc{M}_2 \cong \IN$ and for $d=3$ Milnor \cite{Mil62} showed that the commutative monoid $\mc{M}_3$ is freely generated by the set of diffeomorphism classes of prime $3$-manifolds.

    We construct an equivalence of symmetric monoidal categories
    \[
        F\colon \Csp(\mc{M}_d, [S^1\times S^{d-1}])
        \to \Cob_d^{S^{d-1}} 
    \]
    that sends a finite set $A$ to closed oriented $(d-1)$-manifold $A \times S^{d-1}$.
    Given a labelled cospan $(A \to X \leftarrow B, l)$ we let $W := \coprod_{x \in X} W_{l(x)}$ where $W_{l(x)}$ is a closed oriented $d$-manifold that has the diffeomorphism type $l(x) \in \mc{M}_d$.
    Then we pick an orientation-preserving embedding $i: A \times D^d \amalg B \times (D^d)^- \hookrightarrow W$ such that the induced map on $\pi_0$ is $A \amalg B \to X$, and we let
    \[
        F(A \to X \leftarrow B, l) := W \setminus i(A \times (D^d)^\circ \amalg B \times (D^d)^\circ),
    \]
    which is a bordism from $A \times S^{d-1}$ to $B \times S^{d-1}$.

    First, we check that the definition is independent of choices.
    If we pick any other embedding $i'$ then, because it is orientation preserving and induces the same map on $\pi_0$, it must be in the same component of the embedding space.
    By \cite[Chapter 8, Theorem 3.2]{Hir94} there is a diffeomorphism $\varphi \in \Diff(W)$ with $\varphi \circ i = i'$ and this diffeomorphism will restrict to a boundary-preserving diffeomorphism between $W \setminus i(\dots)$ and $W \setminus i'(\dots)$.
    Similarly, if we make different choices of the manifolds $W_{l(x')}$ in the diffeomorphism type $l(x')$, then the resulting $W$s are diffeomorphic and thus so are the bordisms.
    Second, we check that it is functorial.
    Suppose we have a second morphism $(B \to Y \leftarrow C, l')$ and we set $V = \coprod_{y \in Y} V_{l'(y)}$ and pick an embedding $j: B \times D^d \amalg C \times (D^d)^- \hookrightarrow V$.
    Then the bordism obtained by composing their images in the bordism category is
    \[
        (W \setminus i(A \times (D^d)^\circ \amalg B \times (D^d)^\circ)) \cup_{B \times S^{d-1}} 
        (V \setminus j(B \times (D^d)^\circ \amalg C \times (D^d)^\circ))
    \] 
    and we need to show that this is the image of the composite of weighted cospans.
    Alternatively, this bordism can be obtained by taking $W \amalg V$, attaching a $1$-handle for each $b \in B$ (according to the maps $B \to X = \pi_0(W)$ and $B \to Y = \pi_0(V)$), and then removing $i(A \times (D^d)^\circ)$ and $j(C \times (D^d)^\circ)$.
    In terms of the bipartite graph $G$ from definition \ref{defn:weighting-monoid}, we are performing one $1$-handle attachment for each edge in $G$.
    If we first perform the $1$-handle attachments for a spanning forest in $G$, then this exactly performs connected sums in $\coprod_{x \in X} W_{l(x)} \amalg \coprod_{y \in Y} V_{l(y)}$ that correspond to the summation $\sum_{x \in (X \to X \amalg_B Y)^{-1}(u)} l(x) + \sum_{y \in (Y \in X \amalg_B Y)^{-1}(u)} l'(y)$ in $l*l'$.
    The remaining $1$-handles have both their attachment points in the same path component, so attaching them is equivalent to taking a connected sum with $(S^1 \times S^{d-1})$, thus accounting for the summand $b_1(u) \cdot [S^1 \times S^{d-1}]$ in $l * l'$.
    Finally, the functor is symmetric monoidal because the monoidal product is given by disjoint union on both sides.

    To see that $F$ is an equivalence, we can apply corollary \ref{cor:equiv-of-labelled-cospan-cats}.
    The functor $F$ is essentially surjective by definition of $\Cob_d^{S^{d-1}}$.
    For all finite sets $A$ and $B$ the map
    \[
        \mc{M}_d \cong 
        \Hom_{\Csp(\mc{M}_d, [S^1 \times S^{d-1}])}^\con(A, B) 
        \longrightarrow
        \Hom_{\Cob_d}^\con(A \times S^{d-1}, B \times S^{d-1}) 
    \]
    is a bijection, with inverse defined by 
    \[
        G: (W\colon A \times S^{d-1} \to B \times S^{d-1}) 
        \longmapsto
        (A \times D^d) \cup_{A \times S^{d-1}} W \cup_{B \times S^{d-1}} (B \times (D^d)^-).
    \]
    Indeed, if we apply $F$ to $G(W)$, we are free to choose the embedding of $A \times D^d$ and $B \times (D^d)^-$ to be the canonical one so that removing them recovers $W$.
    For the other composite we start with some $V \in \mc{M}_d$, remove a number of disks (keeping track of the boundary identification with $(A\amalg B) \times S^d$) and glue them back in, which results in a manifold $G(F(V))$ that is diffeomorphic to $V$.
    This verifies the conditions of corollary \ref{cor:equiv-of-labelled-cospan-cats} and therefore $F$ is an equivalence.
\end{ex}

Note that the above example fails in the non-oriented case because the embedding of disks is no longer unique up to isotopy.
In fact, the symmetric monoidal categories $\Cob_2^{\rm unor}$ and $\Csp(\mc{M}^{\rm unor}, [S^1 \times S^1])$ cannot be equivalent:
the object $S^1 \in \Cob_2^{\rm unor}$ has a non-trivial automorphism given by the cylinder where we flip one boundary identification,
while the object $* \in \Csp(\mc{M}_2^{\rm unor}, [S^1 \times S^1])$ has no non-trivial automorphisms because the only invertible element in $\mc{M}_2^{\rm unor}$ is the unit $[S^2]$.

\subsection{The enhanced cospan category}
Given a labelled cospan category $(\mcC \to \Csp)$ we now want to build 
a refinement $\ICsp(\mcC)$ that keeps track of permutations of closed components.
We begin by defining the groupoid-enriched category $\ICsp$:
the $2$-category of finite sets, cospans of finite sets, and isomorphisms of cospans.

In the process of constructing $\ICsp$ we pick some infinite ``background set'' $\gO$, so that objects of $\ICsp$ are finite subsets of $\gO$ rather than abstract sets. This allows us to ask when objects are disjoint and to later define a strictly associative union operation.
We also once and for all pick a total order on $\gO$.
Such strictifications are necessary to directly mirror the surgery arguments of \cite{GRW14, GRW17} later on.
During the writing of this article, a new approach to such surgery arguments was proposed by Hebestreit and Steimle \cite{HebestreitSteimle-surgery}.
It is likely possible to perform our surgery arguments in their set-up, which would allow one to work more invariantly and avoid such strictifications.

By proposition \ref{prop:closed-reduced} the canonical functor $\ICsp(\mcC) \to \mcC$ induces a rational equivalence on classifying spaces,
so this subsection might be skipped at a first read.

\subsubsection{The unlabelled case} \label{subsubsec:ICsp}
We begin by setting up an especially rigid way of representing finite sets, which will make the composition in $\Csp$ strictly associative.
\begin{defn}\label{defn:presented-set}
    A \emph{presented finite set} is a triple $(l, X, R)$ where $l \in \{0,1,\dots\}$,
    $X \subset \{1,\dots, l\}$, and $R \subset X \times X$,
    such that $R$ is an equivalence relation on $X$.
    We let $\UL{X} := X/R$ denote the set of equivalence classes of $X$ under $R$,
    and we will often write $\UL{X}$ when we mean the entire datum $(l, X, R)$.
    
    Two presented finite sets $(l, X, R)$ and $(k, Y, S)$ are called \emph{disjoint}
    if $l = k$ and $X \cap Y = \emptyset$. In this case we write $(l,X,R) \bot (k,Y,S)$
    and we define their disjoint union as $(l,X,R) \amalg (k,Y,S) := (l, X \cup Y, R \cup S)$.
    
    For any two presented finite sets $(l, X, R)$ and $(k, Y, S)$ and maps $i:A \to X/R$ 
    and $j:A \to Y/S$ from some set $A$ we define the \emph{glueing} as 
    \[
        (l, X, R) \cup_A (k, Y, S) := (l+k, X \cup (Y+l), R \cup_A (S+l))
    \]
    where $Y+l := \{ y +l \;|\; y \in Y\}$ and $R \cup_A (S+l)$ is 
    the equivalence relation generated by $R$, $(S+l)$, and $i(a) \sim j(a)$
    for all $a \in A$.
\end{defn}

Note that the disjoint union of disjoint presented finite sets is strictly associative 
and commutative whenever it is defined.
Moreover, the gluing operation 
$(\UL{X} \leftarrow A \to \UL{Y}) \to \UL{X}\cup_A \UL{Y}$
is strictly associative, which allows us to make the following definition:

\begin{defn}
    The $(2,1)$-category $\ICsp$ is defined as follows.
    Objects are finite subsets $A \subset \gO$. 
    Non-identity morphisms $A \to B$ are cospans $(f_X:A \to \UL{X} \leftarrow B:g_X)$
    where $\UL{X}$ is a presented finite set.
    We also have identity morphisms, which we denote as $(A = A = A)$.
    The composition of morphisms is defined as:
    \[
        (f_Y:B \to \UL{Y} \leftarrow C:g_Y) \circ (f_X:A \to \UL{X} \leftarrow B:g_X)
        := (f_X:A \to \UL{X} \cup_B \UL{Y} \leftarrow C:g_Y).
    \]
    $2$-morphisms $\gp:(f_X:A \to \UL{X} \leftarrow B:g_X) 
    \Rightarrow (f_X'\to \UL{X}' \leftarrow B:g_X')$
    are bijections $\gp:\UL{X} \cong \UL{Y}$ such that 
    $\gp \circ f_X = f_X'$ and $\gp \circ g_X = g_X'$. 
    (Note that $\gp$ is not an isomorphism of presented finite sets,
    but just a bijection between the resulting quotient sets.)
    These are composed in the obvious way.
\end{defn}

\begin{rem}
    The addition of identity morphisms in the above definition is a little ad-hoc.
    Let us briefly explain how they interact with the other morphisms.
    The horizontal composition is defined such that they act as identities:
    $(A = A = A) \cup_A (A\to \UL{X} \leftarrow B) := (A \to \UL{X} \leftarrow B)$,
    and similarly when composing on the other side.
    The $2$-morphisms are defined by treating the middle $A$ as if it were 
    a presented finite set $\UL{X}$: 
    a $2$-morphism $\gp:(A = A = A) \Rightarrow (f:A \to \UL{X} \leftarrow A:g)$
    is a bijection $\gp:A \to \UL{X}$ such that $\gp \circ \id_A = f_X$ 
    and $\gp \circ \id_A = g_X$.
    This means that such a $2$-morphism exists if and only if $f = g$ 
    is a bijection and the $2$-morphisms is unique if it exists.
\end{rem}

To describe the operation of disjoint union on $\ICsp$ we will use 
the notion of partial commutative monoidal structures that we introduce in 
definition \ref{defn:PCMC} and \ref{defn:PCM-2-cat}.
See section \ref{subsec:PCMC} for an introduction to how we deal
with symmetric monoidal structures.
\begin{defn}\label{defn:Csp-PCM}
    The partial commutative monoidal structure on the $(2,1)$-category $\ICsp$ 
    is defined by the relation $\bot$ and the operation $\amalg$.
    On objects we set $A \bot B$ whenever $A \cap B = \emptyset$, i.e.~we set $\id_A \bot \id_B$ in this case.
    On non-identity morphisms  morphisms we set $(\UL{X}, i, i') \bot (\UL{Y}, j, j')$ whenever 
    $\UL{X} \bot \UL{Y}$ in the sense of \ref{defn:presented-set}.
    Moreover, we never have $f \bot g$ when one of the two is an identity morphism and the other is not.
    (This is closed under composition as it is not possible to obtain an identity morphism as a composite of non-identity morphisms.)
    The operation $\amalg$ is given by taking disjoint unions, both on identity and on non-identity morphisms.
\end{defn}

\begin{lem}
    The canonical functor $\ICsp \to \Csp$ that sends a cospan 
    $(A \to \UL{X} \leftarrow B)$ to its isomorphism class 
    induces an equivalence of categories $h(\ICsp) \simeq \Csp$.
\end{lem}
\begin{proof}
    The $1$-category $h(\ICsp)$ is defined by identifying all $1$-morphisms 
    in $\ICsp$ that are connected by a $2$-morphism.
    The functor $h(\ICsp) \to \Csp$ is hence faithful by definition:
    on both sides isomorphic cospans are identified.
    It is also full because every cospan $A \to Z \leftarrow B$ is isomorphic
    to a cospan $A \to \UL{X} \leftarrow B$ seeing as any finite set $Z$
    is in bijection with some presented finite set $\UL{X}$.
    The functor is also essentially surjective as any finite set $A$
    can be embedded $\gi:A \hookrightarrow \gO$ and the cospan
    $(A \cong \gi(A) = \gi(A))$ is an isomorphism in $\Csp$.
\end{proof}

\subsubsection{The general case}\label{subsubsec:the-general-case}
We now define a version of $\ICsp$ where the finite sets and cospans are 
labelled by connected objects and morphisms in $\mcC$. 
In lemma \ref{lem:Csp-well-defined} we check that this is indeed well-defined
and that the $2$-functors described below are indeed functorial.
To make sense of the labellings we use the following convention.

\begin{defn}\label{defn:tensor-over-set}
    Given a symmetric monoidal category $\mcC$, a finite subset $A \subset \Omega$, and a map $c:A \to \mcC$, we set 
    \[
        \bigotimes_{a \in A} c(A) :=  (\cdots(c(a_1) \otimes c(a_2)) \otimes c(a_3)) \otimes \dots) \otimes c(a_n)
    \]
    where $\{a_1 < \dots < a_n\} = A$ is the enumeration of $A$ according to the total order we picked on $\gO$.
    (Similarly, we define tensor products indexed by the presented finite sets from definition \ref{defn:presented-set}; here we order $\UL{X} = X/R$ by the smallest numbers appearing in each equivalence class.)
\end{defn}

We note that by the coherence theorem for symmetric monoidal categories \cite[Theorem 4.2]{MacLane-coherence} the order we chose on $\gO$ does not really matter here.
In particular, if $A, B \subset \gO$ are finite subsets, $f: A \to B$ a map and $c: A \to \mcC$ a labelling by objects, then we have a canonical isomorphism
\[
    \bigotimes_{b \in B} \bigotimes_{a \in f^{-1}(b)} c(a) \cong \bigotimes_{a \in A} c(a).
\]

\begin{defn}\label{defn:enhanced-labelled-cospans}
    For any labelled cospan category $(\mcC \to \Csp)$ we define the 
    \emph{enhanced cospan category} $\ICsp(\mcC)$ of $\mcC$ 
    to be the following $(2,1)$-category. 
    \begin{itemize}
    \item 
    An objects is a tuple $(A,c_A)$ where $A \subset \gO$ is a finite set 
    and $c_A: A \to \Obj^\con(\mcC)$ is a labelling by connected objects of $\mcC$.
    \item 
    A non-identity $1$-morphism $(A,c_A) \to (B,c_B)$ is a tuple 
    of a non-identity morphisms $(f:A \to \UL{X} \leftarrow B:g)$ in $\ICsp$
    and a labelling $o:\UL{X} \to \Mor^\con(\mcC)$
    such that for each $x \in \UL{X}$ the label $o(x)$ is a connected morphism 
    \[
        \bigotimes_{a \in f^{-1}(x)} c_A(a) \to \bigotimes_{b \in g^{-1}(x)} c_B(b).
    \]
    We also allow for identity morphisms $\id_{(A,c_A)}$, which we think of as 
    a cospan $(A = A = A)$ where the labelling $o:A \to \Mor^\con(\mcC)$ is
    given by $o(a) := \id_{c_A(a)}$.
    \item 
    A $2$-morphism 
    $(f:A \to \UL{X} \leftarrow B:g, o) \Rightarrow (f':A \to \UL{Y} \leftarrow B:g', o')$
    is a bijection $\gp:\UL{X} \cong \UL{Y}$ such that 
    $f' = \gp \circ f$, $g' = \gp \circ g$, and $o' \circ \gp = o$.
    \end{itemize}
    The composition of two morphisms 
    $(f_X:A \to \UL{X} \leftarrow B:g_X, o_X)$ and $(f_Y:B \to \UL{Y} \leftarrow C:g_Y, o_Y)$
    is defined as the glued cospan 
    $(f_Z:A \to \UL{X} \cup_B \UL{Y} \leftarrow C:g_Z, o_X * o_Y)$
    where the induced labelling $o_X * o_Y$ is defined as follows.
    For each $z \in \UL{X}\cup_B\UL{Y}$ we let
    $\UL{X}_z \subset \UL{X}$ and $\UL{Y}_z \subset \UL{Y}$ denote
    the preimages of $z$ under $\UL{X} \amalg \UL{Y} \to \UL{X}\cup_B\UL{Y}$.
    \[
    \begin{tikzcd}[column sep = 3pc]
        \bigotimes\limits_{a \in f_Z^{-1}(z)} c_A(a) \ar[rr, "(o_X * o_Y)(z)"] \ar[d, "\cong"] &&
        \bigotimes\limits_{c \in g_Z^{-1}(z)} c_C(s) \ar[d, "\cong"] \\
        \bigotimes\limits_{x \in X_z} \bigotimes\limits_{a \in f_X^{-1}(x)} c_A(a) 
        \ar[r, "\ot_x o_X(x)"] & 
        \bigotimes\limits_{x \in X_z} \bigotimes\limits_{b \in g_X^{-1}(x)} c_B(b)
        \cong 
        \bigotimes\limits_{y \in Y_z} \bigotimes\limits_{b \in f_Y^{-1}(y)} c_B(b) 
        \ar[r, "\ot_y o_Y(y)"] &
        \bigotimes\limits_{y \in Y_z} \bigotimes\limits_{c \in g_Y^{-1}(y)} c_C(c).
    \end{tikzcd}
    \]
\end{defn}

\begin{defn}
    There is a functor $\ot_\mcC:\ICsp(\mcC) \to \mcC$ defined by sending 
    $(A,c_A)$ to $\bigotimes_{a \in A} c_A(a)$ and 
    sending morphisms $(f_X:A \to \UL{X} \leftarrow B:g_X, o)$ to
    \[
        \bigotimes_{a \in A} c_A(a) 
        \cong 
        \bigotimes_{x \in \UL{X}} \bigotimes_{a \in f_X^{-1}(x)} c_A(a)
        \xrightarrow{\ot_{x} o(x)}
        \bigotimes_{x \in \UL{X}} \bigotimes_{b \in g_X^{-1}(x)} c_B(a)
        \cong 
        \bigotimes_{b \in B} c_B(b) .
    \]
\end{defn}

\begin{defn}
    We define a partial commutative monoidal structure (see definition \ref{defn:PCMC})
    on $\ICsp(\mcC)$ by saying that two objects or morphisms are disjoint
    if and only if they are disjoint as objects or morphisms in $\ICsp$ after forgetting
    labellings. The union of disjoint objects or morphisms is defined by
    taking the union in $\ICsp$ and equipping it with the evident labelling.
\end{defn}

\begin{lem}\label{lem:Csp-well-defined}
    Horizontal composition of $1$-morphisms in $\ICsp(\mcC)$ is associative 
    and hence $\ICsp(\mcC)$ is a well-defined $2$-category.
    Moreover, the following three constructions are $2$-functors:
    \[
        \mathrm{forget}:\ICsp(\mcC) \to \ICsp, \quad
        \amalg:\ICsp(\mcC)^{2\bot} \to \ICsp(\mcC), \qand
        \ot_\mcC: \ICsp(\mcC) \to \mcC.
    \]
\end{lem}
\begin{proof}
    We will show that associativity holds when composing three
    composable labelled cospans: 
    $(A \to \UL{X} \leftarrow B, o_X)$, 
    $(B \to \UL{Y} \leftarrow C, o_Y)$, 
    $(C \to \UL{Z} \leftarrow D, o_Z)$.
    The functoriality of ``$\mathrm{forget}$", $\amalg$, and $\ot_\mcC$ will follow from the proof.
    The underlying cospan is $\UL{X} \cup_B \UL{Y} \cup_C \UL{Z}$
    and it does not matter how we parenthesise this expression
    as composition in $\ICsp$ is strictly associative.
    We need to check that the two induced labellings on this composed cospan agree.
    
    Before we show this, let us consider the case of a single composition
    $(\UL{X}, o_X) \cup_B (\UL{Y}, o_Y)$.
    For $p \in \UL{X} \cup_B \UL{Y}$ we let $\UL{X}_p \subset \UL{X}$, 
    $\UL{Y}_p \subset \UL{Y}$, and $B_p \subset B$ denote the subsets
    that are sent to $p$ under the quotient map.
    The label $(o_X*o_Y)(p)$ of $p$ in the composite cospan is exactly the morphism 
    $\ot_\mcC(\UL{X}_p \cup_{B_p} \UL{Y}_p, o_X*o_Y)$ in $\mcC$. 
    (If $(\UL{W},o_W)$ is a labelled cospan where $\UL{W} = \{p\}$
    is a single point, then $\ot_\mcC(\UL{W}, o_W) = o_W(p)$.)
    By inspecting the definition it follows that this composite label is
    exactly $\ot_\mcC(\UL{Y}_p, o_Y) \circ \ot_\mcC(\UL{X}_p, o_X)$ 
    as illustrated in the following diagram:
    \[
    \begin{tikzcd}[column sep = .5pc]
        \bigotimes\limits_{a \in A} c_A(a) 
        \ar[rrrr, "(o_X * o_Y)(p) = {\ot_\mcC(\UL{X}_p \cup_{B_p} \UL{Y}_p, o_X*o_Y)}"]
        \ar[d, "\cong"]  \ar[rrd, "{\ot_\mcC(\UL{X}_p,o_X)}"] &&&&
        \bigotimes\limits_{c \in C} c_C(s) \ar[d, "\cong"] \\
        \bigotimes\limits_{x \in X_p} \bigotimes\limits_{a \in f_X^{-1}(x)} c_A(a) 
        \ar[rd, "\ot_x o_X(x)"]
        &&
        \bigotimes\limits_{b \in B_p} c_B(b)
        \ar[rd, "\cong"] \ar[rru, "{\ot_\mcC(\UL{Y}_p,o_Y)}"] 
        &&
        \bigotimes\limits_{y \in Y_p} \bigotimes\limits_{c \in g_Y^{-1}(y)} c_C(c)
        \\
         & 
        \bigotimes\limits_{x \in X_p} \bigotimes\limits_{b \in g_X^{-1}(x)} c_B(b)
        \ar[rr, "\cong"] \ar[ru, "\cong"] &&
        \bigotimes\limits_{y \in Y_p} \bigotimes\limits_{b \in f_Y^{-1}(y)} c_B(b) 
        \ar[ru, "\ot_y o_Y(y)"] &
    \end{tikzcd}
    \]
    This already implies that $\ot_\mcC: \ICsp(\mcC) \to \mcC$ 
    will be a $2$-functor once we show that $\ICsp(\mcC)$ is a $2$-category.
    
    Returning to the case of a triple composition we can now see
    that a point $q \in \UL{X} \cup_B \UL{Y} \cup_C \UL{Z}$ is labelled by:
    \begin{align*}
        ((o_X * o_Y) * o_Z)(q) 
        & = \ot_\mcC(\UL{X}_q \cup_{B_q} \UL{Y}_q, o_X * o_Y) \circ \ot_\mcC(\UL{Z}_q, o_Z) \\
        & = \left(\bigotimes_{p \in \UL{X}_q \cup_{B_q} \UL{Y}_q} \ot_\mcC(\UL{X}_p \cup_{B_p} \UL{Y}_p, o_X * o_Y) \right)
        \circ \ot_\mcC(\UL{Z}_q, o_Z) \\
        & = \left(\bigotimes_{p \in \UL{X}_q \cup_{B_q} \UL{Y}_q} \ot_\mcC(\UL{X}_p, o_X) 
        \circ \ot_\mcC(\UL{Y}_p, o_Y) \right)
        \circ \ot_\mcC(\UL{Z}_q, o_Z) \\
        & = (\ot_\mcC(\UL{X}_q, o_X) \circ \ot_\mcC(\UL{Y}_q, o_Y)) \circ \ot_\mcC(\UL{Z}_q, o_Z)
    \end{align*}
    Since composition in $\mcC$ is associative it follows that this
    agrees with $(o_X * (o_Y * o_Z))(q)$ and hence horizontal composition in $\ICsp(\mcC)$
    is indeed associative.
\end{proof}

\begin{lem}\label{lem:hCsp(C)=C}
    The enhanced cospan category fits into a diagram of $2$-categories 
    \[
    \begin{tikzcd}
        \ICsp(\mcC) \ar[r] \ar[d, "\mathrm{fgt}"] &     
        h(\ICsp(\mcC)) \ar[r, "\simeq"] \ar[d, "\mathrm{fgt}"] &     
        \mcC \ar[d, "\pi"] \\
        \ICsp \ar[r] &     
        h(\ICsp) \ar[r, "\simeq"] &     
        \Csp 
    \end{tikzcd}
    \]
    that commutes up to natural isomorphisms.
    Moreover, for any two objects $(A,c_A)$ and $(B,c_B)$ there is 
    a canonical equivalence of groupoids
    \[
        \Hom_{\ICsp(\mcC)}((A,c_A), (B,c_B)) \longrightarrow
        \Hom_{\ICsp(\mcC)}(\emptyset, \emptyset) \times 
        \Hom_\mcC^\red(\ot(A,c_A), \ot(B,c_B))
    \]
    and $\Hom_{\ICsp(\mcC)}(\emptyset, \emptyset)$ is the free symmetric monoidal 
    groupoid on the set $\Hom_\mcC^\con(1_\mcC, 1_\mcC)$.
\end{lem}
\begin{proof}
    The existence of the commutative diagram is clear by construction.
    We need to check that $\ot_\mcC:h(\ICsp(\mcC)) \to \mcC$ is an equivalence 
    of categories. 
    Note that the hom-sets in $h(\ICsp(\mcC))$ have a decomposition as in \ref{lem:decomposing-morphisms} and hence the argument of corollary \ref{cor:equiv-of-labelled-cospan-cats} applies to the functor $\ot_\mcC$.
    (In fact, $\mathrm{fgt}:h(\ICsp(\mcC)) \to h(\ICsp) \simeq \Csp$ is a labelled cospan category with the symmetric monoidal structure induced from the PCM structure via \ref{ex:PCM-vs-sm}.)
    It is clear that $\ot_\mcC$ is surjective on connected objects, and one also easily sees that $\ot_\mcC$ is fully faithful on connected morphisms.
    
    For the second claim we begin by defining a functor:
    \[
        (c,r): \Hom_{\ICsp(\mcC)}((A,c_A), (B,c_B)) \longrightarrow
        \Hom_{\ICsp(\mcC)}(\emptyset, \emptyset) \times 
        \Hom_{\ICsp(\mcC)}^\red((A,c_A), (B,c_B)).
    \]
    Here $c(\UL{X}, o) = (\UL{Y}, o_{|\UL{Y}})$ where $\UL{Y} \subset \UL{X}$
    is the subset containing those equivalence classes that are not 
    in the image of $f \amalg g: A \amalg B \to \UL{X}$.
    Similarly we let $r(\UL{X}, o) = (\UL{Z}, o_{|\UL{Z}})$ 
    where $\UL{Z} \subset \UL{X}$ is the image of $f \amalg g$.
    One now checks that $(c,r)$ is indeed a functor and that it is an isomorphism
    between $\ICsp((A,c_A),(B,c_B))$ and the full subgroupoid of the product 
    on those $((\UL{Y},o), (\UL{Z},o'))$ such that $\UL{Y}$ and $\UL{Z}$ are disjoint.
    Hence $(c,r)$ is fully faithful, and it is essentially surjective because 
    any tuple is isomorphic to one that satisfies the disjointness.
    It remains to check that the functor $\ot$ induces an equivalence of groupoids:
    \[
        \Hom_{\ICsp(\mcC)}^\red((A,c_A), (B,c_B))
        \longrightarrow
        \Hom_\mcC^\red(\ot(A,c_A), \ot(B,c_B)).
    \]
    By lemma \ref{lem:decomposing-morphisms} this functor induces a bijection
    between the isomorphism classes of objects on the left 
    and the elements on the right.
    Moreover, any object in $\Hom_{\ICsp(\mcC)}^\red((A,c_A), (B,c_B))$
    only has the identity automorphism as $\gp:\UL{X} \to \UL{X}$ 
    is uniquely determined by the requirement that it has to be compatible
    with the surjection $A \amalg B \to \UL{X}$.
    The right-hand side has only identity morphisms by definition
    and hence the two groupoids are equivalent.
    
    Note also that $\Hom_{\ICsp(\mcC)}(\emptyset, \emptyset)$ is equivalent 
    to the groupoid of finite sets labelled in $\Hom_\mcC^\con(1_\mcC, 1_\mcC)$,
    which is indeed the free symmetric monoidal groupoid on this set.
\end{proof}

\subsubsection{Higher labelled cospan categories}\label{subsubsec:higher-labelled-cospans}
While the list of conditions used to characterise labelled cospan categories
in definition \ref{defn:labelled-cospans} is rather long 
and might seem a bit arbitrary, the enhanced version $\ICsp(\mcC)$ satisfies
a much simpler condition:

\begin{lem}\label{lem:ICsp(mcC)-in-pullback}
    For any labelled cospan category $(\pi:\mcC \to \Csp)$ 
    the following diagram is a pullback square of $2$-categories:
    \[
    \begin{tikzcd}
        \ICsp(\mcC)^{2\bot} \ar[r, "\amalg"] \ar[d, "(\mathrm{fgt})^2"] &
        \ICsp(\mcC) \ar[d, "\mathrm{fgt}"] \\
        \ICsp^{2\bot} \ar[r, "\amalg"] &
        \ICsp.
    \end{tikzcd}
    \]
\end{lem}
\begin{proof}
    On objects we have that 
    \begin{align*}
        \{ A, B \subset \gO,&\ l:A,B \to \Obj^\con(\mcC) \;|\; A \cap B = \emptyset\} \\
        \cong& \{A, B \subset \gO \;|\; A \cap B = \emptyset\} \times_{\{C \subset \gO\}}
        \{C \subset \gO, l:C \to \Obj^\con(\mcC)\}
    \end{align*}
    and a similar statement is true for $1$-morphisms and $2$-morphisms.
\end{proof}

This motivates the following definition, which we will give in the language of 
$\infty$-categories.
\begin{defn}\label{defn:labelled-infinity-cospan}
    A labelled cospan $\infty$-category is a symmetric monoidal $\infty$-category
    $\mcC$ together with a symmetric monoidal functor $\pi:\mcC \to \ICsp$.
    This functor is subject to the condition that 
    \[
    \begin{tikzcd}
        \mcC \times \mcC \ar[r, "\ot"] \ar[d, "\pi\times \pi"] &
        \mcC \ar[d, "\pi"] \\
        \ICsp \times \ICsp \ar[r, "\amalg"] &
        \ICsp.
    \end{tikzcd}
    \]
    is a pullback square of $\infty$-categories.
\end{defn}

We defer a more detailed study of this definition to \cite{BS22} -- 
for now it mainly serves as an aesthetically more 
pleasing analogue of definition \ref{defn:labelled-cospans}.
However, let us briefly note that one can show that for any 
labelled cospan $\infty$-category 
$(\mcC \to \ICsp)$ its homotopy category is a labelled cospan category 
$(h\mcC \to h\ICsp \simeq \Csp)$ in the sense of definition \ref{defn:labelled-cospans}.
Together with lemma \ref{lem:ICsp(mcC)-in-pullback} this can be used to prove that labelled cospan categories form a full reflective subcategory of the $\infty$-category of labelled $\infty$-categories of cospans.

Earlier versions of this article included a conjecture that relates the notions of labelled cospan category and labelled $\infty$-category of cospans to previously established notions.
In the meantime this conjecture has been proven: 
the biequivalence to coloured properads has been proven in \cite{BH22}, and the conjectured characterisation of $\infty$-properads is proven in \cite{BS22}.
\begin{thm}[Beardsley--Hackney, Barkan--Steinebrunner]\label{conj:labelled-Csp=properads}
    There is a biequivalence between the $2$-category of labelled cospan categories
    and the $2$-category of coloured properads as described in 
    \cite[Chapter 3]{HRY15}.
    
    Let $\Cat_\infty^\ot$ be the $\infty$-category of symmetric monoidal $\infty$-categories
    and let $\mcP \subset (\Cat_\infty^\ot)_{/\ICsp}$ be the full subcategory
    of the comma category on those $(\mcC \to \ICsp)$ that satisfy the condition 
    of definition \ref{defn:labelled-infinity-cospan}.
    Then $\mcP$ is equivalent to the $\infty$-category of $\infty$-properads 
    described in \cite[Chapter 7]{HRY15}.
\end{thm}

\section{The decomposition theorem}\label{sec:decomposition}
In this section we decompose the classifying space of a labelled cospan category
$(\mcC \to \Csp)$ into various smaller parts that are usually easier to compute.
We will also show that the two variants $\ICsp(\mcC)$ and $\mcC^\red$ 
are closely related to $\mcC$.

For the purpose of this section we fix a labelled cospan category $(\mcC \to \Csp)$ and we assume that $B\mcC$ is group-like. 
In lemma \ref{lem:BmcC-group-like} we explain that this assumption is often satisfied,
and we note that in this case $B\mcC$ is equivalent to an infinite loop space.
Moreover, we let $G := \Hom_\mcC^\con(1_\mcC, 1_\mcC)$ denote the set
of connected endomorphisms of the unit in $\mcC$.
We prove the following comparison result for $\ICsp(\mcC)$, $\mcC$, and $\mcC^\red$,
which is a generalization of \cite[Theorem B]{Stb19}. 
\begin{prop}\label{prop:closed-reduced}
    For every labelled cospan category $\mcC$ such that $B\mcC$ is grouplike, there are homotopy fiber sequences of infinite loop spaces:
    \[
        \begin{tikzcd}
            Q\left(\bigvee_G S^1\right) \ar[d] \ar[r] & 
            B(\ICsp(\mcC)) \ar[d] \ar[r] & 
            B(\mcC^\red) \ar[d, equal] \\
            B\left(\bigoplus_G \IN\right) \ar[r] & 
            B(\mcC) \ar[r] & 
            B(\mcC^\red).
        \end{tikzcd}
    \]
    In particular, the map $B(\ICsp(\mcC)) \to B(\mcC)$ is a rational equivalence.
\end{prop}

The next theorem will decompose the category $\mcC$ into parts 
with no closed components and factorisations of closed components.
We first make the necessary definitions.
\begin{defn}\label{defn:Nnc}
    The simplicial set $N_\cd^\nc\mcC$ is the simplicial subset 
    of the nerve $N_\cd \mcC$ containing only those $n$-simplices $W:[n] \to \mcC$
    where $W(0 \to n)$ is a reduced morphism.
\end{defn}

We briefly check that $N_\cd^\nc \mcC$ is indeed closed under the simplicial structure maps.
For inner face maps and degeneracies this is clear as they do not change the total composite $W(0 \to n)$.
For the outer face maps, suppose we had $W \in N_n^\cd\mcC$ such that $W(1 \to n)$ is not reduced; the case of $W(0 \to n-1)$ is analogous.
Then we can write it as $V \otimes Q$ for some non-trivial $Q\colon 1_\mcC \to 1_\mcC$, and hence $W(0 \to n) = (V \circ W(0 \to 1)) \otimes Q$ would also not be reduced -- a contradiction.

\begin{defn}\label{defn:F_g}
    For any connected endomorphism of the unit $s:1_\mcC \to 1_\mcC$ 
    we define the factorisation category $\mcF_s(\mcC)$, or $\mcF_s$ for short, as follows.
    An object is a triple $(M, W, V)$ where $M \not\cong 1_\mcC$ is a non-trivial object of $\mcC$ 
    and $W:1_\mcC \to M$ and $V:M \to 1_\mcC$ are morphisms such that $V \circ W = s$.
    A morphism $(M,W,V) \to (M',W',V')$ is a morphism $X:M \to M'$ in $\mcC$
    such that $X \circ W = W'$ and $V = V' \circ X$.
\end{defn}

\begin{thm}[Decomposition Theorem]\label{thm:decomposition}
    For any labelled cospan category $\mcC$ such that $B\mcC$ is group-like
    there is a homotopy fiber sequence of infinite loop spaces:
    \[
        |N_\cd^\nc\mcC| \longrightarrow B(\ICsp(\mcC)) \longrightarrow
        Q\left(\bigvee_{g \in G} S^2(B\mcF_g)\right).
    \]
    Here the wedge runs over the set of connected morphisms $g:1_\mcC \to 1_\mcC$.
\end{thm}
We will prove this theorem at the end of subsection \ref{subsec:cut-to-fact}.

\begin{rem}
    The notation $S^2X$ denotes the unreduced double-suspension of an unpointed space. 
    (See example \ref{ex:cone-and-unreduced-suspension}.)
    If we assume that each $\mcF_g$ is non-empty, then we may choose any 
    basepoint to form the reduced suspension and the result will 
    be equivalent to the unreduced suspension: $S^2 (B\mcF_g) \simeq \gS^2 B\mcF_g$.
    However, it can happen that a connected endomorphism $g:1_\mcC \to 1_\mcC$
    does not admit a factorization. (For example if one adds spheres to $\Cobn$.)
    In this case $B\mcF_g = \emptyset$ and by convention $S(\emptyset) = S^0$,
    so $S^2(B\mcF_g) = S^1$.
\end{rem}

\begin{rem}
One might think of the term $Q(\bigvee_{g \in G} \gS S( B\mcF_g))$ 
as an ``error-term" that obstructs $B\ICsp(\mcC)$ from being equivalent
to the subspace with no closed components. 
In the analogy where $\mcC$ is a topologically enriched cobordism category
this term always vanishes: it is a theorem by \cite[Definition 4.3 and lemma 4.7]{GRW10} 
that in these cases $B\mcC$ is equivalent to the subspace with no closed components.
Note however, that this is only an analogy as our setting of labelled cospan categories
does not actually accommodate for topologically enriched cobordism categories,
and moreover the definition of $D_\cd^\nc$ is different in the topological case.
\end{rem}

\begin{rem}
    If $B\mcC$ is not group-like, then the above is still a homotopy cofiber sequence
    of $E_\infty$-algebras. However, this case will not be as useful for computations.
    One of the particular advantages of the group-like case is that any 
    homotopy cofiber sequence of infinite loop spaces is also a homotopy fiber sequence,
    and moreover induces a homotopy fiber sequence of underlying spaces.
    In the general case we still have a homotopy fiber sequence after passing to group-completions.
\end{rem}

\begin{ex}
    In section \ref{subsec:mcF-contractible} we will show that for 
    $\mcC = \Csp$ and $\mcC = \Cob_d$ with $d\ge 2$ all the factorization categories 
    $\mcF_g(\mcC)$ have a contractible classifying space. 
    In these cases it follows that $|N_\cd^\nc \mcC| \simeq B(\ICsp(\mcC))$.
\end{ex}

If $N_\cd^\nc \mcC$ were the nerve of a category that functorically depends on $\mcC$, then it would automatically be homotopy invariant under equivalence in $\mcC$.
However, in general $N_\cd^\nc \mcC$ is not the nerve of any category, so we prove invariance by hand.
\begin{lem}\label{lem:Nnc-invariant}
    Let $(\pi:\mcC \to \Csp)$ and $(\pi':\mcD \to \Csp)$ be two labelled cospan categories
    and let $F:\mcC \to \mcD$ be an equivalence of symmetric monoidal categories
    together with a symmetric monoidal natural isomorphism $F \circ \pi \cong \pi'$.
    Then $F$ induces an equivalence:
    \[
        |N_\cd F|: |N_\cd^\nc \mcC| \xrightarrow{\ \simeq\ } |N_\cd^\nc \mcD|.
    \]
\end{lem}
\begin{proof}
    Pick a homotopy-inverse functor $G:\mcD \to \mcC$ and natural isomorphisms
    $\ga:F \circ G \cong \Id_{\mcD}$ and $\gb:G \circ F \cong \Id_{\mcC}$.
    Then we have maps $BF:B\mcC \longleftrightarrow B\mcD:BG$ 
    and $\ga$ and $\gb$ yield homotopies that witness that $BF$ and $BG$ are homotopy inverse.
    This homotopy equivalence restricts to a homotopy equivalence
    $|N_\cd^\nc \mcC| \simeq |N_\cd^\nc \mcD|$.
    To see this, note that if $a:[n] \to \mcC$ represents an $n$-simplex $a \in N_n\mcC$
    that lies in $N_n^\nc \mcC$ and $a':[n] \to \mcC$ is naturally isomorphic to $a$,
    then $a'$ lies in the no-closed components part as well: $a' \in N_n^\nc \mcC$.
\end{proof}

We briefly discuss some sufficient conditions for $B\mcC$ to be group-like.
\begin{lem}\label{lem:BmcC-group-like}
    Let $(\mcC \to \Csp)$ be a labelled cospan category satisfying that
    for all connected $M \in \mcC$ there is some object $N \in \mcC$ for which 
    there exists a morphism $W:M \ot N \to 1_\mcC$ or a morphism
    $V:1_\mcC \to M \ot N$.
    Then $\pi_0B\mcC$ is group-like and hence $B\mcC$ is equivalent 
    to an infinite loop space.
\end{lem}
\begin{proof}
    The connected components of the classifying space $\pi_0 B(\mcC)$ 
    can be computed as the commutative monoid of isomorphism classes 
    of objects of $\mcC$ modulo the relation $A \sim B$ whenever 
    there is a morphism from $A$ to $B$.
    By axiom (i) of definition \ref{defn:labelled-cospans} every object of $\mcC$
    can be written as a $\ot$-product of connected objects.
    So the connected objects generate $\pi_0 B(\mcC)$.
    The conditions of this lemma imply that for each connected object $M$
    there is some object $N$ with $[M] \ot [N] = [1_\mcC]$ in $\pi_0 B(\mcC)$.
    This implies that all the connected objects have inverses and
    as these generate the monoid we can conclude that $\pi_0 B(\mcC)$
    is indeed a group under $\ot$.
    It now follows from the methods of \cite{Seg74},
    which we recall in section \ref{defn:PCMC},
    that $B\mcC$ is equivalent to an infinite loop space.
\end{proof}

\subsection{Partial commutative monoidal categories}\label{subsec:PCMC}
Throughout this paper we will work with infinite loop spaces that arise as 
the classifying space of symmetric monoidal categories.
The symmetric monodial categories that we consider will often be like 
embedded cobordism categories in the sense that we can make sense of a notion
of disjointness $M\cap N =\emptyset$ for objects $M, N$ and that
for disjoint objects taking their union $(M, N) \mapsto M \cup N$
is a strictly associative and commutative operation.
To extract a symmetric monoidal structure from this we need to fix a
functorial replacement $(M, N) \mapsto (M', N')$ such that $M' \cap N' = \emptyset$.
This choice is always somewhat arbitrary, which would make it difficult
to construct the simplicial objects in symmetric monoidal categories that we need later.
Instead of making a choice of a replacement we will work with the notion
of a partial commutative monoidal category, which formalises the idea 
of having a notion of disjointness and disjoint union.
These partial commutative monoidal categories directly yield $\gC$-categories
(and hence special $\gC$-spaces) in the sense of Segal, 
as we shall see below.

\begin{defn}\label{defn:PCMC}
    A \emph{partial commutative monoidal} structure 
    on a category $\mcC$ is a choice of a relation $\bot$ on the set of objects of $\mcC$,
    an object $\emptyset \in \mcC$, and a functor $\amalg:\mcC^{2\bot} \to \mcC$
    where $\mcC^{2\bot} \subset \mcC^2$ is the full subcategory on those $(x,y)$ where $x \bot y$.
    This data is subject to the following axioms:
    \begin{itemize}
        \item[(i)]
        For all $x \in \mcC$ we have $\emptyset \bot x$ and the functor 
        $\mcC \to \mcC$ defined by $x \mapsto \emptyset \amalg x$ is the identity.
        \item[(ii)]
        Let $\mcC^{3\bot} \subset \mcC^3$ denote the full subcategory on triples $(x,y,z)$
        satisfying $x \bot y$, $x \bot z$, and $y \bot z$. 
        Then for each such triple we have $x \bot (y \amalg z)$ and $(x \amalg y) \bot z$, and 
        the two functors $\mcC^{3\bot} \to \mcC$ defined by $(x,y,z) \mapsto (x \amalg y) \amalg z$
        and $(x,y,z) \mapsto x \amalg (y \amalg z)$ are equal.
        \item[(iii)]
        The relation $\bot$ is symmetric and the two functors $\mcC^{2\bot} \to \mcC$ 
        defined by $(x,y) \mapsto x \amalg y$ and $(x,y) \mapsto y \amalg x$ are equal.
        \item[(iv)]
        For any $n$-tuple $(x_1, \dots, x_n) \in \mcC^n$ one can find $(x_1',\dots,x_n') \in \mcC^n$
        such that $x_i \cong x_i'$ for all $i$ and $x_i' \bot x_j'$ for all $i \neq j$.
    \end{itemize}
    Given two such partial commutative monoidal categories 
    -- or PCM categories for short --
    $(\mcC, \bot_\mcC, \emptyset_\mcC, \amalg_\mcC)$ and 
    $(\mcD, \bot_\mcD, \emptyset_\mcD, \amalg_\mcD)$,
    we say that a commutative monoidal functor (or PCM-functor) between them is a functor $F:\mcC \to \mcD$
    satisfying the properties: 
    $F(\emptyset_\mcC) = \emptyset_\mcD$, $x \bot y \Rightarrow F(x) \bot F(y)$,
    and $\amalg_\mcD \circ (F, F) = F \circ \amalg_\mcC$ as functors $\mcC^{2\bot} \to \mcD$.
\end{defn}

\begin{ex}\label{ex:PCM-vs-sm}
    From any symmetric monoidal category $\mcC$ we can construct a PCM category $\mcC'$ equivalent to $\mcC$ by following a variant of the strategy we used for $\ICsp(-)$ in subsection \ref{subsubsec:the-general-case}.
    Let $\Omega$ be some infinite set.
    Then objects of $\mcC'$ are tuples $(A,c)$ where $A \subset \Omega$ is finite and $c: A \to \mcC$ is a labelling of $A$ by objects of $\mcC$.
    Morphisms and composition in $\mcC'$ are defined in the unique way such that there is a fully faithful functor (and thus an equivalence) $\mcC' \to \mcC$ that sends $(A,c)$ to the tensor product $\bigotimes_{a \in A} c(a)$ (see definition \ref{defn:tensor-over-set}).
    We can then define a PCM structure on $\mcC'$ where $(A,c_A) \bot (B,c_B)$ if and only if $A \cap B = \emptyset$ and the operation $\amalg$ is defined by taking unions. 

    Conversely, given a PCM category $\mcD$ we can first construct a $\Gamma$-category as discussed in \ref{lem:PCM-yields-gC} below, and this in particular induces a symmetric monoidal structure on $\mcD$ \cite{Seg74}.
    These two constructions are inverse to one another in a suitable sense, but we will not need the details here.
\end{ex}

\begin{ex}\label{ex:Cob-PCM}
    We define a partial commutative monoidal structure on the cobordism category $\Cob_d$ as follows.
    To make precise the notion of disjointness, suppose that the objects of $\Cob_d$ are in fact submanifolds of $\IR^\infty$, as is for example the case when we define $\Cob_d$ has the homotopy category of the topological bordism category.
    Then we set $M \bot N$ on objects whenever $M$ and $N$ are disjoint.
    The functor $\amalg: (\Cob_d)^{2\bot} \to \Cob_d$ is then defined on objects by taking the (disjoint) union.
    On morphisms we define it by sending $([W],[V]): (M_1,N_1) \to (M_2,N_2)$ to $[W \amalg V]: M_1 \cup N_1 \to M_2 \cup N_2$.
    Here the set-theoretic details of the disjoint union $[W \amalg V]$ do not matter as it only defines a diffeomorphism class of bordisms -- we can simply pick representatives of $[W]$ and $[V]$ that are already disjoint.
\end{ex}

\begin{rem}
    One can think of a given PCM category as a category internal to the category
    of partial commutative monoids. Conversely, a category $\mcC$ internal
    to partial commutative monoids represents a PCM category if it satisfies 
    the following conditions:
    \begin{itemize}
        \item For any two morphisms $f:x \to y$ and $g:x' \to y'$ we have that
        $f \bot g$ if and only if $x \bot x'$ and $y \bot y'$.
        \item For any $n$-tuples of objects $(x_1, \dots, x_n)$
        one can find $(x_1',\dots,x_n') \in \mcC^n$
        such that $x_i \cong x_i'$ for all $i$ and $x_i' \bot x_j'$ for all $i \neq j$.
    \end{itemize}
\end{rem}

\begin{rem}\label{rem:PCM-functor-property}
    The key advantage of this definition is that it is a property for 
    a functor $F$ to be commutatively monoidal. This will make 
    it much easier to construct simplicial objects in PCM categories
    than it would be to construct them in the category of symmetric monoidal
    categories.
\end{rem}

To obtain infinite loop spaces from PCM categories we will use Segal's $\gC$-spaces,
which we now recall.

\begin{defn}
    Let $\gC^\op$ denote the full subcategory of the category of finite pointed
    sets $\mathrm{Fin}_*$ on the objects $\gle{n} = \{*,1,\dots,n\}$.
    Concretely, the morphisms $\gle{n} \to \gle{m}$ in $\gC^\op$
    are maps $f:\{*,1,\dots, n\} \to \{*,1,\dots, m\}$ satisfying $f(*) = *$.
    We let $\rho_n^i:\gle{n} \to \gle{1}$ denote the unique morphism with
    $(\rho_n^i)^{-1}(1) = \{i\}$.
    
    A $\gC$-space is a functor $X: \gC^\op \to \Top$ and we denote 
    the value on $\gle{n}$ by $X^{\gle{n}}$. We say that $X$
    is a \emph{special} $\gC$-space if for each $n$ the following map
    is a weak equivalence:
    \[
        (\rho_n^1, \dots, \rho_n^n): X^{\gle{n}} 
        \to X^{\gle{1}} \times \dots \times X^{\gle{1}}.
    \]
\end{defn}

Our main source of special $\gC$-spaces will be of the form $B\mcC^{\gle{\cd}}$ 
for $\mcC^{\gle{\cd}}$ is a special $\gC$-category.
\begin{defn}\label{defn:pcm-to-gamma}
    A $\gC$-category is a functor $\mcC^{\gle{\cd}}: \gC^\op \to \Cat$.
    We say that $\mcC^{\gle{\cd}}$ is a \emph{special} if for each $n$ the map 
    $
        (\rho_n^1, \dots, \rho_n^n): \mcC^{\gle{n}} 
        \to \mcC^{\gle{1}} \times \dots \times \mcC^{\gle{1}}
    $
    is an equivalence of categories.
\end{defn}

\begin{lem}\label{lem:PCM-yields-gC}
    For any partial commutative monoidal category $(\mcC, \bot, \emptyset, \amalg)$
    we can define a special $\gC$-category be letting $\mcC^{\gle{n}} := \mcC^{n\bot}$
    be the full subcategory of\, $\mcC^n$ on those tuples $(x_1, \dots, x_n)$ satisfying 
    $x_i \bot x_j$ for all $i \neq j$.
\end{lem}
\begin{proof}
    It follows from axioms (i-iii) that for any finite family of objects 
    $(x_i \in \mcC)_{i \in I}$
    that are disjoint in the sense that $i \neq j \Rightarrow x_i \bot x_j$
    we can define:
    \[
        \coprod_{i\in I} x_i := 
        (\dots((x_{i_1} \amalg x_{i_2}) \amalg x_{i_3}) \dots \amalg x_{i_n})
    \]
    and that this is independent of the enumeration $(i_1, \dots, i_n)$ one chooses of $I$.
    Here we set $\coprod_{i \in \emptyset} x_i := \emptyset_\mcC$.
    We may therefore define for each $\gl:\gle{n} \to \gle{m}$ a functor
    \[
        \gl_*:\mcC^{n\bot} \to \mcC^{m\bot}, \quad
        (x_1, \dots, x_n) \mapsto (y_1,\dots, y_m), \text{ where }
        y_j = \coprod_{i \in \gl^{-1}(j)} x_i.
    \]
    By iteratively applying (ii) we get that $y_j \bot y_{j'}$ for $j \neq j'$.
    This makes $\gle{n} \mapsto \mcC^{n\bot}$ into a $\gC$-category. 
    It remains to check that this is special, i.e.\ that the inclusion 
    $\mcC^{n\bot} \to (\mcC^{1\bot})^n = \mcC^n$ is an equivalence of categories.
    This inclusion is fully faithful by definition and it is essentially surjective
    by axiom (iv).
\end{proof}

\begin{rem}\label{rem:BC-PCM}
    A partial commutative monodial structure on $\mcC$ also induces a partially defined commutative operation on the topological space $B\mcC$.
    The partially defined multiplication is given by
    \[
        B\mcC \times B\mcC \cong B(\mcC \times \mcC) \supset B(\mcC^{2\bot}) \xrightarrow[\qquad]{\cup} B(\mcC),
    \]
    and we now describe this in coordinates.
    A point in $B\mcC$ can represented as
    \[
        [x_1 \xrightarrow{f_1} \dots \xrightarrow{f_n} x_n, (t_0, \dots, t_n) ]
    \]
    with $\sum_i t_i = 1$.
    Using the simplicial relations we can introduce an identity in the $i$th place and split $t_i= t_i^0 + t_i^1$.
    This means that given two points in $B\mcC$ we can always represent them by tuples whose $\Delta^n$-coordinate agrees. 
    (This process is canonical, up to removing unnecessary identity morphisms, as $B\mcC \times B\mcC \cong B(\mcC \times \mcC)$.)
    On such representatives we have
    \begin{align*}
        [x_1 \xrightarrow{f_1} \dots \xrightarrow{f_n} x_n, (t_0, \dots, t_n) ]
        &+ [y_1 \xrightarrow{g_1} \dots \xrightarrow{g_n} y_n, (t_0, \dots, t_n) ]\\
        &= [x_1 \amalg y_1 \xrightarrow{f_1 \amalg g_1} \dots \xrightarrow{f_n \amalg g_n} x_n \amalg y_n, (t_0, \dots, t_n) ]
    \end{align*}
    if $x_i \bot y_i$ for all $i$, and undefined otherwise.
    Note that the $\Gamma$-space we get by applying $B$ to the $\Gamma$-category $\mcC^{\langle\cd\rangle}$ from lemma \ref{lem:PCM-yields-gC} is isomorphic to the $\Gamma$-space that we can construct from the partially defined topological monoid $B\mcC$, by sending $\langle k\rangle$ to the space of composable $k$-tuples in $B\mcC$.
\end{rem}

We also briefly describe how to generalize this to the setting of $2$-categories.
\begin{defn}\label{defn:PCM-2-cat}
    For a $2$-category $\mcC$ let $\Mor(\mcC)$ be the $1$-category
    where objects are $1$-morphisms $f:x \to y$ in $\mcC$
    and morphisms $f \to g$ are $2$-morphisms $\ga:f \Rightarrow g$.
    In particular there are no morphisms between $f$ and $g$
    unless the have the same source and the same target.
    
    A partial commutative monoidal structure on a $2$-category $\mcC$ 
    is a partial commutative monoid structure $(\id_\emptyset, \bot, \amalg)$
    on $\Mor(\mcC)$ where the neutral object is required to be an identity morphism -- except that instead of condition (iv) we require:
    \begin{itemize}
        \item[(iv')]
        For any $n$-tuple $(x_1, \dots, x_n) \in \mcC^n$
        one can find an $n$-tuple $(x_1', \dots, x_n') \in \mcC^n$
        such that $x_i$ is equivalent to $x_i'$ for all $i$ 
        and $\id_{x_i'} \bot \id_{x_j'}$ for all $i \neq j$.
        \item[(iv'')]
        For any $n$-tuple $(f_1:x_1 \to y_1, \dots, f_n:x_n \to y_n) \in \Mor(\mcC)^n$
        satisfying $\id_{x_i} \bot \id_{x_j}$ and $\id_{y_i} \bot \id_{y_j}$ for all $i \neq j$
        one can find 
        an $n$-tuple $(f_1':x_1 \to y_1, \dots, f_n':x_n \to y_n) \in \Mor(\mcC)^n$
        such that $f_i \cong f_i'$ for all $i$ and $f_i' \bot f_j'$ for all $i \neq j$.
        \item[(v)] Composition defines a PCM functor 
        $\circ: \Mor(\mcC) \times_{\Obj(\mcC)} \Mor(\mcC) \to \Mor(\mcC)$ where $\Obj(\mcC)$ denotes the set of objects of $\mcC$.
    \end{itemize}
    We will also write $x \bot y$ for two objects $x, y \in \mcC$ 
    to mean $\id_x \bot \id_y$.
\end{defn}

\begin{cor}
    For any PCM $2$-category $\mcC$ the classifying space $B\mcC$
    naturally has the structure of a special $\gC$-space.
\end{cor}
\begin{proof}
    We would like to mimic the proof of lemma \ref{lem:PCM-yields-gC}.
    For all $n$ let $\mcC^{\bot n} \subset \mcC^n$ be the following sub $2$-category.
    It contains an object $(x_1,\dots,x_n)$ if $\id_{x_i} \bot \id_{x_j}$
    holds for all $i \neq j$, and it contains a $1$-morphism $(f_1,\dots,f_n)$
    between two such objects if $f_i \bot f_j$ holds for all $i \neq j$. 
    It contains all $2$-morphism between any $1$-morphisms it contains.
    By condition (v) this is indeed a sub $2$-category.
    In the same way as in lemma \ref{lem:PCM-yields-gC} these $2$-categories
    assemble into a $\gC$-object in $2$-categories, and 
    therefore $\gle{n} \mapsto B\mcC^{\bot n}$ defines a $\gC$-space.
    
    We need to show that the Segal map $B\mcC^{\bot n} \to B\mcC^n$ is a weak equivalence.
    The functor $\mcC^{\bot n} \to \mcC^n$ is
    surjective up to equivalence by condition (iv') and 
    by condition (iv'') it is an equivalence on $\Hom$-categories.
    The classifying space of a $2$-category 
    is defined by first taking the nerve of $\Hom$-categories to obtain
    a simplicially enriched category and then taking the usual classifying space.
    (See definition \ref{defn:2nerve} for the case of a $(2,1)$-category.)
    By the above, the simplicially enriched functor obtained from 
    the $2$-functor $\mcC^{\bot n} \to \mcC^n$ is a Dwyer-Kan equivalence,
    and hence it induces an equivalence on classifying spaces.
\end{proof}

\subsection{A general fiber sequence}

We now introduce the abstract technique that we will use to show that certain
sequences of simplicial symmetric monoidal groupoids yield fiber sequences 
of infinite loop spaces.

\begin{thm}\label{thm:fiber-sequence}
    Consider three simplicial special $\gC$-spaces 
    $A_\cd, B_\cd, C_\cd: \gD^\op \times \gC^\op \to \Top$
    with two simplicial $\gC$-maps:
    \[
        A_\cd \xrightarrow{\ f\ }  B_\cd \xrightarrow{\ g\ } C_\cd
    \]
    and a homotopy of simplicial $\gC$-maps $\ga: g \circ f \simeq 1_C$ 
    to the unit in $C_\cd$.
    Assume further that for every $n$ there is a $\gC$-map
    $s_n: B_n \to A_n$ such that the induced maps
    \[
        (s_n, g_n): B_n^{\gle{1}} \to A_n^{\gle{1}} \times C_n^{\gle{1}}
        \qand
        s_n \circ f_n: A_n^{\gle{1}} \to A_n^{\gle{1}}
    \]
    are weak equivalences of spaces.
    Then the following is a homotopy cofiber sequence of special $\gC$-spaces
    \[
        \|A_\cd\| \xrightarrow{\ f\ }  \|B_\cd\| \xrightarrow{\ g\ } \|C_\cd\|.
    \]
    If moreover each of these spaces is group-like, then this is a homotopy fiber sequence.
\end{thm}
\begin{proof}
    We will base this proof on the framework of \cite{GGN15}.
    Let $\mcS$ denote the $\infty$-category of spaces, which receives a functor
    $\Top \to \mcS$. The category $\Mon_{E_\infty}(\mcS) \subset \Fun(\gC^\op, \mcS)$
    is defined as the full subcategory on those functors that satisfy a Segal-condition.
    The homotopy category of special $\gC$-spaces is pre-additive, 
    i.e.\ it has a $0$-object $*$ and the canonical map $x \amalg y \to x \times y$
    from the coproduct to the product is an isomorphism for all $x$ and $y$.
    This allows us to apply lemma \ref{lem:pre-additive} and conclude from 
    the existence of $s_n$ that for each $n$ the sequence:
    \[
        A_n \longrightarrow B_n \longrightarrow C_n
    \]
    is a homotopy cofiber sequence of special $\gC$-spaces.
    Since homotopy colimits commute, taking the homotopy colimit over $\gD^\op$ 
    yields the following cofiber sequence in $\Mon_{E_\infty}(\mcS)$:
    \[
        \mathrm{hocolim}_{\gD^\op} A_\cd \longrightarrow 
        \mathrm{hocolim}_{\gD^\op} B_\cd \longrightarrow 
        \mathrm{hocolim}_{\gD^\op} C_\cd.
    \]
    To compute this homotopy colimit note that the functor 
    $\Mon_{E_\infty}(\mcS) \to \mcS$ that sends a special $\gC$-space $A$
    to the underlying space $A^{\gle{1}}$ is conservative and commutes with sifted colimits.
    Therefore 
    $(\mathrm{hocolim}_{\gD^\op} A_\cd)^{\gle{1}} 
    \simeq \mathrm{hocolim}_{\gD^\op} A_\cd^{\gle{1}}$
    where the latter is computed as a homotopy colimit in $\mcS$.
    Given a simplicial topological space $X_\cd:\gD^\op \to \Top$
    the homotopy colimit of $X_\cd:\gD^\op \to \Top \to \mcS$ is modelled
    by the fat geometric realization $\|X_\cd\|$. 
    Similarly, for a simplicial special $\gC$-space 
    $A_\cd^{\gle{\cd}}: \gD^\op\times \gC^\op \to \Top$
    we can construct a $\gC$-space $\gle{n} \mapsto \|A_\cd^{\gle{n}}\|$
    by taking the level-wise realization.
    Note that this $\gC$-space is still special since $\|\blank\|$ commutes with products
    up to weak equivalence (see \cite[Proposition A.1(iii)]{Seg74} or \cite[Theorem 7.2]{ERW19}).
    It follows by our previous comments that $\|A_\cd\|$ is in fact 
    a model for the homotopy colimit in special $\gC$-spaces. 
    This shows that 
    \[
        \|A_\cd\| \xrightarrow{\ f\ }  \|B_\cd\| \xrightarrow{\ g\ } \|C_\cd\|
    \]
    is a homotopy cofiber sequence in $\Mon_{E_\infty}(\mcS)$.
    
    Finally we would like to show that $\|A_\cd\| \to \|B_\cd\| \to \|C_\cd\|$ 
    is homotopy fiber sequence in $\mcS$ if we assume that each of the spaces is group-like.
    Indeed, for any cofiber sequence $X \to Y \to Z$ in $\Mon_{E_\infty}(\mcS)$,
    the group completion $X^{gp} \to Y^{gp} \to Z^{gp}$ is a fiber sequence.
    If we assume that each of $\|A_\cd\|$, $\|B_\cd\|$, and $\|C_\cd\|$ is already group-like,
    then it follows that $\|A_\cd\| \to \|B_\cd\| \to \|C_\cd\|$ is a homotopy fiber sequence.     
\end{proof}

The following lemma holds in any pre-additive $\infty$-category $\mcC$,
i.e.\ in every $\infty$-category $\mcC$ such that the homotopy category $h\mcC$
has a $0$-object and such that coproducts and products agree via the canonical map.
The same argument applies in any model category whose homotopy category is pre-additive.
\begin{lem}\label{lem:pre-additive}
    Let $f:A \to B$ and $g:B \to C$ be two morphisms in a pre-additive 
    $\infty$-category $\mcC$ such that $g \circ f$ is null-homotopic.
    Assume further that there is a map $s:B \to A$ such that the maps
    $(s,g):B \to A \times C$ and $s \circ f:A \to A$ are equivalences.
    Then any choice of a null-homotopy for $g \circ f$ makes $A \to B \to C$
    into a cofiber sequence in $\mcC$.
\end{lem}
\begin{proof}
    Pick a morphism $r:C \to B$ such that $s \circ r = 0: C \to A$ and $g \circ r = \id_C$.
    This can be done by lifting $(0,\id_C): C \to A \times C$ 
    against the equivalence $(s,g):B \to A \times C$.
    We claim that the maps $f$ and $r$ exhibit $B$ as a coproduct of $A$ and $C$.
    So we need to check that $f \amalg r:A \amalg C \to B$ is an equivalence.
    Since $(s,g)$ is an equivalence it will suffice to check that $(s,g) \circ (f \amalg r)$
    is an equivalence. Using that $h\mcC$ is pre-additive we can represent this morphism
    by a $2\times 2$ matrix with entries $s \circ f$, $s \circ r$, $g \circ f$, 
    and $g \circ r$.
    By construction $s \circ r$ and $g \circ f$ are $0$-maps and $s \circ f$ 
    and $g \circ r$ are equivalences. Hence $(s,g) \circ (f \amalg r)$ is the direct
    sum of two equivalences and therefore itself an equivalence.
    
    Consider the following diagram in the homotopy category $h \mcC$:
    \[
    \begin{tikzcd}
        * \ar[r] \ar[d] &    
        A \ar[r] \ar[d, "f"] &    
        * \ar[d] \\
        C \ar[r, "r"] &
        B \ar[r, "g"] &
        C
    \end{tikzcd}
    \]
    This diagram can be lifted to a diagram $\gD^2 \times \gD^1 \to \mcC$ in the 
    $\infty$-category by choosing any null-homotopy of $g \circ f$ in the right square 
    and the trivial homotopy in the left square.
    The left-hand square is a pushout square because we checked that $f$ and $r$
    exhibit $B$ as the coproduct of $A$ and $C$. 
    The entire rectangle is a pushout square (independent of which homotopies we chose!)
    because the top horizontal arrow $* \to *$ and the bottom horizontal arrow
    $g \circ r: C \to C$ both are equivalences.
    It therefore follows from the pushout pasting lemma \cite[Lemma 4.4.2.1.]{LurHTT}
    that the right-hand square 
    is a pushout square, which is exactly what we claimed.
\end{proof}

\subsection{A base-change theorem for simplicial spaces}

We prove a theorem that allows us to adjust the space of $0$-simplices $X_0$ of a simplicial space $X_\cd$ (and to change the spaces of $n$-simplices accordingly) without affecting the homotopy type of $\|X_\cd\|$.
This is motivated by the basic idea in category theory that 
one can change the number of objects of each isomorphism type without
changing the category, up to equivalence.
In this sense, the following definition can be thought of as a generalization
of the notion of a fully faithful functor.

\begin{defn}
    For any semi-simplicial object $X_\cd$ let $\partial_n^X:X_n \to (X_0)^{n+1}$ 
    denote the map that sends an $n$-simplex to its vertices.
    We call a map of semi-simplicial spaces $f:X_\cd \to Y_\cd$ 
    a \emph{homotopy base-change} if the map
    \[
        (\partial_n^X, f_n): X_n \longrightarrow 
        (X_0)^{n+1} \times_{(Y_0)^{n+1}}^h Y_n
    \]
    is an equivalence for all $n \ge 1$.
\end{defn}

The following theorem generalizes \cite[Theorem 5.2]{ERW19} 
and \cite[Lemma 2.30]{Stb21} to arbitrary simplicial spaces
that are not necessarily the nerve of a topological category.
Note, however, that \cite[Theorem 5.2]{ERW19} has the advantage 
of applying to weakly unital categories $\mcC$,
whereas we would need $\mcC$ to have strict units in order 
to apply the theorem to $N\mcC$.

\begin{thm}[Base-change]\label{thm:base-change}
    Let $f:X_\cd \to Y_\cd$ be a map of simplicial topological spaces 
    that is a homotopy base-change in the above sense 
    and satisfies that $\pi_0(f_0): \pi_0(X_0) \to \pi_0(Y_0)$ is surjective.
    Then $\|f\|: \|X_\cd\| \to \|Y_\cd\|$ is a weak equivalence.
\end{thm}
\begin{proof}
    For any space $A$ let $T(A)_\cd$ be the simplicial space defined by 
    $T(A)_n := A^{n+1}$ with face and degeneracy operators given by forgetting and repeating.
    A simplicial map $Z_\cd \to T(B)_\cd$ amounts to the same data as a map $Z_0 \to B$.
    We can factor the map $X_\cd \to Y_\cd$ in the theorem as:
    \[
        X_\cd \longrightarrow T(X_0)_\cd \times_{T(Y_0)_\cd}^h Y_\cd 
        \longrightarrow Y_\cd
    \]
    The first map is a level-wise equivalence because $X_\cd \to Y_\cd$ is a homotopy base-change,
    and in particular it induces an equivalence on geometric realizations.
    To prove the theorem we need to show that the second map also induces 
    and equivalence on geometric realizations.
    
    For the purpose of this proof let us interpret ``space" to mean
    an object of the category of simplicial sets $\mcS := \sSet_Q$,
    which we equip with the Quillen model structure.
    A simplicial space is hence a bisimplicial set and we equip 
    $\mcS^{\gD^\op} := \mrm{Fun}(\gD^\op, \sSet_Q)$ 
    with the projective model structure.
    This proof could also be translated into any other model,
    or it could be given in the $\infty$-category of simplicial spaces.%
    \footnote{
        Let us briefly comment on how to this proof would go in $\infty$-land.
        Claim (i) is proven in the same way. Claim (ii) uses the
        fact that the presheaf category (of $\gD$) is generated by
        representables under small colimits, see \cite[Corollary 5.1.5.8]{LurHTT}.
        Claim (iii) follows because in $\mcS$ colimits are stable under pullback:
        \cite[Lemma 6.1.3.14]{LurHTT}.
    }
    Since the statement of the theorem is invariant under level-wise
    weak equivalences we may assume that $X_0 \to Y_0$ is a fibration,
    and in this case we may replace the homotopy pullback by the strict pullback.
    
    In fact, we will show:
    for any surjective fibration of spaces (i.e.\ Kan-fibration of simplicial sets)
    $f:A \to B$ 
    and for any simplicial space $Z_\cd$ with map $p:Z_0 \to B$ the natural map
    \[
        \ga_{(Z_\cd,p)}: \|T(A)_\cd \times_{T(B)_\cd} Z_\cd\| \longrightarrow \|Z_\cd\|
    \]
    is an equivalence.
    This map defines a natural transformation of functors $\mcS^{\gD^\op}_{/T(B)} \to \mcS$.
    We will prove the theorem by showing:
    \begin{itemize}
        \item[(i)] The map $\ga_{(Z_\cd,p)}$ is an equivalence if $Z_\cd$ 
        is the discrete simplicial space $\gD^n$ for some $n$.
        \item[(ii)] The category $\mcS^{\gD^\op}_{/T(B)}$ is generated 
        under homotopy colimits and weak equivalences
        by objects of the form $(\gD^n, p)$.
        \item[(iii)] The source and target functor of $\ga$ both commute with homotopy colimits.
    \end{itemize}
    
    \noindent
    Claim (i): 
    For $Z_\cd = \gD^n$ the map $p:Z_\cd \to T(B)_\cd$ picks out $(n+1)$ points 
    $p(0), \dots, p(n) \in B$.
    Since $f:A \to B$ is surjective we can find $a \in A$ with $f(a) = p(0)$,
    which we can use to define an extra degeneracy:
    \begin{align*}
        s_{-1}: A^{k+1} \times_{B^{k+1}} (\gD^n)_k 
        &\longrightarrow A^{k+2} \times_{B^{k+2}} (\gD^n)_{k+1} \\
        ((a_0,\dots,a_k), (l_0,\dots,l_k)) 
        &\longmapsto ((a, a_0,\dots,a_k), (0, l_0,\dots,l_k)) 
    \end{align*}
    Hence $\|T(A)_\cd \times_{T(B)_\cd} \gD^n\|$ is contractible and 
    since $\|\gD^n\|$ is also contractible it follows that $\ga_{(\gD^n,p)}$ is an equivalence.

    \noindent
    Claim (ii):
    Let $(p:Z_\cd \to T(B)_\cd) \in \mcS^{\gD^\op}_{/T(B)}$ be any object
    in the over category. 
    The category $\mcS^{\gD^\op}$ is generated
    by the representables $\gD^n$ under homotopy colimits.
    Concretely, \cite[Proposition 2.9]{Dug01} shows that 
    we can find a diagram $F:I \to \mcS^{\gD^\op}$ 
    such that each $F(i)$ is level-wise equivalent to some $\gD^{n_i}$,
    the colimit is $\colim_{i \in I} F(i) \cong Z_\cd$,
    and the map from the homotopy colimit
    $\hocolim_{i \in I} F(i) \to Z_\cd$ is a weak equivalence.
    
    Using the map $p:Z_\cd \to T(B)_\cd$ we can lift $F$
    to a diagram $F:I \to \mcS^{\gD^\op}_{/T(B)}$.
    The homotopy colimit of this diagram is 
    $\hocolim_{i \in I} (F(i) \to T(B)_\cd) \simeq (Z_\cd \to T(B)_\cd)$
    since the forgetful functor $\mcS^{\gD^\op}_{/T(B)} \to \mcS^{\gD^\op}$
    commutes with colimits and preserves weak equivalences.

    \noindent
    Claim (iii):
    The geometric realization functor $\|\blank\|:\mcS^{\gD^\op} \to \mcS$
    commutes with homotopy colimits and so does the forgetful 
    functor $\mcS^{\gD^\op}_{/T(B)} \to \mcS^{\gD^\op}$.
    We need to show that the functor
    \[
        \mcS^{\gD^\op}_{/T(B)} \to \mcS^{\gD^\op}, \qquad
        (Z_\cd \to T(B)_\cd) \mapsto (T(A)_\cd \times_{T(B)_\cd} Z_\cd)
    \]
    commutes with homotopy colimits. 
    This functor preserves level-wise weak equivalences
    (because we assumed that $A \to B$ is a fibration)
    so it will suffice to show that it commutes with strict colimits.
    Strict colimits of bisimplicial sets (over $T(B)_\cd$) can be computed point-wise,
    so the claim follows from the fact that in the category of sets 
    the functor $(U \times_V \blank): \Set_{/V} \to \Set$
    commutes with colimits.
    (This functor has a right-adjoint given by $W \mapsto \coprod_{v \in V} \Map(U \times_V \{v\}, W)$.)
\end{proof}

\begin{rem}
    Note it is crucial for theorem \ref{thm:base-change} and 
    corollary \ref{cor:groupoid-basechange} that $X_\cd$ and $Y_\cd$
    are simplicial spaces. 
    The analogous statements for semi-simplicial spaces are generally not true.
    Consider for example the semi-simplicial space $Y_\cd$
    with $Y_0 = *$ and $Y_i = \emptyset$ for $i>0$.
    The semi-simplicial map $Y_\cd \amalg Y_\cd \to Y_\cd$ is a base-change
    and it is surjective on $Y_0$, but on geometric realizations it is
    $* \amalg * \to *$.
    
    To see how this is used in the proof note that in the above counter-example
    $Y_\cd = \gD_s^0$ is the semisimplicial $0$-simplex,
    i.e.\ the representable object on $[0] \in \gD_{\rm inj}^\op$ in $\mcS^{\gD_{\rm inj}^\op}$.
    This means that already the analogue of claim (i), where we show the theorem
    for the case that $Z_\cd$ is representable, fails.
    Of course one could still prove claim (i) for $\gD^n$
    thought of as a semisimplicial space, but the $\gD^n$ do not generate
    $\mcS^{\gD_{\rm inj}^\op}$ under homotopy-colimits.
\end{rem}

\begin{rem}
    The theorem also holds with the fat geometric realizations $\|X_\cd\|$
    replaced by the standard one $|X_\cd|$, as long as both simplicial spaces
    $X_\cd$ and $Y_\cd$ have the property that the quotient maps $\|X_\cd\| \to |X_\cd|$
    and $\|Y_\cd\| \to |Y_\cd|$ are weak equivalences.
    This is the case if $X_\cd$ and $Y_\cd$ are \emph{good} \cite[Proposition A.1.(iv)]{Seg74},
    which is the case if they are the level-wise realization of a bi-simplicial set.
    This will automatically be the case in \ref{cor:groupoid-basechange} where the simplicial spaces are obtained as the level-wise classifying space of a simplicial groupoid.
\end{rem}

We now state a special case of the theorem where $X_\cd$ and $Y_\cd$ are both
obtained as level-wise classifying spaces of simplicial groupoids.
To do so we quickly recall the following definition:
\begin{defn}\label{defn:base-change}
    A functor $P:\mcE \to \mcB$ is an \emph{iso-fibration} if for any object
    $e \in \mcE$ and isomorphism $f:P(e) \cong b$ there is an isomorphism 
    $g:e \to \ol{b}$ in $\mcE$ such that $P(\ol{b}) = b$ and $P(g) = f$.
\end{defn}

\begin{cor}\label{cor:groupoid-basechange}
    Let $F_\cd:\mc{X}_\cd \to \mc{Y}_\cd$ be a functor between 
    two simplicial groupoids such that 
    \begin{itemize}
        \item $F_0:\mc{X}_0 \to \mc{Y}_0$ is essentially surjective, 
        \item $(\partial_n^{\mc{X}}, F_n): 
        \mc{X}_n \to (\mc{X}_0)^{n+1} \times_{(\mc{Y}_0)^{n+1}} \mc{Y}_n$
        is an equivalence of categories (where the pullback is strict), and 
        \item For all $n$ at least one of the functors
        $\partial_n^{\mc{Y}}:\mc{Y}_n \to (\mc{Y}_0)^{n+1}$
        and $F_n:\mc{X}_n \to \mc{Y}_n$ is an iso-fibration.
    \end{itemize}
    Then $|BF|: |B\mc{X}_\cd| \to |B\mc{Y}_\cd|$ is an equivalence.
\end{cor}
\begin{proof}
    The first condition implies the surjectivity of $\pi_0(BF_0): B\mc{X}_0 \to B\mc{Y}_0$. 
    The second condition implies that $BF_\cd$ is a homotopy base-change,
    if we can show that the map
    \[
        B \left((\mc{X}_0)^{n+1} \times_{(\mc{Y}_0)^{n+1}} \mc{Y}_n \right)
        \longrightarrow
        (B\mc{X}_0)^{n+1} \times_{(B\mc{Y}_0)^{n+1}}^h B\mc{Y}_n
    \]
    is an equivalence. By the third condition at least one of the functors 
    involved in the pullback is an iso-fibration. 
    Taking the pullback of classifying spaces along an iso-fibration of groupoids
    always yields the homotopy pullback -- this can be shown by elementary means
    or by noting that iso-fibrations of groupoids are both Cartesian 
    and coCartesian, so that \cite[Theorem 2.3]{Stm17} may be applied.
\end{proof}

To illustrate the use of the base-change theorem we quickly prove that
the realization of the Rezk nerve is always equivalent to 
the classifying space of the category.
This will be useful later. First recall:
\begin{defn}\label{defn:relative-nerve}
    For any category $\mcC$ its \emph{Rezk nerve}%
    \footnote{
        This is what Rezk calls the classifying diagram of the category, 
        see \cite[Section 3.5]{Rez01}.
    }
    is the simplicial groupoid
    $N_\cd^{\rm R}\mcC$ defined for each $n$ as the groupoid of functors
    $[n] \to \mcC$ with morphisms being natural isomorphisms.
    This contains the usual nerve as a simplicial subgroupoid 
    $N_\cd \mcC \subset N_\cd^{\rm R} \mcC$ with only identity morphisms.
\end{defn}

\begin{lem}\label{lem:relative-nerve}
    The inclusion $N_\cd \mcC \subset N_\cd^{\rm R} \mcC$ always induces 
    an equivalence:
    \[
        B\mcC \simeq \|N_\cd\mcC\| \simeq \|B(N_\cd^{\rm R} \mcC)\|.
    \]
\end{lem}
\begin{proof}
    We need to verify the conditions of the base-change theorem for simplicial groupoids 
    \ref{cor:groupoid-basechange}. The functor $N_0\mcC \to N_0^{\rm R}\mcC$ 
    is a bijection on objects and in particular essentially surjective. 
    For any $n$ the functor
    \[
        N_n\mcC \to (N_0\mcC)^{n+1} \times_{(N_0^{\rm R} \mcC)^{n+1}} N_n^{\rm R} \mcC
    \]
    is an isomorphism of groupoids, as both sides can be identified with the 
    subcategory of $N_n^{\rm R} \mcC$ containing only identity morphisms.
    Finally, we claim that $\partial_n:N_n^{\rm R}\mcC \to (N_0^{\rm R}\mcC)^{n+1}$
    is an iso-fibration for all $n$. Indeed, given a functor $X:[n] \to \mcC$
    and isomorphisms $\ga_i:X(i) \cong X'(i)$ there is a (unique) functor $X':[n] \to \mcC$
    extending the $X'(i)$ such that $\ga:X \cong X'$ is a natural isomorphism.
    (This functor $X'$ is given by $X'(i \le j) = \ga_j \circ X(i \le j) \circ \ga_i^{-1}$.)
\end{proof}

\subsection{The simplicial groupoid D(C)}

Now that all the tools are in place we define the simplicial groupoid 
$D_\cd(\mcC)$ that serves as a more convenient replacement for $N\ICsp(\mcC)$.%
\footnote{
    The name $D$ is based on the topological poset $D_\gt^\mcC$ 
    from \cite[Definition 2.13]{GRW14} that is used as a convenient replacement
    of the topological cobordism category $\mcC_d$ in their work.
}
One can think of $D_\cd(\mcC)$ as a variant of the \emph{Rezk nerve} construction 
applied to the $(2,1)$-category $\ICsp(\mcC)$.
We will see that the realization of $D_\cd(\mcC)$
is equivalent to the classifying space of $\ICsp(\mcC)$.
The advantage of $D_\cd(\mcC)$ is that for each $n$ the groupoid $D_n(\mcC)$ 
inherits a PCM structure from $\mcC$ 
(which would not be the case for $N_\cd\ICsp(\mcC)$)
and so we can think of $BD_\cd(\mcC)$ as a simplicial $\gC$-space.
First, we recall the definition of the nerve and the classifying space of a $(2,1)$-category following \cite[\S1.4]{Til99}.
\begin{defn}\label{defn:2nerve}
    For a $(2,1)$-category $\mcA$ we define a simplicial groupoid $N_\cd \mcA$ as follows.
    Objects of $N_n \mcA$ are composable $n$-tuples of morphisms in $\mcA$.
    A morphism
    \[
        \ga: (a_0 \xrightarrow{f_1} a_1 \xrightarrow{f_2} \dots \to a_n)
        \to
        (a_0' \xrightarrow{g_1} a_1' \xrightarrow{g_2} \dots \to a_n')
    \]
    can only exist if $a_i = a_i'$ for all $i$ and in this case it consists of an $n$-tuple of $2$-morphisms $(\ga_j: f_j \cong g_j)_{j=1,\dots,n}$.
    The $i$th face map forgets $a_i$ and composes $f_{i+1} \circ f_i$ (unless $i \in \{0,n\}$) and the $i$th degeneracy map duplicates $a_i$ and introduces an identity morphism.
    We define the classifying space $B\mcA$ of $\mcA$ as the realisation of the simplicial topological space defined by sending $[n]$ to $B(N_n\mcA)$, the classifying space of the groupoid $N_n\mcA$.
    (Here Tillmann instead realises the diagonal of the bisimplicial set $N_\cd (N_\cd \mcA)$, but this is homotopy equivalent to geometrically realising both simplicial directions.)
\end{defn}

\begin{obs}\label{obs:nerves-of-2-categories}
    There are multiple ways of assigning a ``classifying space'' to a $(2,1)$-category $\mcA$.
    While the definition above is in line with \cite{Til99}, a more modern approach would be to take its Duskin nerve $N^{\rm D}(\mcA)$ \cite{Duskin}, which is a quasi-category \cite[009P]{Kerodon}, and consider its geometric realization $|N^{\rm D}(\mcA)|$.
    We now show that these are equivalent constructions, and in fact they come from two equivalent ways of thinking of a $(2,1)$-category as an $\infty$-category.
    Let $\Cat(\Gpd)$ and $\Cat(\Kan)$ denote the $1$-categories of categories enriched in groupoids and Kan complexes, respectively, and let $\Gpd_\infty = \mcS$ denote the $\infty$-category of $\infty$-groupoids or spaces.
    Consider the diagram of $\infty$-categories
\[\begin{tikzcd}
	& {\Cat(\Gpd)} & {\Fun(\Delta^\op, \Gpd)} & {\Fun(\Delta^\op, \Gpd_\infty)} & \\
	& {\Cat(\Kan)} & {\Fun(\Delta^\op, \Kan)} & {\Fun(\Delta^\op, \mcS)} \\
	{\mathrm{qCat}} &&& {\Cat_\infty} & \mcS
	\arrow["{N^{\rm gpd}}", hook, from=1-2, to=1-3]
	\arrow["{\Cat(N)}", hook, from=1-2, to=2-2]
	\arrow["{N^{\rm Duskin}}"', from=1-2, to=3-1]
	\arrow[from=1-3, to=1-4]
	\arrow["{\Fun(\Delta^\op, N)}", hook, from=1-3, to=2-3]
	\arrow[equals, from=1-4, to=2-4]
	\arrow["{N^{\rm sSet}}"', hook, from=2-2, to=2-3]
	\arrow["{N^{\rm coh}}", from=2-2, to=3-1]
	\arrow[from=2-3, to=2-4]
	\arrow["{\mathrm{ac}}", from=2-4, to=3-4]
	\arrow["\colim", from=2-4, to=3-5]
	\arrow[from=3-1, to=3-4]
	\arrow["{|-|}"', from=3-4, to=3-5]
\end{tikzcd}\]
    where $N^{\rm gpd}$ denotes the groupoid-enriched nerve described in definition \ref{defn:2nerve}, $N\colon \Gpd \to \Kan$ is the usual nerve, and $N^{\rm sSet}$ denotes the simplicially enriched nerve.
    The functor $\mathrm{ac}\colon \Cat_\infty \to \mcS$ is the associated category functor that is left adjoint to the Rezk nerve, see e.g.~\cite{Hebestreit2025}, and the triangle it participates in commutes because the triangle of right adjoints commutes.
    The left triangle commutes by \cite[00KY]{Kerodon} and the two squares on the top commute by construction.
    It remains to check that the big trapezoid commutes, i.e.~that the two ways of obtaining an $\infty$-category from a Kan enriched category agree.
    Recall that the two functors
    \[
        \mathrm{qCat} \xleftarrow[\qquad]{N^{\rm coh}} \Cat(\Kan) \xrightarrow[\qquad]{N^{\rm sSet}} \Fun(\Delta^\op, \Kan)
    \]
    are restrictions of right adjoints in a Quillen equivalence to the respective fibrant objects \cite{Bergner2009}.
    These functors preserve weak equivalences and after inverting the weak equivalences all the functors in the trapezoid become equivalences of $\infty$-categories.
    Therefore, the trapezoid commutes up to a self-equivalence of $\Cat_\infty$.
    By \cite[Theorem 6.3]{Ton2005} any such self-equivalence is either the identity or $(-)^{\rm op}$ and we can exclude the latter case by noting that the functors agree on the $1$-category $(\cd \to \cd \leftarrow \cd)$.
    This shows that the whole diagram commutes and particular the two proposed notions of classifying space of a $(2,1)$-category agree: 
    going around the diagram via the top right yields definition \ref{defn:2nerve}, whereas going around the left and bottom yields the realization of the Duskin nerve.
\end{obs}

The simplicial groupoid $D_\cd(\mcC)$ is now defined as a variant of $N_\cd \ICsp(\mcC)$.
This is inspired by the Rezk nerve (definition \ref{defn:relative-nerve}), but is not quite the same as the morphisms in $D_0(\mcC)$ will be bijections of labelled sets, which are not literally the same as isomorphisms in $\ICsp(\mcC)$.

\begin{defn}\label{defn:Dnc}
    The simplicial groupoid $D_\cd(\mcC)$ is defined as follows.
    An object of $D_n$ consists of $n$ composable morphisms in $\ICsp(\mcC)$,
    i.e.\ a datum $((A_\cd,c_\cd), (\UL{X}_\cd, o_\cd))$ where
    $A_0, \dots, A_n \subset \gO$ are finite sets with 
    labellings by connected objects $c_i:A_i \to \Obj^\con(\mcC)$ and 
    cospans $(f_i:A_{i-1} \to \UL{X}_i \leftarrow A_i:g_i)$ with labellings 
    by connected morphisms $o_i:\UL{X}_i \to \Mor^\con(\mcC)$.
    A morphism of $n$-simplices 
    $((A_\cd,c_\cd), (\UL{X}_\cd, o_\cd)) \to ((A_\cd',c_\cd'), (\UL{X}_\cd', o_\cd'))$
    is a family of bijections $\ga_i:A_i \cong A_i'$ and 
    $\gp_i:\UL{X}_i \cong \UL{X}_i'$ satisfying:
    \[
        c_i \circ \ga_i = c_i', \quad
        o_i \circ \gp_i = o_i', \quad
        \gp_i \circ f_i = f_i' \circ \ga_{i-1}  \qand 
        \gp_i \circ g_i = g_i' \circ \ga_i.
    \]
    The face operator $d_i$ is defined by composing the $i$th and $(i+1)$st morphism
    according to the composition in $\ICsp(\mcC)$.
    The degeneracy operator $s_i$ is defined by inserting identity morphisms.
\end{defn}

\begin{defn}
    We further construct a PCM structure on each $D_n(\mcC)$ as follows.
    Two objects $((A_\cd,c_\cd),(\UL{X}_\cd, o_\cd))$ and $((A_\cd',c_\cd'),(\UL{X}_\cd', o_\cd'))$ are disjoint if we have $A_i \cap A_i' = \emptyset$ and $\UL{X}_j \bot \UL{X}_j$ for all $i$ and $j$, similarly to the PCM structure on $\ICsp$ in definition \ref{defn:Csp-PCM}.
    Then $\amalg$ is defined by taking unions of the disjoint sets and combining the labellings.
    All the face and degeneracy maps in $D_\cd(\mcC)$ preserve disjointness and the union of disjoint tuples.
\end{defn}

\begin{rem}
    In what follows we will often simply write $D_\cd$ for $D_\cd(\mcC)$ when 
    the labelled cospan category $(\mcC \to \Csp)$ is clear from the context.
    In a similar vein we often denote objects of $D_n(\mcC)$ as tuples 
    $(A_\cd,\UL{X}_\cd)$ where the labellings $c_i$ and $o_j$ are left implicit. 
    For $0 \le k < l \le n$ we also use the notation
    \[
        \UL{X}_{k \to l}  := \UL{X}_{k+1} \cup_{A_k} \dots \cup_{A_{l-1}} \UL{X}_l
    \]
    for the composite cospan and we write 
    $o_{k \to l}:\UL{X}_{k\to l} \to \Mor^\con(\mcC)$
    for the labelling induced by composition.
    Note that in this notation $\UL{X}_k = \UL{X}_{k-1\to k}$.
\end{rem}

\begin{rem}
    Note that there is an inclusion of simplicial groupoids:
    \[
        N_\cd \ICsp(\mcC) \longrightarrow D_\cd   
    \]
    which in level $n$ identifies $N_n\ICsp(\mcC)$ with the subcategory of $D_n$
    that contains all objects and all those morphisms $(\ga_\cd, \gp_\cd)$
    where $\ga_i = \id_{A_i}$ holds for all $i$.
\end{rem}

In analogy with $N_\cd^\nc \mcC \subset N_\cd \mcC$ in definition \ref{defn:Nnc} we let 
$D_\cd^\nc \subset D_\cd$ be the simplicial subgroupoid  where the total composite is connected.
To be precise, each $D_n^\nc \subset D_n$ is the full subgroupoid on those 
$(A_\cd, \UL{X}_\cd)$
such that the composite cospan 
$(A_0 \to \UL{X}_{0 \to n} \leftarrow A_n)$
is reduced.

\begin{lem}\label{lem:NICsp=D}
    The inclusion of simplicial groupoids $N_\cd \ICsp(\mcC) \to D_\cd(\mcC)$
    induces equivalences:
    \[
    \begin{tikzcd}
        {|B N_\cd^\nc \ICsp(\mcC)|} \ar[r, "\simeq"] \ar[d] & 
        {|BD_\cd^\nc(\mcC)|} \ar[d] \\
        {|B N_\cd \ICsp(\mcC)|} \ar[r, "\simeq"] & 
        {|BD_\cd(\mcC)|} 
    \end{tikzcd}
    \]
\end{lem}
\begin{proof}
    Both equivalences can be checked by applying the base-change theorem 
    for simplicial groupoids \ref{cor:groupoid-basechange}.
    We will verify the conditions for the bottom equivalence, the top map is similar.
    Both are entirely analogous to lemma \ref{lem:relative-nerve}.
    
    To begin, note that $N_n\ICsp(\mcC) \to D_n$ is an inclusion of groupoids
    that is a bijection on objects. So $N_0\ICsp(\mcC) \to D_0$ 
    is certainly essentially surjective.
    Moreover, for all $n$ both sides of 
    \[
        N_n \ICsp(\mcC) \longrightarrow 
        (N_0 \ICsp(\mcC))^{n+1} \times_{(D_0)^{n+1}} D_n
    \]
    can be thought of the subcategory of $D_n$ that contains all objects 
    and all those morphisms $(\ga_\cd,\gp_\cd)$ where $\ga_i = \id_{A_i}$ for all $i$.
    So this functor is an isomorphism.
    It remains to check that $\partial_D^n:D_n \to (D_0)^{n+1}$ is an isofibration.
    This means that given some $n$-simplex $(A_\cd, \UL{X}_\cd, c_\cd, o_\cd)$, 
    an $(n+1)$ tuple $(B_\cd,c_\cd')$, and $(n+1)$ bijections $\ga_i: A_i \cong B_i$
    such that $c_i' \circ \ga_i = c_i$,
    we need to find an $n$-simplex
    $(B_\cd, \UL{X}_\cd', c_\cd', o_\cd')$ that is isomorphic to the original $n$-simplex
    via some isomorphism of the form $(\ga_\cd, \gp_\cd)$.
    This can be achieved by simply setting $\UL{X}_i' := \UL{X}_i$ and 
    then using the $\ga_i$ to define the structure maps of the cospans as 
    $B_{i-1} \cong A_{i-1} \to \UL{X}_i \leftarrow A_i \cong B_i$.
    The labelling $o_\cd$ stays the same.
\end{proof}

In the case of the no closed component space we can also compare 
this to the much simpler simplicial set $N_\cd^\nc\mcC \subset N_\cd\mcC$.
\begin{lem}\label{lem:NncICsp=NncC}
    The canonical projection $N_\cd^\nc\ICsp(\mcC) \to N_\cd^\nc \mcC$ 
    is an equivalence on realizations. 
\end{lem}
\begin{proof}
    The functor $\ot_\mcC: \ICsp(\mcC) \to \mcC$ is essentially surjective, but 
    not necessarily surjective on the nose. Let $\mcC' \subset \mcC$ 
    be the full subcategory on those objects that are in the image,
    i.e.\ those that are literally $\ot$-products of connected objects.
    The inclusion $N_\cd^\nc \mcC' \subset N_\cd^\nc \mcC$ is an equivalence
    on realizations by lemma \ref{lem:Nnc-invariant}.
    
    We will show that $N_\cd^\nc\ICsp(\mcC) \to N_\cd^\nc \mcC'$ 
    satisfies the condition of corollary \ref{cor:groupoid-basechange}.
    By construction $N_0^\nc\ICsp(\mcC) \to N_0^\nc \mcC'$ is surjective.
    The vertex map $\partial_\mcC':N_n^\nc \mcC' \to (N_0^\nc\mcC')^{n+1}$
    is an isofibration because it is a map between groupoids that 
    only have identity morphisms.
    It remains to check that for all $n$ the map
    \[
        N_n^\nc \ICsp(\mcC) \longrightarrow 
        (N_0^\nc \ICsp(\mcC))^{n+1} \times_{N_0^\nc\mcC'} N_n\mcC'
    \]
    is an equivalence of groupoids.
    This follows from the equivalence 
    \[
    \Hom_{\ICsp}^\red((A,c_A), (B,c_B)) 
    \simeq \Hom_\mcC^\red(\ot_\mcC(A, c_A), \ot_\mcC(B,c_B))
    \]
    shown in lemma \ref{lem:hCsp(C)=C} since every morphism 
    involved in $N_n^\nc \ICsp(\mcC)$ is necessarily reduced.
\end{proof}

\subsection{The groupoid of cuts}

We now proceed to construct a simplicial groupoid $\Cut_\cd$ that represents
the complement of $D_\cd^\nc$ in $D_\cd$.
This is a generalization of the ideas of our previous work \cite[Part II]{Stb21},
with the key difference being that the morphisms involved in $\Cut_\cd$ 
are allowed to have closed components whereas the space of cuts
in \cite[Section 7]{Stb21} only used reduced morphisms.

\begin{defn}\label{defn:Cut}
    The simplicial groupoid $\Cut_\cd$ is defined in each level 
    as the full subgroupoid $\Cut_n \subset D_n$ on those objects
    $(A_\cd, \UL{X}_\cd, c_\cd, o_\cd)$ where $A_0$ and $A_n$ are both empty.
    
    There is a projection functor $R_n:D_n \to \Cut_n$ that discards 
    all elements of $A_i$ and $\UL{X}_i$ that correspond to an element of
    $\UL{X}_{0 \to n}$ 
    that lies in the image of the map $A_0 \amalg A_n \to \UL{X}_{0 \to n}$.
    
    We define the simplicial structure on $\Cut_n$ by letting 
    $\gl:[n] \to [m]$ act as
    \[
        \gl_{\Cut}^*: \Cut_m \subset D_m \xrightarrow{\ \gl^*\ } D_n 
        \xrightarrow{\ R_n\ } \Cut_n.
    \]
    We equip each $\Cut_n$ with the PCM structure restricted from the one on $D_n$ (see definition \ref{defn:Csp-PCM}).
\end{defn}

\begin{lem}\label{lem:Cut-is-simplicial}
    The above definition yields a well-defined simplicial PCM groupoid
    $\Cut_\cd$ such that the functors $R_n$ induce a morphism of
    simplicial PCM groupoids: $R_\cd: D_\cd \to \Cut_\cd$.
\end{lem}
\begin{proof}
    We need to check that the operations $\gl_\Cut^*:\Cut_m \to \Cut_n$ 
    assemble into a simplicial object. Since we already know that
    $D_\cd$ is a simplicial object it will suffice to show that
    $R_n \circ \gl^* \circ R_m = R_n \circ \gl^*$, which will also
    imply that $R_\cd:D_\cd \to \Cut_\cd$ is a simplicial functor.
    
    First consider the case where $\gl=\gd^i:[n] \to [n+1]$, the unique 
    injective morphism that does not hit $i$ for some $0<i<n+1$. 
    Then $\gl^*=d_i$ is the $i$-th face map, which composes the two adjacent morphisms.
    In the case $0<i<n$ the total composite 
    $\UL{X}_{0 \to n}$
    stays the same before and after applying $d_i$.
    Therefore deleting anything that corresponds to an element 
    of the image of 
    $A_0 \amalg A_n \to \UL{X}_{0 \to n}$
    does the same whether we do it before or after applying $d_i$.
    In formulas this means $\gl^* \circ R_{n+1} = R_n \circ \gl^*$.
    Since $R_n$ is idempotent when thought of as an endofunctor of $D_n$,
    this in particular implies $R_{n} \circ \gl^* \circ R_{n+1} = R_{n} \circ \gl^*$,
    which is what we claimed.
    The case where $\gl=\sigma^i\colon [n] \to [n-1]$ is a degeneracy operator follows similarly, as these also do not change the total composite.

    Now suppose $\gl = \delta^0$, the other remaining case of $\delta^n$ is similar.
    When applying $R_{n} \circ d_0 \circ R_{n+1}$ to an $(n+1)$-simplex
    we first delete anything in the image of 
    $A_0 \amalg A_{n+1} \to \UL{X}_{0 \to n+1}$,
    then we forget the first cospan, and then we delete anything in the image of 
    $A_1 \amalg A_{n+1} \to \UL{X}_{1 \to n+1}$.
    It follows from inspection that anything we delete in the first step
    because it lies in the image of $A_0$ will either be in the
    first cospan or lie in the image of $A_1$ later, so it will be deleted anyway.
    This implies that the first $R_{n+1}$ was redundant and so we have
    $R_{n} \circ d_0 \circ R_{n+1} = R_{n} \circ d_0$.
    
    Since $\gD^\op$ is generated by the $\gd^i$ and $\sigma^j$ 
    it follows that $\Cut_\cd$ is a well-defined simplicial object
    in the category of groupoids and that $R_\cd:D_\cd \to \Cut_\cd$
    is simplicial.
    For the partial commutative monoidal structure
    we simply note that $\Cut_n \subset D_n$ is closed under disjoint
    union and isomorphism, so the structure can be restricted.
    Since we defined $R_n$ in terms of components it commutes with 
    disjoint unions. 
    (As pointed out in remark \ref{rem:PCM-functor-property} it is a property, 
    and not a structure, 
    for a functor between two partial commutative monoidal categories 
    to preserve the structure.)
    It follows that $R_\cd$ is a morphism of simplicial PCM groupoids.
\end{proof}

We have constructed $\Cut$ is such a way that every object $W \in D_n$
canonically decomposes as a disjoint union $W = W^\nc \amalg W^\Cut$
where $W^\nc \in D_n^\nc$ and $W^\Cut \in \Cut_n$. 
This means that we can apply the fiber sequence criterion \ref{thm:fiber-sequence}
to show:
\begin{lem}\label{lem:cut-fiber-sequence}
    Assume that $B\mcC$ is group-like.
    Then the inclusion $I:D_\cd^\nc \subset D_\cd$ and the simplicial functor $R_\cd$
    yield a fiber sequence of infinite loop spaces:
    \[
        |BD_\cd^\nc| \longrightarrow |BD_\cd| \xrightarrow{\ R\ } |B\Cut_\cd|.
    \]
\end{lem}
\begin{proof}
    To check the criteria of theorem \ref{thm:fiber-sequence} we need to
    find for each $n$ a PCM-functor $S_n:D_n \to D_n^\nc$ such that the 
    functors $S_n \circ I: D_n^\nc \to D_n^\nc$ and 
    $(S_n, R_n):D_n \to D_n^\nc \times \Cut_n$ are equivalences of groupoids.

    In analogy to the definition of $R_n$ we let $S_n(W)$ be defined 
    by deleting all components of $A_i$ or $\UL{X_i}$ that 
    \emph{do not} lie in the image of 
    $A_0 \amalg A_n \to \UL{X}_{0 \to n}$.
    Note that this functor does not satisfy the conditions we checked 
    in lemma \ref{lem:Cut-is-simplicial} and so $S_\cd$ does not
    define a semi-simplicial map $D_\cd \to D_\cd^\nc$.
    This is not a problem since theorem \ref{thm:fiber-sequence}
    does not require the sections $S_n$ to be related in any way.
    The definition of $S_n$ plays well with disjoint unions 
    and hence defines a PCM-functor.
    
    The composition $S_n \circ I$ is in fact the identity functor $\Id_{D_n^\nc}$
    by construction, so we only need to check that $(S_n, R_n)$ is an equivalence.
    We begin with essential surjectivity.
    The image of $(S_n, R_n)$ in $D_n^\nc \times \Cut_n$ consists 
    precisely of those tuples $(\UL{X}_\cd, \UL{Y}_\cd)$ 
    such that $\UL{X}_i \bot \UL{Y}_i$ for all $i$.
    By standard arguments this disjointness may always be achieved,
    up to isomorphism, and so $(S_n, R_n)$ is essentially surjective.
    We also need to check fully faithfulness. 
    Let $(A_\cd, \UL{X}_\cd), (C_\cd, \UL{Z}_\cd) \in D_n$ be any two objects.
    (We suppress the data of the labellings $c_{A_i}$ and $o_{\UL{X}_i}$.)
    Then we need to show that an isomorphism 
    $(\ga_\cd,\gp_\cd):(A_\cd, \UL{X}_\cd) \cong (C_\cd, \UL{Z}_\cd)$
    is exactly given by a tuple of isomorphisms
    $S_n(A_\cd, \UL{X}_\cd) \cong S_n(C_\cd, \UL{Z}_\cd)$ and
    $R_n(A_\cd, \UL{X}_\cd) \cong R_n(C_\cd, \UL{Z}_\cd)$.
    But note that any such isomorphism induces compatible isomorphisms:
    \[
        \begin{tikzcd}
            A_0 \amalg A_n \ar[d, "\cong", "\ga_0 \amalg \ga_n"'] \ar[r] &
            \UL{X}_{0 \to n} =  
            {\UL{X}_1 \cup_{A_1} \dots \cup_{A_{n-1}} \UL{X}_n} \ar[d, "\cong"]\\
            C_0 \amalg C_n \ar[r] &
            \UL{Z}_{0 \to n} = 
            {\UL{Z}_1 \cup_{C_1} \dots \cup_{C_{n-1}} \UL{Z}_n} .
        \end{tikzcd}
    \]
    So the isomorphism $(\ga_\cd,\gp_\cd)$ preserves the subsets 
    picked out by $S_n$ and $R_n$, respectively, and it necessarily 
    restricts to an isomorphism on both of those subsets.
    This shows that $(S_n, R_n)$ is an equivalence of categories
    and so we may apply the fiber sequence criterion \ref{thm:fiber-sequence}.
    
    All that remains to complete the proof is to check that all three spaces are group-like.
    The groupoid $\Cut_0$ is trivial be definition and hence $|B\Cut_\cd|$ is connected
    and in particular group-like.
    For the middle space we observe that $D_0 \simeq N_0^{\rm R}\mcC$
    and $D_1 \to N_1^{\rm R}\mcC$ is essentially surjective. 
    Hence $\pi_0|BD_\cd| \cong \pi_0B\mcC$, which we assumed to be group-like.
    In the case of $D_\cd^\nc$ it is still true that $D_0^\nc \simeq N_0^{\rm R}\mcC$
    and even though $D_1^\nc \to N_1^{\rm R}\mcC$ is no longer essentially surjective,
    it is still true that for any morphism $W:M \to N$ in $\mcC$ the $1$-cell
    $S(W) \in D_1^\nc$ connects $M$ and $N$. Hence $\pi_0|BD_\cd^\nc| \cong \pi_0|B\mcC|$
    is group-like.
\end{proof}

\subsection{Decomposing the space of cuts}
So far we have seen that $|BD_\cd|$ fits into a fiber sequence 
between $|BD_\cd^\nc|$ and $|B\Cut_\cd|$. 
The purpose of this section is to show that $|B\Cut_\cd|$ is a free infinite loop space.
We will do so by first showing that each of the $\Cut_n$ is a free symmetric monoidal
groupoid on the connected cuts.

\begin{defn}
    Let $\Cut_\cd^\con \subset \Cut_\cd$ denote the simplicial subgroupoid defined 
    in each level as the full subgroupoid on those 
    $(A_\cd, W_\cd, c_\cd, o_\cd) \in \Cut_n$ where $\UL{X}_{0 \to n}$
    is connected or empty.
\end{defn}

    We need to allow $\UL{X}_{0 \to n}$ to be empty in order for $\Cut_\cd^\con$
    to be closed under those face maps that delete closed components. 
    We take $\emptyset \in \Cut_0^\con$ as the base-point.
    We will prove the following.

\begin{lem}\label{lem:Cut-free-on-con}
    The inclusion $\Cut_\cd^\con \subset \Cut_\cd$ induces an equivalence of $\gC$-spaces
    $ Q(|B\Cut_\cd^\con|) \simeq |B\Cut_\cd|$.
\end{lem}

We recall an explicit construction of the free $\gC$-category based on Segal's work \cite{Seg74}.
\begin{defn}
    Let $\mcC$ be a category. Then the category $\gS_\mcC^{\gle{n}}$ has as
    objects triples $(T, \ga, x)$ where $T$ is a finite set, 
    $\ga:T \to \{1,\dots,n\}$ a map, and $x:T \to \Obj(\mcC)$ an assignment
    of objects that we denote by $t \mapsto x_t$.
    Morphisms $(T, \ga, x) \to (S, \gb, y)$ are pairs of a bijection $\gp:T \cong S$
    satisfying $\gb \circ \gp = \ga$ together
    with a collection of morphisms $f_t:x_t \to y_{\gp(t)}$ indexed by $t \in T$.
    The composition of morphisms is given by 
    $(\psi,g) \circ (\gp,f) = (\psi \circ \gp, h)$ where $h_t = g_{\gp(t)} \circ f_t$.
    For a map $\gl:\gle{n} \to \gle{m}$ in $\gC^\op$ we define a functor
    $\gl_*:\gS_\mcC^{\gle{n}} \to \gS_\mcC^{\gle{m}}$ by sending $(T, \ga, x)$
    to $(T', \gl \circ \ga, x_{|T'})$ where $T'=\{ t \in T \;|\; \gl(\ga(t)) \neq *\}$.
    Moreover, we let 
    \[
        \nu_\mcC: \mcC_+ = \mcC \amalg * \to \Sigma_\mcC^{\gle{1}}
    \]
    denote the functor defined by sending $c \in \mcC$ to $(\{1\}, \id_{\{1\}}, c)$ and the base point to $(\emptyset, \emptyset \to \{1\}, \emptyset)$.
\end{defn}

This construction is functorial in $\mcC$ and in fact it is even functorial in $\mcC_+$, which will allow us to discard certain components of the category by sending them to the base point.
To make this precise, let $\Cat^+$ denote the category whose objects are tuples $(\mcA, a_0)$ of categories $\mcA$ with an object $a_0$ such that $\mcA = \mcA' \amalg \{a_0\}$ for $\mcA' \subset \mcA$ the full subcategory on all objects except $a_0$,
and where morphisms are base-point preserving functors.

\begin{lem}\label{lem:plus-functorial}
    The construction of $\gS_\mcC$ defines a functor
    \begin{align*}
        \Cat^+ & \longrightarrow \gC\text{-}\Cat \\
        (\mcA,a_0) & \longmapsto \gS_{\mcA'}
    \end{align*}
    and $\nu_{\mcA'}: \mcA \cong (\mcA')_+ \to \gS_{\mcA'}^{\gle{1}}$ defines a natural transformation from the identity.
\end{lem}
\begin{proof}
    Let $\mcC$ and $\mcD$ be categories and $F: \mcC_+ \to \mcD_+$ a base-point preserving functor.
    We define $\gS_F: \gS_\mcC \to \gS_\mcD$ by sending $(T,\ga, x)$ to $(T',\ga_{|T'}, F \circ (x_{|T'}))$ where $T' \subset T$ is the subset of those $t \in T$ for which $F(x_t) \in \mcC \subset \mcC_+$, i.e.~those whose label is not sent to the base point by $F$.
    (Note that $F \circ (x_{|T'})$ is indeed valued in $\mcD \subset \mcD_+$ by construction.)
    Each $\gS_F$ is a functor and the construction is functorial in the sense of $\gS_G \circ \gS_F = \gS_{G \circ F}$.
    The naturality of $\nu_\mcC$ follows from the construction: if an object is sent to the base point by $F$, then both $\gS_F^{\gle{1}} \circ \nu_\mcC$ and $\nu_\mcD \circ F$ send it to the empty labelled set.
\end{proof}

Note that the $\gC$-space $\gle{n} \mapsto B(\gS_\mcC^{\gle{n}})$ is isomorphic to the $\gC$-space that is denoted by $\gS_X$ for $X = B\mcA$ in \cite[p. 299]{Seg74}.
Let us denote by $( \blank )^{\rm gp}$ denote the construction that sends a special $\gC$-space to its group completion, which we assume to come with the data of an infinite loop space.
In particular, there is a canonical infinite loop space map from the free infinite loop space $Q(Y^{\gle{1}}) = \colim \gO^N\gS^N Y^{\gle{1}}$ to the group completion $Y^{\rm gp}$ for any special $\gC$-space $Y^{\gle{\cd}}$.
Segal shows that the spectrum associated to $\gS_X$ is the suspension spectrum of $X_+ = X \amalg \{*\}$, so that $Q(X_+) \simeq (\gS_X)^{\rm gp}$.
For every category $\mcC$ the functor $\nu_\mcC: \mcC_+ \to \gS_\mcC^{\gle{1}}$ induces a map 
    \[
        Q(B\mcC_+) \xrightarrow{Q(B\nu_\mcC)} Q(B\gS_{\mcC}^{\gle{1}}) \to (B\gS_\mcC)^{\rm gp}.
    \]
This composite is an equivalence because it is homotopic to the map Segal describes, in the case of $X = \mcC$.
To see this, it suffices to compare their restrictions to $B\mcC$ (where they are identical) and to the base point $* \subset B\mcC_+$ (which is sent to the base point component).

\begin{cor}\label{cor:level-wise-free-sm-cat}
    For every simplicial object $\mcA_\cd: \Delta^\op \to \Cat^+$ there is an equivalence of $\gC$-spaces
    \[
        \nu: Q(|B\mcA_\cd|) \simeq (|B\gS_{\mcA_\cd'}|)^{\rm gp},
    \]
    and if $\mcA_0 = *$, then this is moreover equivalent to $|B\gS_{\mcA_\cd'}|$ without the group completion.
\end{cor}
\begin{proof}
    For all $[n] \in \Delta^\op$ we have an equivalence
    $Q(B\mcA_n) \simeq (B\gS_{\mcA'_n})^{\rm gp}$
    and these are natural in $[n]$ by lemma \ref{lem:plus-functorial}.
    After taking the geometric realisation on both sides, we can commute the geometric realisation inside of $Q(-)$ and $(-)^{\rm gp}$ by \cite[Proposition 12.1 and Theorem 12.3]{May72}.
    See also \cite{BBPTY17} for a more detailed account for the case of the group-completion functor.
    This gives the indicated equivalence.
    Under the additional assumption that $\mcA_0 = *$ we have that $B\gS_{\mcA_0'} = B\gS_{\emptyset}$ is contractible.
    Hence $|B\gS_{\mcA_\cd'}|$ is connected and in particular the canonical map to its group-completion is an equivalence.
\end{proof}

We can now prove lemma \ref{lem:Cut-free-on-con}
by comparing $\Cut_n$ to the free $\gC$-category on $\Cut_n^\con$.
\begin{proof}[Proof of lemma \ref{lem:Cut-free-on-con}]
    The category of connected cuts decomposes as 
    $\Cut_n^\con = \{\emptyset\} \amalg \Cut_n^{\con, \neq\emptyset}$
    and the simplicial operators all preserve $\emptyset$, so that we can consider it as a simplicial object in $\Cat^+$.
    Note that $\Cut_0^\con = \{\emptyset\}$.
    From the functoriality in lemma \ref{lem:plus-functorial} we obtain a $\gC$-category $\mcD_\cd^{\gle{\cd}} := \gS_{\Cut_\cd^{\neq \emptyset}}^{\gle{\cd}}$.
    By corollary \ref{cor:level-wise-free-sm-cat} it satisfies
    \[
        Q(|B\Cut_\cd^\con|) \simeq |B\mcD_\cd|.
    \]
    It therefore suffices to construct maps of $\gC$-functors $\mcD_\cd \leftarrow \mcE_\cd \to \Cut_\cd$ that induce equivalences
    \[
        |B\mcD_\cd| \simeq |B\mcE_\cd| \simeq |B\Cut_\cd|.
    \]
    For fixed $n$ and $m$ we let $\mcE_n^{\gle{m}} \subset \mcD_n^{\gle{m}}$
    be the full subcategory on those $(T, \ga, (A_\cd^t, \UL{X}_\cd^t)_{t \in T})$
    where the $\UL{X}_i^t \in \Cut_n^{\con, \neq \emptyset}$ 
    are pairwise disjoint as $t\in T$ varies.
    One checks that this is closed under face operators,
    and also under functoriality in $\gle{m} \in \gC$.
    As usual this hits all isomorphism classes and hence the inclusion 
    $\mcE_\cd \to \mcD_\cd$ is a level-wise equivalence of simplicial $\gC$-categories.
    In particular, it realises to an equivalence $|B\mcE_\cd| \simeq |B\mcD_\cd|$.
    
    The advantage of this subcategory is that the disjointness condition 
    allows us to define a functor $\amalg:\mcE_n^{\gle{m}} \to \Cut_n^{\gle{m}} \subset (\Cut_n)^m$
    (recall that $\Cut_n^{\gle{m}}$ is defined as in lemma \ref{lem:PCM-yields-gC} as the full subcategory of $(\Cut_n)^m$ on those $m$-tuples that are pairwise disjoint)
    by 
    \[
    (T, \ga: T \to \{1,\dots,m\}, (A_\cd^t, \UL{X}_\cd^t)_{t \in T}) 
    \longmapsto \big(\coprod_{t \in \alpha^{-1}(i)} A_\cd^t, \coprod_{t \in \alpha^{-1}(i)} \UL{X}_\cd^t\big)_{i=1,\dots m}
    \]
    This is well-defined because $\Cut_n$ is a partial commutative 
    monoidal category, and moreover it is natural in $n$ and $\gle{m}$.
    To complete the proof we need to check that 
    $\amalg:\mcE_n^{\gle{1}} \to \Cut_n^{\gle{1}}$ is an equivalence of categories.
    This is indeed the case, and an inverse is given by 
    \[
    \big(A_\cd, \UL{X}_\cd\big)
    \longmapsto (T:= \UL{X}_{0 \to n}, \ga: T \to \{1\}, (A_\cd^t, \UL{X}_\cd^t)_{t \in T}) 
    \]
    where $A_\cd^t \subset A_\cd$ and $\UL{X}_\cd^t \subset \UL{X}$ are the subsets that lie over $t \in \UL{X}_{0 \to n}$.
\end{proof}

\begin{cor}\label{cor:HBCut-is-free}
    The rational homology of $|B\Cut_\cd^\con|$ injects into the 
    rational homology of $|B\Cut_\cd|$ and its image freely generates
    $H_*(|B\Cut_\cd|;\IQ)$ as a commutative algebra.
\end{cor}
\begin{proof}
    For any connected space $X$ the rational homology of $Q(X)$ is a free
    symmetric algebra on the rational homology $X$.
    The zig-zag of equivalences 
    $|B\Cut_\cd| \simeq \dots \simeq Q(|B\Cut_\cd^\con|)$
    constructed in lemma \ref{lem:Cut-free-on-con} is compatible 
    with the $\gC$-space structure and hence induces an algebra isomorphism
    on homology.
    Moreover, this zig-zag is compatible with the natural map from 
    $|B\Cut_\cd^\con|$.
\end{proof}

\subsection{From cuts to factorizations}\label{subsec:cut-to-fact}
In this section we compute the realization of the space of connected cuts in terms 
of the category of factorizations.
We will denote the set of connected endomorphism of $1_\mcC$ by:
\[
    G:= \Hom_\mcC^\con(1_\mcC, 1_\mcC).
\]

We first introduce a variant of $\Cut_\cd^\con$ where we identify isomorphic morphisms and express them as morphisms in $\mcC$ rather than $\ICsp(\mcC)$.
We will shortly see that because we are only considering connected morphisms this variant is level-wise equivalent to $\Cut_\cd^\con$.

\begin{defn}\label{defn:CutmcC}
    We let $\Cc_\cd(\mcC)$ denote the following simplicial groupoid.
    For all $n$ we let $\Cc_n(\mcC)$ be the full subgroupoid of the Rezk nerve $N_n^{\rm R}(\mcC)$ on those diagrams
    \[
        x_0 \xrightarrow{f_1} x_1 \xrightarrow{f_2} \dots \xrightarrow{f_{n-1}} x_{n-1} \xrightarrow{f_n} x_n
    \]
    where $x_0 = 1_\mcC = x_n$ and $f_n \circ \dots \circ f_1: 1_\mcC \to 1_\mcC$ is either a connected morphism or the identity.
    We moreover require%
    \footnote{
        This additional requirement does not change the homotopy type, but it ensures that in lemma \ref{lem:Cut-con=gS} we actually have a homeomorphism to the suspended cone.
    }
     that every $x_i$ that is isomorphic to $1_\mcC$ is in fact equal to $1_\mcC$.
    The degeneracy maps are defined by inserting identity morphisms and the face maps are defined by composing, except for $d_0$ and $d_n$ where we set
    \begin{align*}
        d_0(f_1,\dots,f_n) &= \begin{cases}
            (f_2, \dots, f_n) & \text{ if } f_1 = \id_{1_\mcC} \\
            (\id_{1_\mcC}, \dots, \id_{1_\mcC}) & \text{ if } f_1 \neq \id_{1_\mcC} 
        \end{cases} \\
        d_n(f_1,\dots,f_n) &= \begin{cases}
            (f_1, \dots, f_{n-1}) & \text{ if } f_{n} = \id_{1_\mcC} \\
            (\id_{1_\mcC}, \dots, \id_{1_\mcC}) & \text{ if } f_n \neq \id_{1_\mcC} .
        \end{cases}
    \end{align*}
\end{defn}

These face maps satisfy the simplicial identities.
For example, for $0 < i \le n$ we have
\[
    d_0d_1(f_1,\dots,f_n) = 
    \begin{cases}
        (f_3 ,\dots,\widehat{f_i}, \dots f_n) & \text{ if } f_1 = \id_{1_\mcC} = f_2 \\
        (\id_{1_\mcC}, \dots, \id_{1_\mcC}) & \text{ otherwise }
    \end{cases}
    = d_{0}d_0(f_1,\dots,f_n) 
\]
as $f_2 \circ f_1 = \id_{1_\mcC}$ if and only if $f_1 = \id_{1_\mcC} = f_2$. (If one of them is not the identity on $1_\mcC$ then $x_1$ or $x_2$ is not isomorphic to $1_\mcC$, and then $f_2 \circ f_1 \neq 1_\mcC$ as no object is a non-trivial retract of $1_\mcC$.)

\begin{lem}\label{lem:Cut-to-CutmcC}
    The functor $\otimes_\mcC: \ICsp(\mcC) \to \mcC$ induces a functor of simplicial groupoids $\Cut_\cd^\con \to \Cc_\cd(\mcC)$ that is a level-wise equivalence.
\end{lem}
\begin{proof}
    The functor $\otimes_\mcC$ induces a simplicial functor $D_\cd \to N_n^{\rm R}(\mcC)$ and by how we defined $\Cut_n^\con$ and $\Cc_n(\mcC)$ this restricts%
    \footnote{
        Here we in particular need to check that $\otimes_\mcC(A_i,c_i) \cong 1_\mcC$ implies $\otimes_\mcC(A_i,c_i) = 1_\mcC$, and this is indeed the case as the isomorphism can only exist if $A_i = \emptyset$.
    }
     for all $n$ to a functor 
    \[
        q_n: \Cut_n^\con \to \Cc_n(\mcC).
    \]
    The degeneracy maps and inner face maps of $\Cut_\cd^\con$ and $\Cc_\cd(\mcC)$ are defined by restricting those of $D_\cd$ and $N_n^{\rm R}(\mcC)$, respectively, and therefore $q_\cd$ respects them.
    We now check $d_0$ and the case of $d_n$ is analogous.
    In $\Cut_\cd$ (and thus also in its simplicial subgroupoid $\Cut_\cd^\con$) the first face map $d_0$ is defined on 
    \[
        ((A_0,c_0) \xrightarrow{(X_1,o_1)} \dots \xrightarrow{(X_n,o_n)} A_n)
    \]
    by discarding the first morphism and then deleting all components that intersected $A_1$ so that the result again a simplex that starts and ends at the empty set.
    Since the definition of $\Cut_n^\cd$ assumes that the composite morphism $[X_{0 \to n}, c_{0\to n}]$ is connected or empty, there are two cases:
    either $A_1$ is empty and we indeed just delete the first morphism, 
    or $A_1$ is non-empty and thus intersect the only component of $X_{0\to n}$, so that we must delete it and $d_0((A_0,c_0)\to \dots \to (A_n,c_n)) = (\emptyset = \dots = \emptyset)$ is the trivial $(n-1)$-simplex.
    This exactly agrees with how we defined the outer face maps for $\Cc_\cd(\mcC)$, so $q_n$ is a simplicial map.

    To check that $q_n$ is a level-wise equivalence we can factor it as
    \[
        \Cut_n^\con \xrightarrow{q_n'} \Cc_n(h\ICsp(\mcC)) \to \Cc_n(\mcC).
    \]
    The second functor is a level-wise equivalence because it is obtained by taking the equivalence of categories $\otimes_\mcC: h\ICsp(\mcC) \to \mcC$, applying the Rezk nerve $N_n^{\rm R}$ to it, and then restricting to certain full subcategories defined by connectedness conditions.
    (Note also that the condition that $x_i \cong 1_\mcC \Rightarrow x_i = 1_\mcC$ is trivially satisfied in $h\ICsp(\mcC)$ because no other set is isomorphic to the empty set.)
    The first functor $q_n'$ is an iso-fibration: isomorphism are bijections of the sets $A_i$ compatible with the labellings and morphisms, and we can lift those to isomorphism in $\Cut_n^\con$.
    It will hence suffice to show that the fibers of $q_n'$ are trivial.
    The fiber at some $n$-simplex $((A_0,c_0)\to \dots \to (A_n,c_n))$ is the product of the automorphism groups of the morphisms $(X_i,o_i): (A_{i-1},c_{i-1}) \to (A_i, c_i)$.
    Since we assumed that $[X_{0\to n}, c_{0\to n}]$ is a connected morphism there are two cases for each $X_i$:
    either $A_{i-1} = \emptyset = A_i$, in which case $X_i$ must be a single point,
    or the map $A_{i-1} \amalg A_i \to X_i$ is surjective.
    In either case the automorphism group of $X_i$ as an object of $\Hom_{\ICsp(\mcC)}((A_{i-1},c_{i-1}), (A_i, c_i))$ is trivial.
    This shows that the fibers of $q_n'$ are trivial and thus $q_\cd$ is a level-wise equivalence as claimed.
\end{proof}

We now specifiy a certain simplicial subgroupoid of (a double-decalage of) $\Cc_\cd(\mcC)$, which we will later see is equivalent to the Rezk nerve of the disjoint union of the factorization categories from definition \ref{defn:F_g}.
\begin{defn}\label{defn:CutF}
    Define a simplicial groupoid $\rmF_\cd(\mcC)$ as follows.
    For each $n$ let $\rmF_n(\mcC)$ be the full subgroupoid of $\Cc_{n+2}(\mcC)$ on those
    $(1_\mcC \xrightarrow{f_1} x_1 \to \dots \to x_{n+1} \xrightarrow{f_{n+2}} 1_\mcC) \in N_{n+2}^{\rm R}\mcC$ such that $x_1$ and $x_{n+1}$ are not isomorphic to $1_\mcC$.
    (Recall that $f_{n+2} \circ \dots \circ f_{1}$ is already required to be connected or the identity on $1_\mcC$, and hence this implies that $x_i \not\cong 1_\mcC$ for all $1\le i \le n+1$.)
    The $i$th face map on $\rmF_n(\mcC)$ is defined as the restriction of 
    the $(i+1)$st face map on $\Cc_n(\mcC)$, i.e.~it composes the $i$th and $(i+1)$st morphism.
    The $i$th degeneracy map is defined as the restriction the $(i+1)$st degeneracy map on $\Cc_n(\mcC)$, i.e.~it inserts an identity morphism in position $i+1$.
    There is an augmentation $\rmF_\cd(\mcC) \to G = \Hom_\mcC^\con(1_\mcC, 1_\mcC)$ 
    defined by sending any $(f_1,\dots, f_{n+2})$
    to the composite $f_{n+2} \circ \dots \circ f_1$.
\end{defn}

\begin{defn}
    For a map of spaces $f:X \to Y$ we will use the following coordinates on the (reduced) suspension on the cone:
    \[
        \Sigma\mathrm{Cone}(f) = 
        \left(* \amalg \{ (x, a, b) \in X \times [0,1]^2 \;|\; a+ b \le 1\} 
        \amalg \{ (y, a, b) \in Y \times [0,1]^2 \;|\; a+b = 1 \} \right)
        /\sim
    \]
    where $\sim$ is generated by $(x,a,b) \sim *$ whenever $a\cdot b=0$
    and by $(x,a,b) \sim (f(x),a,b)$ whenever $a+b=1$.
\end{defn}

\begin{ex}\label{ex:cone-and-unreduced-suspension}
    If $Y$ is discrete, then we can decompose $X = \coprod_{y \in Y} X^{(y)}$ and we can rewrite the suspended cone as
    \[
        \Sigma\mathrm{Cone}\left( X \to Y \right)
        \cong \bigvee_{y \in Y} \Sigma\mathrm{Cone}\left( X^{(y)} \to \{y\} \right)
        \cong \bigvee_{y \in Y} S^2(X^{(y)})
    \]
    where $S^2(-)$ denotes the unreduced suspension defined as $S^2(X) := (S^1 \amalg D^2 \times X)/\sim$ with $(a,x) \sim a$ for all $a \in S^1 \subset D^2$ and $x \in X$.
    (For example, $S^2(\emptyset) = S^1$, $S^2(*) = D^2$, and $S^2(* \amalg *) = S^2$.)
\end{ex}

    The space $|B\Cc_\cd(\mcC)|$ is the geometric realization of the bisimplicial set whose set of $(m,n)$-simplices is $N_m\Cc_n(\mcC)$.
    This means that an $(m,n)$-simplex is an $(m+1)$-tuple of objects $(f_1^j,\dots,f_n^j) \in \Cc_n(\mcC)$ for $j = 0,\dots, m$ and isomorphisms between them.
    (We suppress the isomorphisms in the notation.)
    Note that $x_i^0 \not\cong 1_\mcC$ if and only if $x_i^j \not\cong 1_\mcC$ for all $j$, as we have isomorphisms $x_i^0 \cong x_i^j$.

\begin{lem}\label{lem:Cut-con=gS}
    There is a homeomorphism 
        $\alpha:|B\Cc_\cd(\mcC)| \xrightarrow{\cong} 
        \Sigma\mathrm{Cone}(|B\rmF_\cd(\mcC)| \to G)$
    defined by
    \begin{align*}
        &[(1_\mcC \xrightarrow{f_1^j} x_1^j \to \dots \to x_{n-1}^j \xrightarrow{f_n^j} 1_\mcC)_{j=0}^m, (s_0,\dots, s_m), (t_0, \dots, t_n)]\\
        &\longmapsto
        \big[[(1_\mcC \xrightarrow{f_{c}^j} x_c^j \to \dots \to x_{d-1}^j \xrightarrow{f_{d}^j} 1_\mcC)_{j=0}^m, 
        (s_0, \dots, s_m),
        (\tfrac{t_c}{T_{cd}}, \dots, \tfrac{t_{d-1}}{T_{cd}})],
        T_{<c}, T_{d\le}\big]
    \end{align*}
    if at least one $x_i^0$ is not isomorphic to $1_\mcC$.
    Here $0 \le c < d \le n$ where $c$ is the smallest number such that $f_c^0$ is not $\id_{1_\mcC}$ and $d$ is the largest number such that $f_d^0$ is not $\id_{1_\mcC}$ and 
    \[
        T_{<c} := t_0 + \dots + t_{c-1}, \quad
        T_{cd} := t_c + \dots + t_{d-1}, \quad
        T_{d\le} := t_{d} + \dots + t_n.
    \]
    If all the $x_i^0$ are isomorphic to $1_\mcC$, then either all $f_i^0$ are equal to $\id_{1_\mcC}$ and we send it to the basepoint, or there is a unique $c$ such that $f_c^0$ is the is not $\id_{1_\mcC}$ and we set
    \begin{align*}
        &[(1_\mcC \xrightarrow{f_1^j} x_1^j \to \dots \to x_{n-1}^j \xrightarrow{f_n^j} 1_\mcC)_{j=0}^m, (s_0,\dots,s_m), (t_0, \dots, t_n)]\\
        &\longmapsto
        \big[(1_\mcC \xrightarrow{f_{c}^0} 1_\mcC) \in G, 
        T_{<c}, T_{c\le}\big].
    \end{align*}
\end{lem}
\begin{proof}
    For every $(m,n)$-simplex in $(f_1^j,\dots,f_n^j)_{j=0}^m \in N_m\Cc_n(\mcC)$ that involves an object not isomorphic to $1_\mcC$, the formula in the statement defines a map of the form
    \[
        |\Delta^m| \times |\Delta^n| 
        \to N_m \rmF_{d-c-1}(\mcC) \times |\Delta^m| \times |\Delta^{d-c-1}| \times \{(a,b) \in [0,1]^2 \;|\; a+b\le 1\}
        \to \Sigma\mrm{Cone}(|B\rmF_\cd(\mcC)| \to G),
    \]
    where the first map picks the bi-simplex in $N_m \rmF_{d-c-1}(\mcC)$ that is specified by the morphisms $(f_c^j, \dots, f_d^j)_{j=0}^m$, i.e.~that is obtained by restricting to the non-trivial part of the diagram we started with.
    If all the $x_i^0$ are the unit, the map is instead of the form
    \[
        |\Delta^m| \times |\Delta^n| 
        \to G \times \{(a,b) \in [0,1]^2 \;|\; a+b = 1\}
        \to \Sigma\mrm{Cone}(|B\rmF_\cd(\mcC)| \to G).
    \]
    These maps glue together under the face and degeneracy maps to define the map $\alpha$ from the statement of the lemma, as we shall now argue.
    To illustrate the compatibility with face maps in the second simplicial direction, suppose $(f_1^j,\dots,f_n^j)_{j=0}^m$ is an $(m,n)$ simplex where $f_1^0 \neq \id_{1_\mcC}$ then by definition of the face maps in $\Cc_\cd(\mcC)$ (\ref{defn:CutmcC}) we have $d_0(f_1^j,\dots,f_n^j)_{j=0}^m = (\id_{1_\mcC},\dots, \id_{1_{\mcC}})$, which is sent to the base point by $\alpha$.
    On the other hand, $\alpha$ sends a point $(t_0,\dots,t_n)$ in the original simplex to $[\dots, T_{<c}, T_{d\le}]$ and in this case $c=1$ so $T_{<c} = t_0$ and if we restrict to the $0$th face then $t_0 = 0$, so that $\alpha$ sends the point to the base point in $\Sigma\mrm{Cone}(\dots)$.
    A dual argument deals with $d_n$, and compatibility with inner face maps and degeneracies holds by inspection of the formulas.
    
    To check that the compatibility with face maps in the first simplicial direction, the key is to observe that definitions we made in terms of $x_i^0$ and $f_i^0$ could equivalently be said in terms of any other $x_i^j$ and $f_i^j$.
    This is true because we know that $x_i^0 \cong x_i^j$ and $f_i^0 \cong f_i^j$ (in $\Ar(\mcC)$) and all the conditions we used are invariant under isomorphism.
    (Even the condition $f_b^0 \neq \id_{1_\mcC}$ is invariant under isomorphism in $\Cc_n(\mcC)$ because in definition \ref{defn:CutmcC} we required that $x_i \cong 1_\mcC$ implies $x_i = 1_\mcC$ and moreover every automorphism of $1_\mcC$ is the identity.)

    We now describe an inverse map $\beta: \Sigma\mathrm{Cone}(|B\rmF_\cd(\mcC)| \to G) \to |B\Cc_\cd|$.
    First, we describe it on $|B\rmF_\cd(\mcC)| \times \{(a,b)\;|\; a+b\le 1\}$, where it is given as
    \begin{align*}
        &\big[[(1_\mcC \xrightarrow{f_1^i} x_1^i \to \dots \to x_{n-1}^i \xrightarrow{f_n^i} 1_\mcC)_{i=0}^m, (s_0,\dots, s_m), (t_1, \dots, t_{n-1})], a, b\big]\\
        &\mapsto
        [(1_\mcC \xrightarrow{f_1^i} x_1^i \to \dots \to x_{n-1}^i \xrightarrow{f_n^i} 1_\mcC)_{i=0}^m, 
        (s_0,\dots, s_m), (a, (1-a-b)\cdot t_1, \dots, (1-a-b)\cdot t_{n-1}, b)]
    \end{align*}
    for all $m\ge 0$ and $n\ge 2$.
    Again, this should be interpreted as a map that is glued together from maps of topological simplices $|\Delta^m| \times |\Delta^{n-2}| \times \{(a,b) \;|\; a+b \le 1\} \to |\Delta^m| \times |\Delta^{n}|$.
    Second, we describe it on $G \times \{(a,b)\;|\; a+b=1\}$.
    Here we send $[(g:1_\mcC \to 1_\mcC) \in G, a, b]$
    to $[[1_\mcC \xrightarrow{g} 1_\mcC, (1), (a,b)], a, b]$.
    Inspecting definitions we see that this glues together to give a well-defined map out of $\Sigma\mrm{Cone}(\dots)$.

    By construction the composite $\alpha \circ \beta$ is the identity on $\Sigma\mrm{Cone}(|B\rmF_\cd(\mcC)| \to G)$,
    but the other composite $\beta \circ \alpha$ might not look like $\id_{|B\Cc_\cd(\mcC)|}$ at first sight.
    On the simplices with some $x_i \neq 1_\mcC$ it is given by
    \begin{align*}
        &[(1_\mcC \xrightarrow{f_1^j} x_1^j \to \dots \to x_{n-1}^j \xrightarrow{f_n^j} 1_\mcC)_{j=0}^m, (s_0, \dots, s_m), (t_0, \dots, t_n)]\\
        &\longmapsto
        [(1_\mcC \xrightarrow{f_{c}^j} x_c^j \to \dots \to x_{d-1}^j \xrightarrow{f_{d}^j} 1_\mcC)_{j=0}^m,
        (s_0, \dots, s_m),
        (T_{<c}, t_c, \dots, t_{d-1}, T_{d\le})]
    \end{align*}
    Since we assume that $x_1, \dots, x_{c-1}$ and $x_{d}, \dots, x_{n-1}$
    are all isomorphic (and hence equal) to $1_\mcC$, it follows that $f_1, \dots f_{c-1}$ and $f_{d+1}, \dots, f_n$
    are the identity morphisms on $1_\mcC$.
    (If they were non-identities, they would contribute another connected component to $f_{0 \to n}$, which is required to be connected.)
    By the definition of the geometric realization we are allowed 
    to remove the identity morphism $f_1$ and add the adjacent weights $t_0$ and $t_1$. 
    Iterating this shows that $\beta \circ \alpha$ simply sends every point to a different representative of the same point, i.e.\ it is the identity on the bi-simplices with $x_i^0 \not\cong 1_\mcC$ for some $i$.
    In the case where all $x_i^j$ are the unit, as similar but simpler argument shows the same.
\end{proof}

    The augmentation yields a disjoint decomposition:
    \[
        \rmF_\cd(\mcC) \cong  \coprod_{g \in G} \rmF_\cd^{(g)}(\mcC)
    \]
    where we let $\rmF_n^{(g)}(\mcC)$ be the fiber of $\rmF_n(\mcC) \to G$ at $g$.

Now we show that $|B\rmF_\cd(\mcC)|$ is in fact the classifying space of the category 
of factorisations.
\begin{lem}\label{lem:CutF=F}
    There is a level-wise equivalence between the simplicial groupoid $\rmF_\cd^{(g)}(\mcC)$ and the Rezk nerve of the factorization category $\mcF_g(\mcC)$.
    In particular, we by corollary \ref{cor:groupoid-basechange} have homotopy equivalences
    \[
        B\mcF_g(\mcC) \simeq |BN_\cd^{\rm R} \mcF_g(\mcC)| \simeq |B\rmF_\cd^{(g)}\mcC|.
    \]
\end{lem}
\begin{proof}
    Without loss of generality we may assume that every object in $\mcC$ that is isomorphic to $1_\mcC$ is in fact equal to $1_\mcC$ -- this only changes $\mcF_g(\mcC)$ and its Rezk nerve up to equivalence.
    Under this hypothesis there is an isomorphism of simplicial groupoids $\rmF_\cd^{(g)}(\mcC) \cong N_\cd^{\rm R}\mcF_g(\mcC)$, which is given by
    \[
        \left(
        \begin{tikzcd}[column sep = 2.25em]
	{1_\mcC} & {x_1} & \dots & {x_{n+1}} & {1_\mcC}
	\arrow["{f_1}", from=1-1, to=1-2]
	\arrow["{f_2}", from=1-2, to=1-3]
	\arrow[from=1-3, to=1-4]
	\arrow["{f_{n+1}}", from=1-4, to=1-5]
\end{tikzcd}
        \right)
        \longmapsto
        \left(
        \begin{tikzcd}
        	& {1_\mcC} \\
        	{x_1} & {x_2} & \dots & {x_{n+1}} \\
        	& {1_\mcC}
        	\arrow["{f_1}"', from=1-2, to=2-1] \arrow[from=1-2, to=2-2] \arrow["{f_n \circ \dots \circ f_1}", from=1-2, to=2-4] \arrow["f_2", from=2-1, to=2-2] \arrow["{f_{n+1} \circ \dots \circ f_2}"', from=2-1, to=3-2] \arrow[from=2-2, to=2-3] \arrow[from=2-2, to=3-2] \arrow[from=2-3, to=2-4] \arrow["{f_{n+1}}", from=2-4, to=3-2]
        \end{tikzcd}
        \right)
    \]
    on $n$-simplices.
    This is compatible with the face maps and degeneracy maps:
    except for $d_0$ and $d_n$ they are on both sides just given by composing morphisms or inserting identity morphisms.
    (The face map $d_0$ forgets $x_1$ and composes $f_2 \circ f_1$, which on the right side corresponds to forgetting the object $(1_\mcC \to x_1 \to 1_\mcC)$ of $\mcF_g(\mcC)$.)
\end{proof}

To summarise, we record:
\begin{cor}\label{cor:Cut=Q(SBF)}
    There are equivalences of $\gC$-spaces:
    \begin{align*}
        |B\Cut_\cd| 
        &\simeq Q(|B\Cut_\cd^\con|) 
        \simeq Q(|B\Cc_\cd(\mcC)|) 
        \simeq Q\left(\bigvee_{g \in G} S^2|B\rmF_\cd^{(g)}| \right) 
        \simeq Q\left(\bigvee_{g \in G} S^2 (B\mcF_g) \right) .
    \end{align*}
\end{cor}
\begin{proof}
   Combine lemma \ref{lem:Cut-free-on-con}, lemma \ref{lem:Cut-to-CutmcC}, example \ref{ex:cone-and-unreduced-suspension}, lemma \ref{lem:Cut-con=gS}, 
   and lemma \ref{lem:CutF=F}.
\end{proof}

Together with lemma \ref{lem:cut-fiber-sequence}  and lemma \ref{lem:NncICsp=NncC} this proves the decomposition theorem \ref{thm:decomposition}.

\subsection{The reduced category}
We now prove the comparison result proposition \ref{prop:closed-reduced} 
between $\ICsp(\mcC)$, $\mcC$ and $\mcC^\red$, in particular showing that $\ICsp(\mcC) \to \mcC$ is a rational equivalence on classifying spaces.
This uses similar results to what we have already seen in the proof of the decomposition theorem \ref{thm:decomposition}.

\begin{defn}
    Let $\ICsp(\mcC)^\cl \subset \ICsp(\mcC)$ denote the full subcategory 
    on the object $\emptyset$ and let $\mcC^\cl \subset \mcC$ denote 
    the full subcategory on the object $1_\mcC$.
\end{defn}

\begin{lem}\label{lem:B-of-cl}
    There are canonical equivalences of infinite loop spaces:
    \[
        B(\mcC^\cl) \simeq B\left(\bigoplus_G \IN\right)
        \qand
        B(\ICsp(\mcC)^\cl) \simeq Q\left(\bigvee_G S^1\right) .
    \]
\end{lem}
\begin{proof}
    By definition $\mcC^\cl$ is a category with a single object $1_\mcC$
    and the endomorphisms of this object are 
    $\Hom_\mcC(1_\mcC, 1_\mcC) \cong \bigoplus_G \IN$ by axiom (ii) 
    of definition \ref{defn:labelled-cospans}.
    This implies the first equivalence.
    
    The second equivalence can easily be proven from the perspective 
    of \cite{Seg74}, but for the readers convenience 
    we will prove it as an application of corollary \ref{cor:level-wise-free-sm-cat}.
    (Note, however, that this corollary is also based on \cite{Seg74}.)
    Indeed, the groupoid $\Hom_{\ICsp(\mcC)^\cl}(\emptyset, \emptyset)$ 
    is ``freely generated'' by its connected objects in the sense that it is equivalent to $\Sigma_{\mc{H}}$ for $\mc{H} = \Hom_{\ICsp(\mcC)^\cl}^\con(\emptyset, \emptyset)$.
    Let $K_\cd$ be the level-wise full simplicial subgroupoid of the nerve 
    $N_\cd \ICsp(\mcC)^\cl$ containing all those objects $\UL{X}:[n] \to \ICsp(\mcC)^\cl$
    where $\UL{X}_{0 \to n}: \emptyset \to \emptyset$ is connected or empty.
    Because face and degeneracy maps preserve this subgroupoid, it defines a functor $K_\cd: \Delta^\op \to \Cat^+$ and we can apply \ref{cor:level-wise-free-sm-cat} to obtain an equivalence
    \[
        Q(\|B K_\cd\|) \simeq 
        \|B\gS_{K_\cd^{\neq \emptyset}}\| \simeq 
        |BN_\cd\ICsp(\mcC)^\cl| = B(\ICsp(\mcC)^\cl) .
    \]
    Here the middle equivalence comes from an equivalence of simplicial $\gS$-categories $\gS_{K_\cd^{\neq \emptyset}} \simeq N_\cd\ICsp(\mcC)^\cl$ which can be constructed in the same way as in the proof of lemma \ref{lem:Cut-free-on-con}.

    We will show that $BK_\cd$ is level-wise equivalent to 
    $Y_\cd := (\gD^1 \times G)/(\partial \gD^1 \times G)$,
    the geometric realization of which is a wedge of circles, one for each element of $G$.
    The $n$-simplices of $\gD^1 \times G$ can be thought of as tuples
    $(i,g) \in \{0,\dots,n+1\} \times G$ where $i$ represents the unique map
    $[n] \to [1]$ with $i \mapsto 0$ and $(i+1) \mapsto 1$.
    The simplicial set $Y_\cd$ is obtained by identifying all tuples $(i,g)$
    with $i \in \{0,n+1\}$.
    Indeed, there is a level-wise equivalence of simplicial groupoids
    $K_\cd \simeq Y_\cd$ defined by sending $(\UL{X}:[n] \to \ICsp(\mcC)) \in K_n$
    to the basepoint if $\UL{X}_{0 \to n}$ is empty and otherwise
    to $(i, [\UL{X}_{0 \to n}])$ where $i \in \{1,\dots,n\}$ is the unique
    $i$ such that $\UL{X}_{i-1 \to i}$ is non-empty.
    From this it follows that
    \[
        \| BK_\cd \| \simeq \| (\gD^1 \times G)/(\partial \gD^1 \times G) \|
        \simeq \bigvee_{g \in G} S^1. \qedhere
    \]
\end{proof}

Now we can prove proposition \ref{prop:closed-reduced}, which states that we have fiber sequences of infinite loop spaces
    \[
        \begin{tikzcd}
            Q\left(\bigvee_G S^1\right) \ar[d] \ar[r] & 
            B(\ICsp(\mcC)) \ar[d] \ar[r] & 
            B(\mcC^\red) \ar[d, equal] \\
            B\left(\bigoplus_G \IN\right) \ar[r] & 
            B(\mcC) \ar[r] & 
            B(\mcC^\red).
        \end{tikzcd}
    \]
whenever $B\mcC$ is grouplike, and in particular that the middle map is a rational equivalence.

\begin{proof}[Proof of proposition \ref{prop:closed-reduced}]
    For the purpose of the proof we will assume that $\mcC$ has a PCM structure, rather than a symmetric monoidal structure, and that the functor $\ICsp(\mcC) \to \mcC$ is a PCM functor. 
    (This can be achieved by the construction from example \ref{ex:PCM-vs-sm}.)
    Consider the following diagram of simplicial PCM groupoids:
    \[
        \begin{tikzcd}
            D_\cd^\cl \ar[d] \ar[r] & 
            D_\cd \ar[d] \ar[r] & 
            D_\cd^\red \ar[d] \\
            N_\cd \mcC^\cl \ar[r] & 
            N_\cd^{\rm R} \mcC \ar[r] & 
            N_\cd^{\rm R} \mcC^\red.
        \end{tikzcd}
    \]
    Here $D_\cd^\cl \subset D_\cd$ is the level-wise full subgroupoid 
    on those $(\UL{X}_\cd, A_\cd)$ with $A_i = \emptyset$ for all $i$.
    $D_\cd^\red$ is defined level-wise as the full subgroupoid $D_n^\red \subset D_n$
    on those $n$-simplices that represent $n$-tuples of composable reduced cospans.
    Similar to the definition of $\Cut_\cd$ we define the face maps in $D_\cd^\red$
    by taking the face map in $D_\cd$ and then forgetting closed components.
    In the bottom row $N_\cd \mcC^\cl$ is simply the simplicial commutative monoid
    where each $N_n\mcC^\cl$ is isomorphic to $(\IN\gle{G})^n$.
    
    It follows from lemma \ref{lem:B-of-cl} and our previous results 
    that the realization of this diagram is a commutative diagram 
    of $\gC$-spaces that is equivalent to the diagram stated in 
    the claim of this proposition.
    Moreover, each of the terms is group-like (note that $\pi_0B\mcC^\red \cong \pi_0B\mcC$ since there is a morphism $x \to y$ in $\mcC^\red$ if and only if there is such a morphism in $\mcC$), and so we have 
    a diagram of infinite loop spaces.
    
    To prove the proposition we apply the general fiber sequence criterion
    theorem \ref{thm:fiber-sequence} to the two sequences of simplicial PCM groupoids.
    For the first one, consider the PCM functor:
    \[
        S_n: D_n \to D_n^\cl, \qquad 
        (A_\cd, \UL{X}_\cd, c_\cd, o_\cd) \mapsto 
        (\emptyset, \UL{Y}_\cd, \emptyset, (o_\cd)_{|\UL{Y}_\cd})
    \]
    where each $\UL{Y}_i \subset \UL{X}_i$ is defined as the complement 
    of the image of $A_{i-1} \amalg A_i \to \UL{X}_i$.
    This map is a section to the inclusion $D_n^\cl \to D_n$, so to apply the fiber
    sequence theorem we need to check that $(S_n, R_n):D_n \to D_n^\cl \times D_n^\red$
    is an equivalence. This functor is given by
    \[
        (A_\cd, \UL{X}_\cd, c_\cd, o_\cd) \mapsto 
        \left((\emptyset, \UL{Y}_\cd, \emptyset, (o_\cd)_{|\UL{Y}_\cd}), 
        (A_\cd, \UL{Z}_\cd, c_\cd, (o_\cd)_{|\UL{Z}_\cd})\right)
    \]
    where $\UL{Y}_i$ is as above and $\UL{Z}_i \subset \UL{X}_i$ is the image
    of $A_{i-1} \amalg A_i$, i.e.\ the complement of $\UL{Y}_i$.
    Just as in the proof of \ref{lem:cut-fiber-sequence} we see that this functor is an 
    isomorphism onto the subcategory of disjoint tuples and that hence 
    it is an equivalence.
    
    To conclude that the first row is a homotopy fiber sequence we note that 
    $\pi_0|BD_\cd^\cl| \cong *$, $\pi_0|BD_\cd| \cong \pi_0B\mcC$, 
    and $\pi_0|BD_\cd^\red| \cong \pi_0 B\mcC^\red$ are all group-like.
    
    For the second row the proof is similar. The section is given by
    \[
        S_n: N_n^{\rm R}\mcC \to N_n \mcC^\cl, \quad
        (W:[n] \to \mcC) \mapsto (V_1, \dots, V_n) \in (\Hom_\mcC(1_\mcC,1_\mcC))^n = N_n\mcC^\cl
    \]
    where $V_i \subset W(i-1 \to i)$ is the union of the closed components.
    This is a PCM functor where the commutative monoid $N_n\mcC^\cl$ 
    is thought of as a PCM category with only identity morphisms 
    and where all objects are disjoint.
    Since this is indeed a section to the inclusion $N_n\mcC^\cl \to N_n^{\rm R}\mcC$
    we again need to check that the functor 
    $(S_n,R_n): N_n^{\rm R}\mcC \to N_n\mcC^\cl \times N_n\mcC^\red$
    is an equivalence. 
    This follows from axiom (iii) of definition \ref{defn:labelled-cospans}.
    This completes the proof that both rows are homotopy fiber sequences.
    
    Finally, we want to check that $B(\ICsp(\mcC)) \to B(\mcC)$ is a rational equivalence.
    By comparing the homotopy fibers of the first two vertical maps between 
    the two homotopy fiber sequence we obtain the following homotopy fiber sequence:
    \[
        \tau_{\ge 2} Q(\bigvee_{g \in G} S^1) \longrightarrow
        B(\ICsp(\mcC)) \longrightarrow 
        B(\mcC).
    \]
    The homotopy groups of the left-hand space are 
    $\pi_k \tau_{\ge 2} Q(\bigvee_{g \in G} S^1) 
    = \bigoplus_{g \in G} \pi_{k-1}^{\rm st}(S^0)$ for $k \ge 2$.
    This is rationally trivial because the stable homotopy groups of spheres
    $\pi_*^{\rm st}(S^0)$ are finite for $* \ge 1$.
    Hence the map $B(\ICsp(\mcC)) \to B(\mcC)$ is an isomorphism on rational 
    homotopy groups, as claimed.
\end{proof}

\subsection{Contracting the factorisation category}
\label{subsec:mcF-contractible}

In this section we show that the factorisation categories for 
$\Cob_d$ have contractible classifying spaces for all $d \le 2$.
The argument is based on an argument showing $B\mcF(\Csp) \simeq *$
and crucially uses the existence of the disk morphism $D^d:\emptyset \to S^{d-1}$.
In particular, it will not apply to $\Cobn$.

\begin{defn}
    For a non-empty set $X$ let $\Fin_{/X}^{\neq \emptyset, \rm inj}$
    denote the category where objects are tuples $(A,a)$ 
    of a non-empty finite set $A$ together with a map $a:A \to X$,
    and morphisms $(A,a) \to (B,b)$ are injections $i:A \to B$
    such that $a = b \circ i$.
\end{defn}

\begin{lem}\label{lem:Fin-contractible}
    For every non-empty set $X$ the classifying space $B\Fin_{/X}^{\neq \emptyset, \rm inj}$
    is contractible.
\end{lem}
\begin{proof}
    Pick any $x_0 \in X$ and consider the functor 
    \[
        F:\Fin_{/X}^{\neq \emptyset, \rm inj} \to \Fin_{/X}^{\neq \emptyset, \rm inj} \qquad
        (A,a) \mapsto (A \amalg \{x_0\}, a \amalg \id_{\{x_0\}})
    \]
    that adds a new point to $A$ and labels it by $x_0$.
    There is a natural transformation $\ga:\Id_{\Fin_{/X}^{\neq \emptyset, \rm inj}} \Rightarrow F$
    coming from the canonical inclusion $\ga_A : A \hookrightarrow A \amalg \{x_0\}$.
    There also is a natural transformation $\gb:G \Rightarrow F$ from the 
    constant functor $G(A,a) := (\{x_0\}, \id_{\{x_0\}})$ coming
    from the other inclusion $\gb_A: \{x_0\} \hookrightarrow A \amalg \{x_0\}$.
    Both $\ga$ and $\gb$ induce homotopies after taking classifying spaces 
    and together they imply that the identity on $B\Fin_{/X}^{\neq\emptyset, \rm inj}$
    is homotopic to $BG$, which is constant.
\end{proof}

\begin{lem}\label{lem:F(Csp)-contractible}
    The factorisation category $\mcF(\Csp)$ has a contractible classifying space.
\end{lem}
\begin{proof}
    Consider the functor $\partial_0:\mcF(\Csp) \to \Fin$ that sends 
    the factorisation $(W:1_\mcC \to M, W':M \to 1_\mcC)$ to the underlying 
    set $\pi(W)$ of the first morphism.
    To a morphism $X:M \to N$ from $(M,W,W')$ to $(N,V,V')$ 
    this assigns the composite $\pi(W) \to \pi(W) \cup_{\pi(M)} \pi(X) \cong \pi(V)$.
    A short explanation is due to argue why we can talk about $\pi(W)$ 
    as a well-defined set. A priori the morphism $W$ only yields an 
    isomorphism class of cospans $[\emptyset \to \pi(W) \leftarrow \pi(M)]$,
    but not a well-defined set $\pi(W)$. But because this is part 
    of a factorisation we know that $\pi(M) \to \pi(W)$ is surjective.
    In particular we can define $\partial_0(M, W, W')$ as the quotient
    of $\pi(M)$ by the relation induced by $\pi(M) \to \pi(W)$.
    
    We define a new category $\mcF'$ as the Grothendieck construction 
    of the functor $\mcF(\Csp) \to \Cat$ that sends $(M,W,W')$ to 
    $\Fin_{/\partial_0(M,W,W')}^{\neq \emptyset, \rm inj, \op}$.
    Spelling out the construction we see that an object in $\mcF'$
    is a tuple $((M,W,W'),(A,a))$ where $(M,W,W') \in \mcF(\Csp)$,
    $A$ is a non-empty finite set and $a:A \to \partial_0(M,W,W')$.
    A morphism $((M,W,W'),(A,a)) \to ((N,V,V'),(B,b))$ is a 
    tuple of a morphism $X:M \to N$ in $\mcF(\Csp)$ and an injection $i:A \to B$ 
    such that $\partial_0(X) \circ a = b \circ i$.
    Consider the forgetful functor $F:\mcF' \to \mcF(\Csp)$ that sends 
    $((M,W,W'),(A,a))$ to $(M,W,W')$.
    
    It is a consequence of Quillen's theorem A that,
    given any functor $\Phi:\mcC \to \Cat$ where each $B(\Phi(c))$ is contractible,
    the forgetful functor $\int_\mcC \Phi \to \mcC$ is an equivalence on classifying spaces.
    (One can also think of this as an instance of Thomason's homotopy colimit theorem \cite{Tho79}.)
    In lemma \ref{lem:Fin-contractible} we showed that each 
    $\Fin_{/\partial_0(M,W,W')}^{\neq \emptyset, \rm inj, \op}$
    has a contractible classifying space,
    hence $F:\mcF' \to \mcF(\Csp)$ is a homotopy equivalence on classifying spaces.
    We will prove the lemma by showing that the homotopy equivalence 
    $B\mcF' \to B\mcF(\Csp)$ is null-homotopic.
    
    Define a functor $G:\mcF' \to \mcF(\Csp)$ by
    \[
        ((M, [\emptyset \to W \leftarrow M], [M \to V \leftarrow \emptyset]),(a:A \to W))
        \mapsto 
        (A, [\emptyset \to A \xleftarrow{=} A], [A \to * \leftarrow \emptyset])
    \]
    to a morphism $(X,i):((M,W,W'),(A,a)) \to ((N,V,V'),(B,b))$ this
    assigns the morphism 
    \[
        [A \xrightarrow{i} B \xleftarrow{=} B].
    \]
    This functor by definition factors through the forgetful functor $\mcF' \to \Fin^{\neq\emptyset, \rm inj}$ that sends $((M,W,W'), (A,a))$ to $A$. 
    Since the classifying space $B\Fin^{\neq\emptyset, \rm inj}$ is contractible
    (set $X=*$ in lemma \ref{lem:Fin-contractible}),
    we can conclude that $BG:\mcF' \to \mcF(\Csp)$ is null-homotopic.
    
    We construct a natural transformation $\ga:G \Rightarrow F$
    by assigning to each $((M,W,W'), (A,a))$ the morphism
    \[
        \ga_{((M,W,V), (A,a))}:= [A \xrightarrow{a} W \leftarrow M] :
        (A, [\emptyset \to A \xleftarrow{=} A], [A \to * \leftarrow \emptyset])
        \to (M, W, V).
    \]
    To see that this is a well-defined morphism in the factorisation category we check
    \begin{align*}
        [A \to W \leftarrow M] \cup_M [M \to W' \leftarrow \emptyset] 
        &= [A \to * \leftarrow \emptyset] \\
        [\emptyset \to A \leftarrow A] \cup_A [A \to W \leftarrow M] 
        &= [\emptyset \to W \leftarrow M].
    \end{align*}
    Here the first equation used that $W \cup_M W' = *$ holds for all objects 
    $(M, W, W') \in \mcF(\Csp)$.
    We also need to check that $\ga$ is natural. 
    To see this, consider a morphism 
    $(X,i):((M,W,W'),(A,a)) \to ((N,V,V'),(B,b))$
    where $X = [M \to X \leftarrow N]:(M,W,W) \to (N,V,V')$ 
    and $i:A \hookrightarrow B$.
    For any such morphism we have that:
    \[
        [A \xrightarrow{i} B \xleftarrow{=} B] \cup_B [B \xrightarrow{b} V \leftarrow N]
        = [A \xrightarrow{b \circ i} V \leftarrow N]
        = [A \xrightarrow{a} W \leftarrow M] \cup_M [M \to X \leftarrow N]
    \]
    since $V = W \cup_M X$ and $b \circ i = \partial_0(X) \circ a$.
    
    In summary, we have shown that the forgetful functor induces 
    a homotopy equivalence $BF:B\mcF' \to B\mcF(\Csp)$,
    that the map $BG:B\mcF' \to B\mcF(\Csp)$ factors through
    the contractible space $B\Fin^{\neq\emptyset,\rm inj}$,
    and that $\ga$ induces a homotopy between $BF$ and $BG$.
    This implies that $B\mcF(\Csp)$ is contractible as claimed.
\end{proof}

The same argument applies to $B\mcF_a(\Csp(A,A_1,\ga)) \simeq *$ for all $a \in A$
as long as $0 \in A_1$, i.e.\ as long as $A_1 = A$.

\begin{figure}[h]
    \centering
    \small
    \def\svgwidth{\figurerescalefactor\linewidth}
\begingroup%
  \makeatletter%
  \providecommand\color[2][]{%
    \errmessage{(Inkscape) Color is used for the text in Inkscape, but the package 'color.sty' is not loaded}%
    \renewcommand\color[2][]{}%
  }%
  \providecommand\transparent[1]{%
    \errmessage{(Inkscape) Transparency is used (non-zero) for the text in Inkscape, but the package 'transparent.sty' is not loaded}%
    \renewcommand\transparent[1]{}%
  }%
  \providecommand\rotatebox[2]{#2}%
  \newcommand*\fsize{\dimexpr\f@size pt\relax}%
  \newcommand*\lineheight[1]{\fontsize{\fsize}{#1\fsize}\selectfont}%
  \ifx\svgwidth\undefined%
    \setlength{\unitlength}{195.39065624bp}%
    \ifx\svgscale\undefined%
      \relax%
    \else%
      \setlength{\unitlength}{\unitlength * \real{\svgscale}}%
    \fi%
  \else%
    \setlength{\unitlength}{\svgwidth}%
  \fi%
  \global\let\svgwidth\undefined%
  \global\let\svgscale\undefined%
  \makeatother%
  \begin{picture}(1,0.28473509)%
    \lineheight{1}%
    \setlength\tabcolsep{0pt}%
    \put(0,0){\includegraphics[width=\unitlength,page=1]{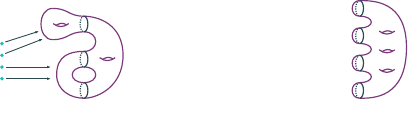}}%
    \put(-0.01608103,0.05966367){\color[rgb]{0,0,0}\makebox(0,0)[lt]{\lineheight{1.25}\smash{\begin{tabular}[t]{l}$a:A \to \pi_0(W)$\end{tabular}}}}%
    \put(0.12628202,0.00211416){\color[rgb]{0,0,0}\makebox(0,0)[lt]{\lineheight{1.25}\smash{\begin{tabular}[t]{l}$\emptyset \xrightarrow{\ W\ } M \xrightarrow{\ W'\ } \emptyset$\end{tabular}}}}%
    \put(0.77890078,0.00211416){\color[rgb]{0,0,0}\makebox(0,0)[lt]{\lineheight{1.25}\smash{\begin{tabular}[t]{l}$\emptyset \xrightarrow{A \times D^2} A \times S^1 \xrightarrow{Q_A} \emptyset$\end{tabular}}}}%
    \put(0,0){\includegraphics[width=\unitlength,page=2]{contracting-F.pdf}}%
    \put(0.37068426,0.00211427){\color[rgb]{0,0,0}\makebox(0,0)[lt]{\lineheight{1.25}\smash{\begin{tabular}[t]{l}$\emptyset \xrightarrow{A \times D^2} A \times S^1 \xrightarrow{\ W_A\ } M \xrightarrow{\ W'\ }  \emptyset$\end{tabular}}}}%
    \put(0,0){\includegraphics[width=\unitlength,page=3]{contracting-F.pdf}}%
  \end{picture}%
\endgroup%

    \caption{Left: an object $((M,W,W'),(A,a))\in\mcF_{g=3}'(\Cob_2)$.
    Middle: the value of the natural transformation 
    $\ga:G \Rightarrow \Id$ at this object.
    Right: the value the functor $G$ on this object.
    }
    \label{fig:contracting-F}
\end{figure}

\begin{lem}\label{lem:F(Cob)-contractible}
    For all $d \ge 2$ and all diffeomorphism types $[Q]$ of connected closed $d$-manifolds 
    the factorisation category $\mcF_{[Q]}(\Cob_d)$ has a contractible 
    classifying space.
\end{lem}
\begin{proof}
    An object of $\mcF_{[Q]}(\Cob_d)$ is a triple $(M, [W], [W'])$ where $M$ is 
    a closed oriented $(d-1)$-manifold and $[W]$ and $[W']$ are diffeomorphism
    classes of $d$-dimensional cobordisms $W:\emptyset \to M$ and $W':M \to \emptyset$
    such that their composite $W \cup_M W'$ is diffeomorphic to $Q$.
    A morphism $(M, [W], [W']) \to (N, [V], [V'])$ is a diffeomorphism class
    of cobordisms $[X]:M \to N$ satisfying $W \cup_M X \cong V$ and $W' \cong X \cup_N V'$.
    
    We will choose explicit representatives for $[W]$ and $[W']$ throughout.
    This does not change the category $\mcF_{[Q]}(\Cob_d)$, as long as we still define
    morphisms as above.
    
    As in the proof of lemma \ref{lem:F(Csp)-contractible} we construct
    a category $\mcF'$ as follows.
    Objects are tuples $((M, W, W'), (a:A \to \pi_0(W)))$ where $(M, W, W')$
    is an object as above, $A$ is a finite non-empty set, and $a:A \to \pi_0(W)$ a map.
    Because $\Fin_{/\pi_0(W)}^{\neq \emptyset, \rm inj}$ is contractible by lemma \ref{lem:Fin-contractible}, Quillen's theorem A implies that the forgetful functor $F:\mcF' \to \mcF_{[Q]}(\Cob_d)$ is an equivalence on classifying spaces.
    
    The following constructions are illustrated in figure \ref{fig:contracting-F}.
    Consider the functor $G:\Fin^{\neq\emptyset,\rm inj} \to \mcF_{[Q]}(\Cob_d)$
    that sends a finite set $A$ to the factorisation 
    \[
        G(A) = (A \times S^{d-1}, A \times D^d: \emptyset \to A \times S^{d-1}, 
        Q_A:A \times S^{d-1} \to \emptyset)
    \]
    where the cobordism $Q_A: A \times S^{d-1} \to \emptyset$ is obtained 
    by picking an injection $A \to Q$, removing a neighbourhood of its image,
    and identifying its boundary with $A \times S^{d-1}$ in the canonical orientation
    preserving way.
    Since $Q$ is connected and of dimension $d \ge 2$
    any choice of injection leads to the same diffeomorphism class of cobordism.
    To injections $i:A \to B$ the functor $G$ assigns the cobordism:
    \[
        G(i: A \to B) := [(A\times S^{d-1} \times [0,1]) \amalg ((B \setminus i(A)) \times D^d)]:
        A \times S^{d-1} \to B \times S^{d-1}.
    \]
    This cobordism connects the sphere $\{a\} \times S^{d-1}$ with $\{i(a)\} \times S^{d-1}$
    via a cylinder and caps off each of the spheres $\{b\} \times S^{d-1}$
    with $b \not\in i(A)$ with a disk.
    This defines a morphism in the factorisation category because 
    closing off the $(B \setminus i(A))$-spheres of $Q_A$ yields a manifold
    diffeomorphic to $Q_B$: 
    \[
        G(i:A \to B) \cup_{B \times S^{d-1}} Q_B \cong Q_A
        \qand
        (A \times D^d) \cup_{A \times S^{d-1}} G(i:A \to B) \cong (B \times D^d).
    \]
    Let us denote the composite of $G$ with the forgetful functor 
    $\mcF' \to \Fin^{\neq\emptyset, \rm inj}$ also by $G$.
    
    To conclude the proof in the same manner as for lemma \ref{lem:F(Csp)-contractible}
    we would like to construct a natural transformation
    $\ga: G \Rightarrow F$ of functors $\mcF' \to \mcF_{[Q]}(\Cob_d)$.
    For an object $((M,W,W'),(a:A \to \pi_0(W))) \in \mcF'$ we let 
    $W_A: A \times S^{d-1} \to M$ be the cobordism obtained by lifting the map
    $a:A \to \pi_0(W)$ to an embedding $a:A \hookrightarrow W$ and 
    then removing a small neighbourhood of its image.
    We define the components of $\ga$ as:
    \[
        \ga_{((M,W,W'),(a:A \to \pi_0(W)))} := W_A :
        (A \times S^{d-1}, A \times D^d, Q_A) \to (M, W, W').
    \]
    To check that this is a well-defined morphism in the factorisation category
    we note:
    \[
        (A \times D^d) \cup_{A \times S^{d-1}} W_A \cong W
        \qand
        W_A \cup_M W' \cong Q_A.
    \]
    To check naturality let $([X],i):((M,W,W'),(A,a)) \to ((N,V,V'),(B,b))$
    be a morphism in $\mcF'$, i.e.\ $X:M \to N$ a cobordism and $i:A \to B$
    and injection, both satisfying certain conditions.
    Then we have:
    \[
        G(i:A \to B) \cup_{B \times S^{d-1}} V_B  \cong 
        V_A \cong
        W_A \cup_M X.
    \]
    Here the first diffeomorphism comes from the fact that both sides
    are obtained from $V:\emptyset \to N$ by removing disks according to 
    $(\partial_0(X)\circ a:A \to \pi_0(W) \to \pi_0(V)) = (b\circ i:A \to B \to \pi_0(V))$,
    and the second diffeomorphism comes from the fact that we can arrange
    for the diffeomorphism $V \cong W \cup_M X$ to fix the disks corresponding to $A$.
    
    This shows that $BG, BF:B\mcF' \to B\mcF_{[Q]}(\Cob_d)$ are homotopic.
    Since $BF$ is a homotopy equivalence and $BG$ is null-homotopic,
    this implies that $B\mcF_{[Q]}(\Cob_d)$ is contractible, as claimed.
\end{proof}

\begin{cor}
    For $d \ge 2$ there are equivalences:
    \[
        B\ICsp \simeq |N_\cd^\nc\Csp| \qand
        B\ICob_d \simeq |N_\cd^\nc\Cob_d|.
    \]
\end{cor}
\begin{proof}
    Combine the decomposition theorem \ref{thm:decomposition}
    with lemma \ref{lem:F(Cob)-contractible} and lemma \ref{lem:F(Csp)-contractible}.
\end{proof}
\section{The surgery theorem}\label{sec:surgery}

In this section we try to understand the realization of the simplicial set
$N_\cd^\nc\mcC$, which is one of the two main components in the decomposition theorem.
We will prove the surgery theorem \ref{thm:Csp-surgery}, which says that under favourable circumstances $|N_\cd^\nc(\mcC)|$ is equivalent
to the classifying space of the positive boundary subcategory $\mcC^\pb \subset \mcC$,
which will often be computable.

\begin{defn}
    For a labelled cospan category $(\mcC \to \Csp)$ we define the 
    \emph{positive boundary category} $\mcC^\pb$ as the subcategory of $\mcC$
    that contains all objects, but only those morphisms $W:M \to N$
    for which $\pi(N) \to \pi(W)$ is surjective.
\end{defn}

Inspecting definition \ref{defn:labelled-cospans} we see that the restricted functor $\mcC^\pb \to \Csp$ exhibits $\mcC^\pb$ again as labelled cospan category.
In this labelled cospan category $\Hom_{\mcC^\pb}(1, 1) = \{\id_1\}$ is trivial and every morphism is reduced.

\begin{defn}\label{defn:admits-surgery}
    We say that a labelled cospan category $(\mcC \to \Csp)$ \emph{admits surgery} 
    if one can pick the following data:
    \begin{itemize}
        \item A connected object $O \in \mcC$ and a connected morphism $T:1_\mcC \to O$.
        \item For any connected object $A \in \mcC$ a connected morphism $P_A:A \to O \ot A$.
    \end{itemize}
    subject to the condition that for any two connected objects $A,B \in \mcC$
    and any two arbitrary objects $M, N \in \mcC$
    the following diagrams commute for all connected morphisms 
    $U:A \ot M \to B \ot N$,
    $V:A \ot B \ot M \to N$, and 
    $W:M \to A \ot B \ot N$:
    \[
        \begin{tikzcd}[column sep = 3pc]
            A \ot M \ar[d, "U"] \ar[r, "P_A \ot \id_M"] &
            O \ot A \ot M \ar[d, "\id_O \ot U"] \\
            B \ot N \ar[r, "P_B \ot \id_N"] &
            O \ot B \ot N
        \end{tikzcd}
        \quad
        \begin{tikzcd}[column sep = 5pc]
            A \ot B \ot M 
            \ar[d, "(\gb_{O,A} \ot \id_{B \ot M}) \circ (\id_A \ot P_B \ot \id_M)"] 
            \ar[r, "P_A \ot \id_{B \ot M}"] &
            O \ot A \ot B \ot M \ar[d, "\id_O \ot V"] \\
            O \ot A \ot B \ot M \ar[r, "\id_O \ot V"] &
            O \ot N
        \end{tikzcd}
    \]
    \[
        \begin{tikzcd}[column sep = 3pc]
            M \ar[d, "W"] \ar[r, "W"] &
            A \ot B \ot N 
            \ar[d, "(\gb_{A,O} \ot \id_{B \ot N}) \circ (\id_A \ot P_B \ot \id_N)"] \\
            A \ot B \ot N \ar[r, "P_A \ot \id_{B \ot N}"] &
            O \ot A \ot B \ot N.
        \end{tikzcd}
    \]
    Here $\beta_{X,Y}\colon X \otimes Y \cong Y \otimes X$ is the (symmetric) braiding in $\mcC$.
    See figure \ref{fig:admits-surgery} for an illustration of these three equations.
\end{defn}

\begin{rem}
    Note that the morphism $T$ is not required to satisfy any equation.
    (In fact, we do not even need it to be connected.)
    Nevertheless, it is crucial that there exists a morphism from $1_\mcC$ to $O$
    as we use it to introduce new connected components during the surgery.
    This would for example be impossible in the negative boundary category 
    $\Cob_d^{\partial_-} = (\Cob_d^{\partial_+})^\op$, which is also a labelled cospan category.
    This category cannot satisfy the conclusion of the surgery theorem as 
    \[
        \pi_0 B((\Cob_d^{\partial_-})^{\partial_+})
        \longrightarrow
        \pi_0 |N_\cd^\nc (\Cob_d^{\partial_+})|
        = \pi_0 B(\Cob_d^{\partial_+}) \cong \gO_{d-1}^{\rm or}
    \]
    is not injective because in $(\Cob_d^{\partial_-})^{\partial_+}$
    there are no non-identity morphisms to or from the empty set.
\end{rem}

\begin{rem}
    The string calculus that we use throughout the paper, and for example in figure \ref{fig:admits-surgery}, to depict composite morphisms in labelled cospan categories can be read as follows.
    Each of the six diagrams in figure \ref{fig:admits-surgery} depicts a morphism going from left to right.
    Each line correspond to an object and the source and target of the morphism can be read off by tensoring the incoming and outgoing lines top to bottom.
    The objects $M$ and $N$ are represented by multiple lines as they are tensors of a number of connected objects.
    The purple dots represent morphisms whose source and target are again the tensor product of the incoming and outgoing lines.
    For example, the braiding $\beta_{A,O}\colon A \otimes O \cong O\otimes A$, which is represented by a small purple dot, has as input a line labelled by $A$ on top of a line labelled by $O$, and its outputs are the same lines reversed.
\end{rem}

\begin{figure}[ht]
    \centering
    \def\svgwidth{.7\linewidth}
    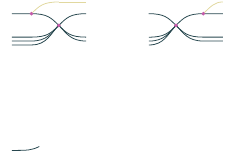
    \caption{The three conditions that the morphism $P_A:A \to O \ot A$ has to satisfy 
    for all connected objects $B$, arbitrary objects $M$, $N$, and connected morphisms
    $U$, $V$, $W$.}
    \label{fig:admits-surgery}
\end{figure}

The purpose of this section is to prove the surgery theorem that relates the 
space of ``no compact components" with the positive boundary subcategory.
\begin{thm}[Surgery theorem]\label{thm:Csp-surgery}\label{thm:surgery}
    For any labelled cospan category $\mcC$ that admits surgery the inclusion 
    $N\mcC^\pb \to N_\cd^\nc\mcC$ of simplicial sets induces an equivalence
    \[
        B(\mcC^\pb) \simeq |N_\cd^\nc\mcC|.
    \]
\end{thm}

Combining this with the decomposition theorem we have the following immediate consequence:
\begin{cor}[Theorem \ref{theorem:decomposition-and-surgery}]\label{cor:decomposition-and-surgery}
    If the labelled cospan category $(\mcC \to \Csp)$ admits surgery
    and $B(\mcC)$ is group-like, then there is a homotopy fiber sequence of infinite loop spaces
    \[
        B(\mcC^\pb) \longrightarrow B(\ICsp(\mcC)) \longrightarrow
        Q\left(\bigvee\nolimits_{g \in G} S^2(B\mcF_g)\right).
    \]
\end{cor}

    The condition of admitting surgery might seem unreasonably strong
    if one thinks about a general labelled cospan category $(\mcC \to \Csp)$.
    We are essentially requiring that $P_A: A \to O \ot A$ 
    commutes with all other operations. 
    However, it turns out to satisfied in the cases we are interested in.

\begin{ex}
    We will now show that many geometric examples admit surgery.
    Let, for example, $\mcC = \Cob_d$ be the $d$-dimensional cobordism category
    for $d \ge 2$. Then we can let $O := S^{d-1}$ be the sphere and 
    $T := D^d: \emptyset \to S^{d-1}$ the disk.
    For the morphism $P_M: M \to S^{d-1} \amalg M$ we can choose the 
    connected sum of the cobordisms $D^d:\emptyset \to S^{d-1}$ 
    and $M\times [0,1]:M \to M$.
    Since both cobordisms are oriented and connected it does not matter at which
    point we take the connected sum.
    We need to make sure that this makes the relevant diagrams commute.
    For example, for all connected morphisms $W:A \amalg M \to B \amalg N$
    one checks that both ways of going around the diagram 
    can be described as the connected sum of $W$ with the cobordism 
    $D^{d-1}:\emptyset \to S^{d-1}$. Since $W$ is connected any 
    two such connected sums are diffeomorphic.

    We can also replace the morphism $T:\emptyset \to S^{d-1}$ by 
    any other connected manifold with boundary $S^{d-1}$.
    In particular for $d=2$ we can choose a genus $g \ge 1$ surface 
    with one boundary component, which implies that $\Cobn \subset \Cob_2$
    also admits surgery.
\end{ex}

\begin{ex}
    As similar argument as above also applies to the weighted cospan categories
    $\Csp(A, A_1, \ga)$ from definition \ref{defn:weighted-cospans}.
    One can choose $O := *$ to be the singleton, $T:\emptyset \to *$
    the cospan $(\emptyset \to * = *)$ weighted by any element in $A_1$ (for example $\ga$),
    and $P: * \to * \amalg *$ the cospan $(* \to * \leftarrow * \amalg *)$
    weighted by the unit $0 \in A$.
    This makes three squares in definition \ref{defn:admits-surgery} commute: 
    in each case the composite is a connected morphisms (uniquely determining it as a morphism in $\Csp$) and so it suffices to compare the labels.
    In the first square the label is that of $U$, in the second that of $V$ and in the third that of $W$.

    By lemma \ref{lem:weighted-Csp}, example \ref{ex:Cob2g<gc}, and example \ref{ex:Cob_d^Sd-1}
    this means that all of the following labelled cospan categories admit surgery:
    \[
        \Csp, \quad 
        \Cobn, \quad 
        \Cob_{d}^{S^{d-1}}, \quad 
        \Cob_2^{g < \gc}, \qand 
        \Cob_2^{\chi\le0, g < \gc}.
    \]
\end{ex}

\begin{ex}
    If a labelled cospan category $\mcC$ admits surgery,
    has a group-complete classifying space $B\mcC$,
    and satisfies that $\mcF_g(\mcC)$ is contractible for all $g:1_\mcC \to 1_\mcC$,
    then the inclusion of the positive boundary subcategory induces an equivalence:
    \[
        B(\mcC^\pb) \simeq B(\ICsp(\mcC)).
    \]
    We will show that this is the case for $\Csp$ and $\Cob_d$ with $d \ge 2$
    in section \ref{subsec:mcF-contractible}. This implies Theorem \ref{theorem:BCobd}
    and together with the computations $B\Csp^\pb \simeq *$ 
    and $B\Cob_2^\pb \simeq S^1$ made in proposition \ref{prop:computing-pb}
    it implies Theorem \ref{theorem:BICsp} and Theorem \ref{theorem:ICob2}.
\end{ex}

\subsection{Strategy}
The proof of the surgery theorem will be given by implementing a surgery 
similar to the one described in \cite[section 4.2, figure 3]{Gal11} 
using language similar to that of \cite[section 3 and 6]{GRW14}.
For each $W \in N_\cd^\nc\mcC$ and point $p \in W(i \le i+1)$
there is either a path connecting it to $W(0)$ or to $W(n)$.
We will use this path to do a surgery that connects $p$ to $W(0)$ or $W(n)$.
One of the key problems will be that there is not always a unique or canonical surgery,
but rather a contractible space of choices. We will have to make sure that any 
combination of surgeries commutes.

To construct the homotopies that are crucial to the proof of theorem \ref{thm:surgery}
it will be extremely useful to work with a variant of $\mcC$ where objects 
are simply subsets of some big background set $\gO$.

\begin{defn}
    Fix a set $\gO$ and a map $F:\gO \to \Obj^\con(\mcC)$ such that every 
    connected object $M \in \mcC$ has infinitely many preimages in $\gO$.
    For any finite subset $A \subset \gO$ we let $F(A) := \bigotimes_{a \in A} F(a)$.
    (See definiton \ref{defn:tensor-over-set}.)
    
    The category $\mcC^\gO$ has as objects finite subset of $\gO$
    and a morphism $W:A \to B$ is defined as a morphism $W:F(A) \to F(B)$ in $\mcC$.
\end{defn}

In what follows we will often use $n$-simplices in the nerve of $\mcC^\gO$
so it will be useful to fix some notation.
An element of $N_n\mcC^\gO$ is a functor $W:[n] \to \mcC^\gO$ 
sending any $0 \le i \le n$ to an object $W(i) \in \mcC^\gO$, 
i.e.\ a finite subset $W(i) \subset \gO$, 
and any $0 \le i \le j \le n$ to a morphism $W(i\le j): W(i) \to W(j)$.
We will denote the underlying cospan of $W(i\le j)$ by 
$W(i) \to W_\pi(i \le j) \leftarrow W(j)$.

\begin{defn}
    We let $C_\cd^\nc \subset N_\cd \mcC^\gO$ be the simplicial subset containing
    those $n$-simplices $W:[n] \to \mcC^\gO$ where $W(0 \le n)$ 
    has no closed components,
    i.e.\ those $W$ for which 
    $W(0) \amalg W(n) \to W_\pi(0 \le n)$ is surjective.
    
    Moreover, we let $C_\cd^\pb \subset N_\cd \mcC^\gO$ be the simplicial subset
    containing those $W:[n] \to \mcC^\gO$ where $W(k \le k+1)$ 
    is positive boundary for all $k$, i.e.\ where
    $W(k+1) \to W_\pi(k \le k+1)$ is surjective for all $k$.
\end{defn}

\begin{lem}
    With the usual notions of disjointness and disjoint union
    $\mcC^\gO$ is a partial commutative monoidal category
    and the functor $F:\mcC^\gO \to \mcC$ is an equivalence of categories.
    Moreover, $F$ induces equivalences:
    \[
        |C_\cd^\pb| \simeq B\mcC^\pb
        \qand
        |C_\cd^\nc| \simeq |N_\cd^\nc\mcC| .
    \]
\end{lem}
\begin{proof}
    Since the objects of $\mcC^\gO$ are simply finite subsets of $\gO$ 
    it is clear that the notion of disjointnes and disjoint union behave as expected.
    For $\mcC^\gO$ to satisfy definition \ref{defn:PCMC} we need to 
    check that for any finite sequence of objects $A_1,\dots, A_n \subset \gO$
    one can find isomorphic replacements $A_1', \dots, A_n' \subset \gO$ 
    that are pairwise disjoint. 
    We required $F:\gO \to \Obj^\con(\mcC)$ to have infinite preimages for this exact purpose.
    Starting with $A_2$ we can find for each $a \in A_i$ an element $a' \in \gO$
    with $F(a) = F(a')$ and $a'$ disjoint to all previous elements.
    Then $a$ and $a'$ are isomorphic via the morphism  $\id_{F(a)}:a \to a'$.
    Moreover, they are all disjoint by construction.
    
    The functor $F$ is fully faithful by construction and it is essentially surjective
    because every object in the labelled cospan category $\mcC$ can be written
    as a product of connected objects.
    Therefore $F$ is an equivalence of categories and it also restricts
    to an equivalence between $(\mcC^\gO)^\pb$ and $\mcC^\pb$.
    Since $C_\cd^\pb \cong (\mcC^\gO)^\pb$ we see that $|C_\cd^\pb| \simeq B\mcC^\pb$.
    The equivalence $|C_\cd^\nc| \simeq |N_\cd^\nc\mcC|$ follows from 
    lemma \ref{lem:Nnc-invariant}.
\end{proof}

\begin{rem}
    We have so far treated $W_\pi(k \le l)$ as if it were a well-defined set,
    but one has to be careful with this. In general, for a morphism $W(k \le l)$
    we only have an isomorphism class of cospans $[W(k) \to W_\pi(k \le l) \leftarrow W(l)]$.
    It does therefore not really make sense to talk about elements of $W_\pi(k \le l)$.
    However, if the $n$-simplex $W:[n] \to \mcC^\gO$ lies in $C_\cd^\nc \subset N_\cd\mcC^\gO$,
    then for each $k\le l$ the map $W(k) \amalg W(l) \to W_\pi(k \le l)$ is surjective and hence $W_\pi(k \le l)$ is in canonical bijection with the quotient of $W(k) \amalg W(l)$ by the relation induced by the cospan.
    (In other words, the groupoid $N_n^\nc \ICsp(\mcC)$ has contractible components and thus $N_\cd^\nc \ICsp(\mcC) \to N_\cd^\nc \mcC^\Omega$ is a level-wise equivalence.)
    As this is the only case that features in this section, we will confidently speak about elements of $W_\pi(k\le l)$.
\end{rem}

\subsection{Surgery data}

Given some $n$-simplex $W \in C_n^\nc$ that is not yet in $C_n^\pb$ there must
be parts of $W$ that are not positive boundary. 
Concretely, we can find an element $u \in W_\pi(i\le i+1)$ that is not in the 
image of $W(i+1) \to W_\pi(i \le i+1)$.
We would like to find a homotopy that homotops this to a simplex in $C_n^\pb$.
To do so we first choose a ``path" from $u$ to $W(0)$ or $W(n)$
along which we can introduce a morphism that connects $U$ to a positive boundary.
This is the datum of a surgery path.
To make sense of the idea of a path we introduce the following notion
of representing space:

\begin{defn}\label{defn:|W|}
    For an $n$-simplex $W \in C_n^\nc$ we define the representing space $|W|$ as
    \[
        |W| = \left(\{ (t, k, \go) \in [0,1] \times \{0,\dots, n\} \times \gO \;|\; 
                \go \in W(k) \} 
                \amalg \{\bot, \top\}
                \right)/\sim
    \]
    The equivalence relation $\sim$ is defined as follows:
    for $k \in \{0,\dots, n\}$ it identifies two elements of 
    $\{ (t, i, \go) \in |W| \;|\; t+i = k\} \cong W(k-1) \amalg W(k)$
    whenever they are mapped to the same element under 
    $W(k-1) \amalg W(k) \to W_\pi(k-1\le k)$.
    Moreover, we set $(0,0,\go_0) \sim \bot$ and $(1,n,\go_n) \sim \top$ 
    for all $\go_0 \in W_\pi(0)$ and $\go_n \in W_\pi(n)$.
    
    This space naturally comes with a continuous projection $\pr:|W| \to [0,n+1]$
    sending $[t,k,\go]$ to $t+k$, $\bot$ to $0$, and $\top$ to $n+1$.
    We let $|W|_0 \subset |W|$ be the finite subset $\pr^{-1}(\{1,\dots,n\})$.
\end{defn}

\begin{rem}\label{rem:|W|}
    We can also think of $|W|$ as a long sequence of pushouts:
    \[
\begin{tikzcd}[column sep =-2.5pc, row sep =1pc]
  &                                              & {W(0) \times [0,1]} &                                              &                            &                                              & {W(1) \times [1,2]} &                            & {W(n) \times [0,n+1]} &                                                &   \\
* &                                              &                         &                                              & W_\pi(0\le 1) \times \{1\} &                                              &                         & \ \qquad \dots \qquad\                      &                           &                                                & * \\
  & W(0) \times \{0\} \arrow[lu] \arrow[ruu] &                         & W(0) \times \{1\} \arrow[luu] \arrow[ru] &                            & W(1) \times \{1\} \arrow[lu] \arrow[ruu] &                         & {\dots} \arrow[luu] \arrow[ruu] &                           & W(n) \times \{n+1\} \arrow[ru] \arrow[luu] &  
\end{tikzcd}
    \]
    The finite subset $|W|_0 \subset |W|$ is exactly the image of the sets
    $W_\pi(i-1\le i) \times \{i\}$ for $1 \le i \le n$.
    We can therefore define a map $g:|W|_0 \to \Mor^\con(\mcC)$
    that sends each $(i,k,\go)$ with $k+i \in \{1,\dots,n\}$ to the connected morphism
    in $\mcC$ that represents the relevant component of $W_\pi(k+i-1 \le k+i)$.
    This is exactly the labelling of the cospan in the enhanced $\ICsp(\mcC)$-model.
    The datum of $W$ together with $g$ in fact uniquely encodes 
    the $n$-simplex $W \in C_n^\nc$.
    Motivated by this we will often be drawing pictures of 
    $(|W|, g:|W|_0 \to \Mor^\con(\mcC))$ as in figure \ref{fig:simplified-presentation}.
\end{rem}

\begin{figure}[ht]
    \centering
    \small
    \def\svgwidth{\figurerescalefactor\linewidth}
    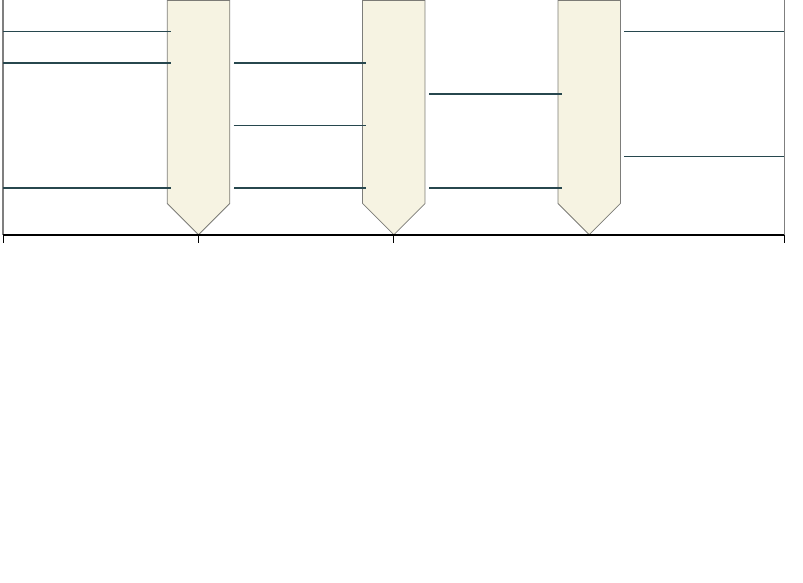
    \caption{On top: a $3$-simplex $W \in C_3^\nc$ for the labelled cospan category 
    $\mcC = \Cob_2$. 
    Below: The representing space $|W|$ with labels in $\IN$ recording the genus
    and a possible surgery path.}
    \label{fig:simplified-presentation}
\end{figure}

\begin{defn}\label{defn:surgery-path}
    For $W \in C_n^\nc$ a \emph{surgery path} is a continuous path $p: [0,1] \to |W|$
    such that the composite $p_\IR: [0,1] \to |W| \to [0,n+1]$ is piece-wise linear,
    $p_\IR(0) \in \{0,n+1\}$, and $p_\IR(1) \in \{0, 1, \dots, n+1\}$.
    
    If the integer $i := p_\IR(1)$ is in $\{1,\dots,n\}$
    then the end-point of the path corresponds 
    to some component of the morphism $W(i-1 \le i)$, which we denote by $p_{\rm end} \in W_\pi(i-1 \le i)$.
\end{defn}

Of course one surgery path by itself is not sufficient to make some $W$ be positive boundary.
We need several surgery paths at once, which we collect in a surgery datum:
\begin{defn}\label{defn:surgery-datum}
    A surgery datum for $W \in C_n^\nc$ is a finite subset 
    $A \subset \gO_O := F^{-1}(O)$,
    disjoint to $\cup_{i=0}^n W(k) \subset \gO$,
    together with the choice of a surgery path $p^\ga:[0,1] \to |W|$ for all $\ga \in A$
    such that they satisfy the following connectivity condition:
   \begin{itemize}
       \item[]
    For any $i\in \{1,\dots, n\}$ and $u \in W_\pi(i-1\le i)$
    either $u$ is in the image of $W(i) \to W_\pi(i-1 \le i)$,
    or there is an $\ga \in A$ such that the path $p^\ga$ ends at $u = p_{\rm end}^\ga$.
   \end{itemize} 
    For a surgery datum $(p^\ga)_{\ga \in A}$ we let 
    $A^{\rm in} := \{\ga \in A \;|\; p_\IR^\ga(0) = 0\} \subset A$ 
    denote the subset of those $\ga$ where 
    the surgery path starts at the incoming boundary.
\end{defn}

Note that the empty collection is a surgery datum for $W$ if and only if $W$ 
already lies in $C_n^\pb$. In that case we do not need to do any surgery.
We now describe how the surgery data behaves with respect to the simplicial structure:

\begin{defn}\label{defn:surgery-datum-functoriality}
    For any injective morphism $\gl:[n] \to [m]$ in $\gD$ we define a map
    \[ 
    \gl^*:\{0, \dots, m+1\} \to \{0, \dots, n+1\} \quad
    k \mapsto \min\{j \;|\; \gl(j) \ge k\}
    \]
    where we set $\gl^*(0) = 0$ and $\gl^*(m+1) = n+1$ by convention.
    Extending affine linearly we obtain the piece-wise linear map 
    $\gl^\cd:[0,m+1] \to [0,n+1]$.
    This defines a functor $\gD_{\rm inj}^\op \to \Top$.
    
    For an $m$-simplex $W \in C_m^\nc$ we let $\gl_W:|W| \to |\gl^*W|$ denote 
    the unique continuous map that satisfies $\gl_W([t,\gl(k),\go]) = [t,k,\go]$ for
    all $[t,k,\go] \in |\gl^*W|$, $\gl_W(\bot) = \bot$, $\gl_W(\top) = \top$,
    and makes the following diagram commute:
    \[
    \begin{tikzcd}
        {|W|} \ar[r, "\gl_W"] \ar[d, "\pr"]    &
        {|\gl^*W|} \ar[d, "\pr"]  \\
        {[0,m+1]} \ar[r, "\gl^\cd"] & 
        {[0,n+1]}
    \end{tikzcd}
    \]
\end{defn}

The above definition indeed uniquely characterizes $\gl_W$:
the equation $\gl_W([t,\gl(k),\go]) = [t,k,\go]$ defines $\gl_W$ on all points $[t,i,\go]$ where $i$ is in the image of $\gl$,
and if $i$ is not in the image of $\gl$, then $\gl^\cd(i) = \gl^\cd(i+t) = \gl^\cd(i+1)$ for all $t \in [0,1]$, so the commutative square forces $\gl_W$ to be constant on the edge $[0,1] \times \{i\} \times \{\go\}$
Because $W$ has no closed components this recursively determines the value of $\gl_W$ on all points.
We can use this unique characterization to check that $\gl_W$ is functorial, i.e.~that if $\rho: [l] \to [n]$ is another injective morphism in $\Delta$, then $(\gl\circ \rho)_W = \rho_{\gl^*W} \circ \gl_W$ are the same map $|W| \to |\rho^*\gl^*W|$.
(This holds because either map satisfies both the equations and the commutative diagram, as $(\rho \circ \gl)^\cd = \rho^\cd \circ \gl^\cd$.)

\begin{defn}
    The semisimplicial set $C_\cd^\gs$ has as $n$-simplices pairs of an $n$-simplex
    $W \in C_n^\nc$ with a surgery datum $(p^\ga)_{\ga \in A}$ on $W$.
    The semisimplicial structure is defined by 
    \[
        \gl^*(W,(p^\ga)_{\ga \in A})
        := (\gl^*W, (\gl_W \circ p^\ga)_{\ga \in A}).
    \]
\end{defn}

\begin{lem}
    If $(p^\ga)_{\ga \in A}$ is a surgery datum for $W$ 
    then $(\gl_W \circ p^\ga)_{\ga \in A}$ is a surgery datum for $\gl^*W$.
    Therefore $C_\cd^\gs$ is a well-defined semisimplicial set.
\end{lem}
\begin{proof}
    Since $\gl_W:|W| \to |\gl^*W|$ is continuous the new path 
    $\gl_W \circ p: [0,1] \to |W| \to |\gl^*W|$ is also continuous.
    Moreover, $\gl^\cd:[0,n+1] \to [0,m+1]$ is a piece-wise linear map and hence 
    $(\gl_W\circ p)_\IR = \gl^\cd \circ p_\IR$ is piece-wise linear, too.
    The condition $p_\IR(0) \subset \{0,n+1\}$ is preserved because 
    $\gl^\cd:[0,n+1] \to [0,m+1]$ preserves the minimal and the maximal element.
    The condition $p_\IR(1) \in \{0,\dots,n+1\}$
    implies $(\gl^\cd \circ p_\IR)(1) \in \{0,\dots,m+1\}$
    since $\gl^\cd$ sends integers to integers.
    Therefore $\gl_W \circ p$ is indeed still a surgery path.
    
    Next we need to check that $(\gl_W\circ p^\ga)_{\ga \in A}$ is a surgery datum
    for $\gl^*W$. It is clear that $A$ is disjoint to $\gl^*W$ as we 
    can only forget elements of $\gO$ when passing from $W$ to $\gl^*W$.
    To check the connectivity condition let us assume that $\gl$ 
    is the unique morphism $\gd^i:[n] \to [n+1]$ whose image does not contain $i$. 
    Write $d_iW := (\gd^i)^*W$. The general case follows as the category 
    $\gD_{\rm inj}$ is generated by the $\gd^i$.
    Consider some $u \in (d_iW)_\pi(j-1\le j)$ that is not in the image of $(d_iW)_\pi(j)$.
    If $j \neq i$ it is clear from the connectivitiy condition for $W$ 
    that there is an $\ga \in A$ such that $p^\ga$ ends in $p_{\rm end}^\ga = u$.
    We may therefore assume that $j=i$, in which case $(d_iW)(i-1\le i)$
    is the composite of $W(i-1\le i)$ and $W(i\le i+1)$
    and the set $(d_iW)_\pi(i-1\le i)$ is the pushout 
    $W_\pi(i-1 \le i) \amalg_{W(i)} W_\pi(i \le i+1)$.
    
    There are two cases: either $u \in (d_iW)_\pi(i-1 \le i)$ 
    is not hit by an element of $W_\pi(i \le i+1)$ or it is.
    In the former case $u$ is represented by an element $v \in W_\pi(i-1 \le i)$
    that is not in the image of $W(i)$. 
    The connectivity condition for $(p^\ga)_{\ga \in A}$ 
    implies that there is $\ga \in A$ such that the surgery path $p^\ga$ 
    ends at $p_{\rm end}^\ga = v$. 
    Consequently $\gd^i_W \circ p^\ga$ ends at 
    $(\gd^i_W \circ p^\ga)_{\rm end} = \gd^i_W(p^\ga_{\rm end}) = u$ and we are done.
    In the other case we can represent $u$ by an element $v' \subset W_\pi(i \le i+1)$.
    This $v'$ cannot lie in the image of $W(i+1) \to W_\pi(i \le i+1)$,
    as otherwise $u$ would be in the image of $(d_iW)(i) \to (d_iW)_\pi(i-1\le i)$,
    which we assumed not to be the case.
    This means that $v$ does not have positive boundary as part of $W$ 
    and so we can find $\ga \in A$ such that $p^\ga$ ends in $p_{\rm end}^\ga = v$,
    and then $(\gd^i_W \circ p^\ga)_{\rm end} = \gd^i_W(p_{\rm end}^\ga) = u$.
    
    This shows that $(\gl_W \circ p^\ga)_{\ga \in A}$ is a surgery datum for $\gl^*W$
    and hence the face operators $\gl^*$ yield well-defined elements of $C_{m}^\gs$.
    Functoriality follows from the argument below definition \ref{defn:surgery-datum-functoriality} and hence this assembles into a well-defined functor $C_\cd^\gs: \gD_{\rm inj}^\op \to \Set$.
\end{proof}

\subsection{The basic surgery}
We now want to describe how to do the surgery given a surgery datum.
The problem of choosing this surgery datum will be dealt with
in the next subsection.
In formulas this means that we want to construct a homotopy from 
the forgetful map $|C_\cd^\gs| \to |C_\cd^\nc|$ to another map that lands
in the subspace $|C_\cd^\pb| \subset |C_\cd^\nc|$.
Moreover, we would like this homotopy to be constant in the case of an empty surgery datum.
This is summarized in the following diagram:
\[
\begin{tikzcd}[column sep = 4pc, row sep = 2pc]
    {|C_\cd^\pb|} \ar[d, equal] \ar[r]     &
    {|C_\cd^\gs|} \ar[d, "\mcS_0"] \ar[ld, "\mcS_1"', dashed]  \\
    {|C_\cd^\pb|} \ar[r]     &
    {|C_\cd^\nc|} 
\end{tikzcd}
\]
Here $\mcS:|C_\cd^\gs| \times [0,1] \to |C_\cd^\nc|$ is a homotopy, 
which we think of as a continuous family of maps $\mcS_r$ indexed 
by $r \in [0,1]$ such that $\mcS_0$ is the forgetful map and $\mcS_1$
is a retraction onto the subspace $|C_\cd^\pb|$.

The homotopy $\mcS$ will be obtained by concatenating two homotopies $\rho$ and $K$.
With $\rho$ we introduce a disjoint copy of $\id_\ga$ for all $\ga \in A^{\rm in}$
as pictured in figure \ref{fig:introducing-omega}.
This is based on the idea of how the morphism $T: 1_\mcC \to O$ 
yields a natural transformation $\Id_{\mcC} \Rightarrow O \ot \Id_\mcC$,
which in turn induces a homotopy of maps $B\mcC \to B\mcC$.
The second homotopy $K$ will be more complicated as we have to 
follow a surgery path to move a copy of the morphism $P_M: M \to O \ot M$
from $W(0)$ or $W(n)$ to the morphism at the end-point of the surgery path,
in order to make that morphism positive boundary; see figure \ref{fig:surgery-homotopy}.

\begin{defn}
    We let $\gO_O := F^{-1}(O) \subset \gO$.
    For any $\ga \in \gO_O$ and $\go \in \gO$ with $\go \neq \ga$ we define
    morphisms $T_\ga:\emptyset \to \{\ga\}$ 
    and $P_{\go,\ga}: \{\go\} \to \{\go,\ga\}$ in $\mcC^\gO$ by:
    \[
        T_\ga := (F(\emptyset) = 1_\mcC \xrightarrow{T} O = F(\ga)) 
        \qand
        P_{\go,\ga} := (F(\go) \xrightarrow{P_{F(\go)}} O \ot F(\go) \cong F(\{\ga,\go\})) .
    \]
\end{defn}

We first check that it follows from the axioms in definition \ref{defn:admits-surgery} that two copies of $P$ always commute.
\begin{lem}\label{lem:surgery-paths-commute}
    For all $\ga,\ga'\in \gO_O$ and $\go\in \gO$, all three distinct, we have
    \[
        (P_{\go,\ga'} \amalg \id_\ga) \circ P_{\go, \ga} = (P_{\go,\ga} \amalg \id_{\ga'}) \circ P_{\go,\ga'}.
    \] 
\end{lem}
\begin{proof}
    This follows by combining the top left and bottom diagram in definition \ref{defn:admits-surgery}, as indicated in figure \ref{fig:sigmas-commute}.
    Note that in writing the morphism in $\mcC^\gO$ we do not have to write the braiding $\beta_{O,O}$, as the morphism $(P_{\go,\ga} \amalg \id_{\ga'}) \circ P_{\go',\ga}$ is already a morphism from the set $\{\go\}$ to the set $\{\go,\ga',\ga\} = \{\go,\ga,\ga'\}$ and when interpreting this as a morphism in $\mcC$ we apply $F$ which uses the order on $\{\go,\ga,\ga'\}$ induced from whatever order we had picked on $\gO$.
\end{proof}

\begin{figure}[ht]
    \centering
    \def\svgwidth{\figurerescalefactor\linewidth}
    \small
    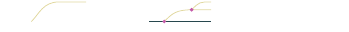
    \caption{When attaching two arcs, the order does not matter.}
    \label{fig:sigmas-commute}
\end{figure}

The first half of the basic surgery will be a semisimplicial map, but the second half will not; rather it is only defined on geometric realizations, so we fix some notation.
Recall that the space $|C_\cd^\nc|$ is a quotient of the disjoint union
\[
    \coprod_{n\ge 0} C_n^\nc \times |\gD^n| \longrightarrow |C_\cd^\nc|.
\]
We will parametrise the space $|\gD^n|$ as 
\[
    |\Delta^n| = \{ \mbf{t} = (t_1,\dots,t_{n}) \in [0,1]^{n} \;|\; 0 \le t_1 \le \dots \le t_n \le 1 \}
\]
and by convention we set $t_0 = 0$ and $t_{n+1} = 1$.
A point in $|C_\cd^\nc|$ can be represented by $(W, \mbf{t})$ 
where $W:[n] \to \mcC^\gO$ and $\mbf{t}$ is as above.

The first step in the surgery is to homotope $\mcS_0$ to a map $\mcS_{1/2}$ that introduces an identity morphism for each surgery path that starts at the incoming boundary.
This does not really interact with the existing morphisms in $W$ and so we can simply write is as a disjoint union.
To make sense of this, note that if $A, B \subset \gO$ are disjoint subsets and we write $C_\cd^{\nc,\subset A}, C_\cd^{\nc, \subset B} \subset C_\cd^{\nc}$ for the simplicial subsets where we require the simplices to only use elements of $A$ or $B$, respectively, then taking the union defines a map of simplicial sets
\[
    \cup: C_\cd^{\nc,\subset A} \times C_\cd^{\nc, \subset B} \longrightarrow C_\cd^{\nc, \subset A \amalg B}
\]
and hence a map on geometric realizations.

\begin{defn}\label{defn:rho}
    For $A \subset \gO$ finite we define a homotopy as the composite
    \[
        \rho_A: |C_\cd^{\nc, \subset \gO\setminus A}| \times [0,1] 
        \xrightarrow{\id \times T_A} |C_\cd^{\nc, \subset \gO\setminus A}| \times |C_\cd^{\nc, \subset A}|
        \xrightarrow{\cup} |C_\cd^\nc|
    \]
    where $T_A$ denotes the inclusion of the $1$-simplex $T_A = \coprod_{\ga \in A} (T_\ga: \emptyset \to \{\ga\})$.
    We will denote the value of this map on some point $([W,\bbt],r)$ by $\rho_A^r(W,\bbt)$.
\end{defn}

\begin{rem}
    If we unpack the details of the definition, we obtain a description of the homotopy on representatives as
    \[
        \rho_A^r(W, \bbt) = (W^*, \bbt^*)
    \]
    where $\bbt^*$ is obtained by inserting $r$ as $(t_1,\dots, t_i, r, t_{i+1}, \dots, t_n)$ such that $r \in [t_i,t_{i+1}]$ and $W^*$ is the $(n+1)$-simplex with 
    \begin{align*}
        W^*(j) &= \begin{cases}
            W(j) & \text{ for } j \le i \\
            W(j-1) \amalg A & \text{ for } j \ge i+1
        \end{cases}
        \qand\\
        W^*(j\le j+1) &= \begin{cases}
            W(j \le j+1) & \text{ for } j < i \\
            \id_{W(j)} \amalg\ \coprod_{\ga \in A} T_\ga & \text{ for } j = i \\
            W(j-1 \le j) \amalg \id_{A} & \text{ for } j \ge i+1.
        \end{cases}
    \end{align*}
\end{rem}

The semisimplicial set $C_\cd^\gs$ decomposes as a disjoint union $\coprod_{A \subset \gO, \text{ finite}} C_\cd^{\gs,A}$ where in $C_\cd^{\gs,A}$ we require that the surgery data is labelled by $A$.
In particular, if $(W, p^\cd) \in C_n^{\gs, A}$ then $W$ must in fact all be in $C_n^{\nc, \subset \gO \setminus A}$ and we have a union map with simplices in $C_\cd^{\nc, \subset A}$.
This allows us to take the union of some $(W, (p^\ga)_{\ga \in A}) \in C_\cd^\gs$ with the $1$-simplices $T_\ga$ for $\ga \in A^{\rm in}$.

\begin{defn}\label{defn:rho-mcS}
    We define the first half of the basic homotopy as
    \begin{align*}
        \mcS: |C_\cd^\gs| \times [0,1/2] &\longrightarrow |C_\cd^\nc| \\
        ([W, (p^\ga)_{\ga \in A}, \bbt],r) &\longmapsto \rho_{A^{\rm in}}^{1-2r}(W, \bbt).
    \end{align*}
\end{defn}

    In particular, $\mcS_{1/2}$ is the realization of the semisimplicial map that on $n$-simplices is given by
    \[
        (W, (p^\ga)_{\ga \in A}) \longmapsto (W \amalg \id_{A^{\rm in}}).
    \]

Now that we have the homotopy that introduces the cylinder $\id_\ga$ we proceed to describe a similar, but more complicated homotopy that homotops this cylinder along a surgery path to make a certain part of $W$ positive boundary.

\begin{figure}[ht]
    \centering
    \def\svgwidth{\figurerescalefactor\linewidth}
    \small
    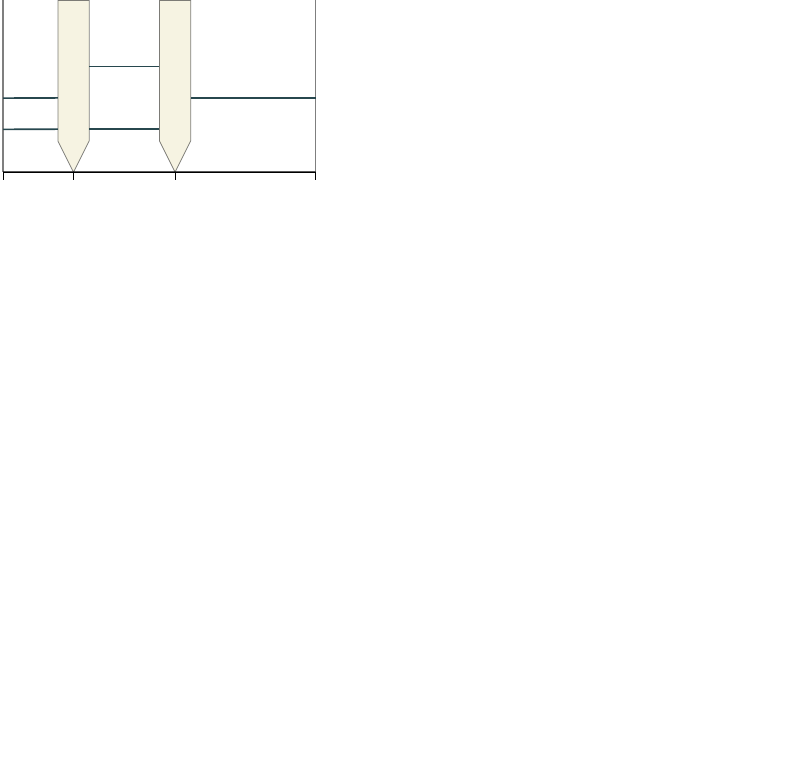
    \caption{A slide show of the homotopy $\rho_\ga^s(W, \bbt)$ for a $2$-simplex $W$
    as $s$ moves from $1$ to $0$.}
    \label{fig:introducing-omega}
\end{figure}

\begin{defn}\label{defn:K-new}
    For $W \in C_n^\nc$, $A \subset \Omega$ disjoint to $W$, $x: A \to W$ a map specifying attachment points,  and $\bbt \in |\Delta^n|$ we will define
    \[
        \gs(W,\bbt,A,x) := [W', \bbt'] \in |C_\cd^\nc|
    \]
    as follows.
    Write $x_\IR := \pr_\IR \circ x: A \to [0,n+1]$ for the $\IR$-coordinate of $x$.
    For $N = n+ |A|$, pick a bijection $\varepsilon: \{1,\dots,N\} \cong \{1,\dots,n\} \amalg A$ such that the map
    \begin{align*}
        \varepsilon_\IR: \{1,\dots,N\} & \longrightarrow [0,n+1] \\
        j & \longmapsto \begin{cases}
            \varepsilon(j) & \text{ if }\varepsilon(j) \in \{1,\dots,n\} \\
            x_\IR(\varepsilon(j)) & \text{ if }\varepsilon(j) \in A
        \end{cases}
    \end{align*}
    is monotonous.
    Write $\bbt'$ for the composite of $\varepsilon_\IR$ with the piece-wise linear map $[0,n+1] \to [0,1]$ that is obtained from $\bbt$ by interpolation.
    Let $(-)_\bot: \{1,\dots,N\} \to [n]$ be the map that sends $j$ to the largest $i$ such that $i \le \varepsilon(j)$.
    Pick $\go: A \to \Omega$ such that $\go(\ga) \in W(\ga_\bot)$ and such that $(\go(\ga), x_\IR(\ga))$ maps to $x(\ga)$ under $W(\ga_\bot) \times [\ga_\bot, \ga_\bot+1] \to |W|$.
    Then we define the $N$-simplex $W' \in C_N^\nc$ as follows.
    \begin{align*}
        W'(j) &:= 
            W(j_\bot) \amalg \{\ga \in A \;|\; \varepsilon(\ga) \le j \}
        \qand \\
        W'(j-1\le j) &:= \begin{cases}
            W(i-1 \le i) \amalg \id_{\{\ga \in A \;|\; \varepsilon(\ga) \le j\}} & \text{ if } i = \varepsilon(j) \in \{1,\dots,n\} \\
            P_{\go(\ga), \ga} \amalg \id_{W'(j-1) \setminus \{\ga\}} & \text{ if } \ga = \varepsilon(j) \in A.
        \end{cases}
    \end{align*}
\end{defn}

\begin{lem}\label{lem:gs-continuous}
    The construction of $\gs(W,\bbt,A,x) \in |C_\cd^\nc|$ in definition \ref{defn:K-new} is independent of the choice of the auxiliary maps $\varepsilon$ and $\omega$, and for fixed $W$ and $A$ it yields a well-defined continuous map 
    \begin{align*}
       |\Delta^n| \times \Map(A,|W|) & \longrightarrow |C_\cd^\nc| \\
        (\bbt, x) &\longmapsto \gs(W,\bbt,A,x).
    \end{align*}
\end{lem}
\begin{proof}
    Fix $W$ and $A$ as in definition \ref{defn:K-new}.
    If we also fix $\varepsilon$ and $\go$, then it is clear that $\bbt'$ is continuous in $\bbt$ and $x$ and thus $\gs(W,-,A,-)$ defines a continuous map on the subspace of $|\Delta^n| \times \Map(A, |W|)$ on those $(\bbt,x)$ such that $\varepsilon$ and $\go$ satisfy the conditions of definition \ref{defn:K-new}.
    (It defines a continuous map to $|\Delta^N|$ which then maps to $|C_\cd^\nc|$ as the $N$-simplex $W'$.)
    As we go through the finite set of choices for $\varepsilon$ and $\go$ this yields a cover of $|\Delta^n| \times \Map(A, |W|)$ by \emph{closed} subspaces (some of which might be empty) such that $\gs(W, -, A, -)$ is continuous on each of them.
    Note that this is indeed a cover because it is always possible to pick $\varepsilon$ and $\go$: this uses that for every $i \in \{1,\dots,n\}$ the source and target map $W(i-1) \to W(i-1\le i) \leftarrow W(i)$ are jointly surjective since we are in $C_\cd^\nc$.
    Therefore, we just have to show that for each fixed $\bbt$ and $x$ any choice of $\varepsilon$ and $\go$ actually results in the same point in $|C_\cd^\nc|$.

    Fix $(W,\bbt,A,x)$. 
    When choosing $\varepsilon$ there is ambiguity whenever there is $j \in \{1,\dots,N-1\}$ with $\varepsilon_\IR(j) = \varepsilon_\IR(j+1)$ and we need to show that in this case composing $\varepsilon$ with the transposition that swaps $j$ and $j+1$ does not change the resulting point in $|C_\cd^\nc|$.
    When this happens we always have a repetition in $\bbt'$ as $\bbt'_j = \bbt'_{j+1}$, and we write $\bbt' = \delta^j \hat{\bbt'}$ where $\hat{\bbt'} \in |\Delta^{N-1}|$ is obtained by merging these two values.
    Now let $W''$ be the $N$-simplex that we obtain if we swapped $j$ and $j+1$ in our choice of $\varepsilon$.
    Then we have that in $|C_\cd^\nc|$
    \[
        [W', \bbt' = \delta^j \hat{\bbt'}] = [d_j W', \hat{\bbt'}] \stackrel{?}{=}
        [d_j W'', \hat{\bbt'}] = [W'', \bbt']
    \]
    and for the middle equality we will check that $d_j W' = d_j W''$.
    These two $(N-1)$-simplices already agree by construction expect possibly for one morphism, so it will suffice to argue that $W'(j-1\le j+1) = W''(j-1 \le j+1)$ in $\mcC$.

    First, suppose that $\varepsilon(j) = \ga$ and $\varepsilon(j+1) = \ga'$ are both in $A$.
    Additionally, suppose that $\go(\ga) \neq \go(\ga')$.
    In this case 
    \begin{align*}
        W'(j-1 \le j+1) &= (P_{\go(\ga'),\ga'} \amalg \id_{W'(j) \setminus \{\ga'\}}) \circ (P_{\go(\ga),\ga} \amalg \id_{W'(j-1) \setminus \{\go(\ga)\}})\\
        &= P_{\go(\ga'),\ga'} \amalg P_{\go(\ga),\ga} \amalg \id_{W'(j-1) \setminus \{\go(\ga), \go(\ga')\}}
    \end{align*}
    and for $W''$ the roles of $\ga$ and $\ga'$ are reversed, but the result is the same.
    If we instead have that $\go(\ga) = \go(\ga')$ then
    \begin{align*}
        W'(j-1 \le j+1) &= (P_{\go(\ga),\ga'} \amalg \id_{W'(j) \setminus \{\go(\ga)\}}) \circ (P_{\go(\ga),\ga} \amalg \id_{W'(j-1) \setminus \{\go(\ga)\}}) \\
        &= ((P_{\go(\ga),\ga'} \amalg \id_\ga) \circ P_{\go(\ga),\ga}) \amalg \id_{W'(j-1) \setminus \{\go(\ga)\}} \\
        W''(j-1 \le j+1) &= ((P_{\go(\ga),\ga} \amalg \id_{\ga'}) \circ P_{\go(\ga),\ga'}) \amalg \id_{W'(j-1) \setminus \{\go(\ga)\}}
    \end{align*}
    and the equality $W'(j-1\le j+1)=W''(j-1\le j+1)$ follows from lemma \ref{lem:surgery-paths-commute}.

    For the second case, suppose that $\varepsilon(j) = i \in \{1,\dots,n\}$.
    Then we must have $\varepsilon(j+1) = \ga \in A$ because $\varepsilon_\IR$ is strictly monotone on $\{1,\dots,n\}$.
    (The argument will also deal with the opposite case where $\varepsilon(j) \in A$ and $\varepsilon(j+1) \in \{1,\dots,n\}$.)
    Let $\go': A \to \gO$ be a map compatible with the transposed $\varepsilon$.
    (If it is not possible to pick such a $\go'$ we can ignore this transposition.)
    This means that $\go(\ga) \in W(i)$ and $\go'(\ga) \in W(i-1)$ must map to the same element $w \in W(i-1\le i)$, which exactly corresponds to $x(\ga) \in |W|$.
    We can assume, without loss of generality, that $\go$ and $\go'$ agree away from $\ga$; the independence of the choice of $\go$ will be discussed below.
    The morphisms we need to compare are now 
    \begin{align*}
        W'(j-1 \le j+1) &= (P_{\go(\ga),\ga} \amalg \id_{W'(j) \setminus \{\go(\ga)\}}) \circ (W(i-1 \le i) \amalg \id_{\{\ga \in A \;|\; \varepsilon(\ga) \le j\}}) \\
        W''(j-1 \le j+1) &=
        (W(i-1 \le i) \amalg \id_{\{\ga \in A \;|\; \varepsilon(\ga) \le j\}}) \circ 
         (P_{\go'(\ga),\ga} \amalg \id_{W'(j-1) \setminus \{\go'(\ga)\}}) .
    \end{align*}
    Ignoring identity morphisms, we only have to consider the connected component of the morphism $W(i-1 \le i)$ that corresponds to $x(j) \in |W|$.
    So, without loss of generality, we may assume that $W(i-1 \le i)$ is connected.
    Write $M := F(W(i-1) \setminus \{\go'(\ga)\})$ and $N := F(W(i) \setminus \{\go(\ga)\})$.
    Then $W$ is a connected morphism $F(\go'(\ga)) \otimes M \to F(\go(\ga)) \otimes N$. (As before, we freely use the braiding and associator of $\mcC$.)
    By definition \ref{defn:admits-surgery} we have a commutative diagram
    \[
        \begin{tikzcd}[column sep = 5pc]
            {F(\go'(\ga))} \ot M \ar[d, "{W(i-1 \le i)}"'] \ar[r, "P_{\go'(\ga),\ga} \ot \id_M"] &
            O \ot {F(\go(\ga))} \ot M \ar[d, "\id_O \ot {W(i-1 \le i)}"] \\
            {F(\go(\ga))} \ot N \ar[r, "P_{\go(\ga),\ga} \ot \id_N"] &
            O \ot {F(\go(\ga))} \ot N,
        \end{tikzcd}
    \]
    which exactly shows $W'(j-1\le j+1) = W''(j-1 \le j+1)$ as morphisms in $\mcC$.

    Finally, we need to argue that the choice of $\go$ does not affect $\gs(W,\bbt,A,x) \in |C_\cd^\nc|$.
    If $\ga \in A$ is such that $x_\IR(\ga) \not\in \{1,\dots,n\}$, then $x(\ga) \in \gO$ uniquely determines the $\gO$-coordinate and we have no choice.
    Suppose we have $x_\IR(\ga) = i \in \{1,\dots,n\}$.
    Then, up to reordering $\varepsilon$, we can arrange for $\ga = \varepsilon(j\pm 1)$ with $j = \varepsilon^{-1}(i)$.
    We will deal with the case of $\ga = \varepsilon(j+1)$, the other being analogous.
    As before, we have a repetition in $\bbt'$ and eventually the claim can be reduced to checking that the two morphism
    \begin{align*}
        W'(j-1 \le j+1) &= (P_{\go(\ga),\ga} \amalg \id_{W'(j) \setminus \{\go(\ga)\}}) \circ (W(i-1 \le i) \amalg \id_{\{\ga \in A \;|\; \varepsilon(\ga) \le j\}}) \\
        W''(j-1 \le j+1) &= (P_{\go'(\ga),\ga} \amalg \id_{W'(j) \setminus \{\go'(\ga)\}}) \circ (W(i-1 \le i) \amalg \id_{\{\ga \in A \;|\; \varepsilon(\ga) \le j\}}) 
    \end{align*}
    are equal, with the only difference being that we use $P_{\go'(\ga),\ga}$ instead of $P_{\go'(\ga),\ga}$.
    Up to using the braiding we have so far left implicit, this is exactly the content of the bottom diagram in definition \ref{defn:admits-surgery}, see also the bottom picture in figure \ref{fig:admits-surgery}.
    (In the case of $\ga = \varepsilon(j-1)$ we instead use the top right square in \ref{defn:admits-surgery})

    This concludes the proof that definition \ref{defn:K-new} is independent of the choices made and, as argued in the first paragraph of the proof, we therefore get a continuous map as claimed.
\end{proof}

\begin{figure}
    \centering
    \small
    \def\svgwidth{\figurerescalefactor\linewidth}
    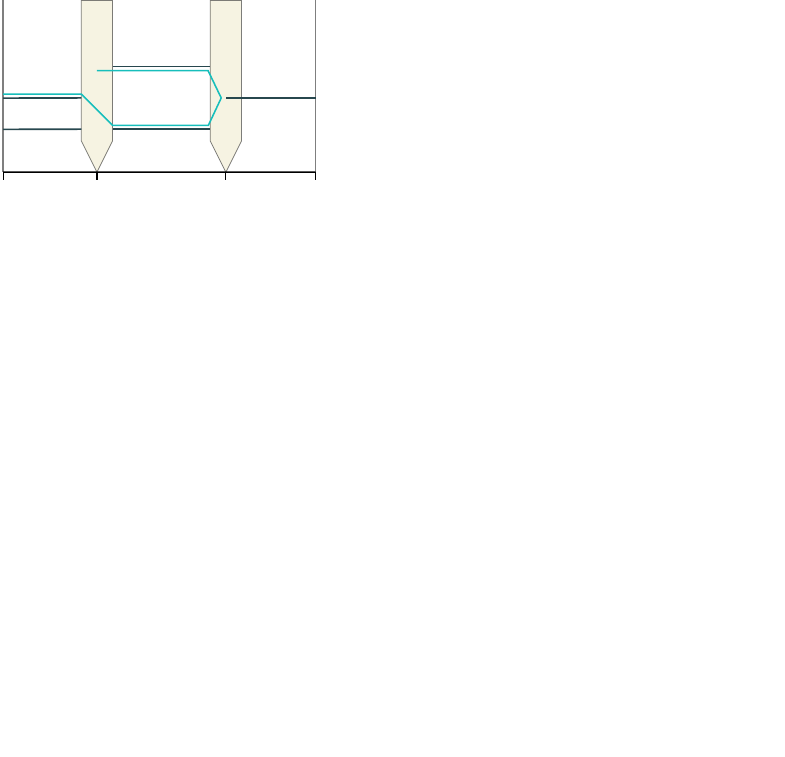
    \caption{A slide-show depiction of the homotopy $\gs(W,\bbt,A,p(r))$ for $A = \{\ga\}$ and $p$ the path indicated in the first picture.
    Here $\gamma:[0,1] \to [0,1]$ denotes the composite of $p_\IR$ and the map $[0,n+1] \to [0,1]$ obtained by linearly interpolating $\bbt$.
    See remark \ref{rem:fig-surgery-homotopy}.
    }
    \label{fig:surgery-homotopy}
\end{figure}

\begin{figure}
    \centering
    \small
    \def\svgwidth{\figurerescalefactor\linewidth}
    \import{}{surgery-homotopy-2.pdf_tex}
    \caption{A slide-show depiction of the homotopy $\gs(W,\bbt,\{\ga\},p(r))$ for a path starting at $1$. 
    }
    \label{fig:surgery-homotopy-2}
\end{figure}

\begin{rem}\label{rem:fig-surgery-homotopy}
    Figure \ref{fig:surgery-homotopy} illustrates the homotopy $\gs(W,\bbt,A,p)$ in the case of $A = A^{\rm in} = \{\ga\}$.
    The morphisms in the simplex $W$ are $f:\go_2 \ot \go_3 \to \go_3$,
    $g:1_\mcC \to \go_1$, and $h: \go_1 \ot \go_3 \to \go_2$.
    Throughout the homotopy we see other morphisms 
    $f' = P_{\go_3,\ga} \circ f = (\id_\ga \ot f) \circ (P_{\ga, \go_2} \ot \id_{\go_3})$,
    $g' = P_{\go_1,\ga} \circ g$, and 
    $h' = (\id_\ga \ot h) \circ (\id_{\go_1} \ot P_{\ga, \go_3})
    = (\id_\ga \ot h) \circ (P_{\ga, \go_1} \ot \id_{\go_3})$.
    A key step of lemma \ref{lem:gs-continuous} is to show that the two different
    expressions we have for $f'$ and for $h'$ yield the same morphisms.
    
    See also figure \ref{fig:surgery-homotopy-2}, for an illustration of the case $\ga \not\in A^{\rm in}$,
    where the surgery path starts at $p_{\IR}(0)=1$.
    Note that in this case we do not need to start out with a disjoint copy of $\id_\ga$ on top of the diagram.
    This difference will play a role when gluing the homotopies in proposition \ref{prop:mcS}.
\end{rem}

We can now concatenate the two homotopies to obtain the basic surgery.
\begin{defn}
    The homotopy $\mcS$ is defined as follows:
    \begin{align*}
        \mcS: |C_\cd^\gs| \times [0,1] &\to |C_\cd^\nc|, \\
        ((W, (p^\ga)_{\ga \in A},\bbt), r) &\mapsto 
        \mcS_A^r(W,\bbt) :=
        \begin{cases}
            \rho_{A^{\rm in}}^{1-2r}(W, \bbt) & \text{ for } 0 \le r \le \oh,\\
            \gs(W, \bbt, A, p^{(-)}(2r-1)) & \text{ for } \oh \le r \le 1.
        \end{cases}
    \end{align*}
\end{defn}

To conclude this subsection we record the fact that the homotopy $\mcS$ we constructed has all the desired properties.
\begin{prop}\label{prop:mcS}
    The above defines a continuous family of maps $\mcS^r: |C_\cd^\gs| \to |C_\cd^\nc|$ 
    satisfying:
    \begin{itemize}
        \item[(i)] $\mcS^0: |C_\cd^\gs| \to |C_\cd^\nc|$ is the realization of 
        the map $C_\cd^\gs \to C_\cd^\nc$ that forgets the surgery data.
        \item[(ii)] $\mcS^1: |C_\cd^\gs| \to |C_\cd^\nc|$ factors through the subspace
        $|C_\cd^\pb| \subset |C_\cd^\nc|$.
        \item[(iii)] Precomposed with the inclusion $|C_\cd^\pb| \to |C_\cd^\gs|$ that equips 
        the empty surgery data, the family $\mcS^r$ is the standard inclusion 
        $|C_\cd^\pb| \to |C_\cd^\nc|$ for all $r$.
    \end{itemize}
\end{prop}
\begin{proof}
    We already know that $\gs(W,\bbt,A,x)$ is continuous in its inputs, and since each $p^\alpha(2r-1) \in |W|$ varies continuously with $r$ it follows that $\mcS$ defines a continous map on each $|\Delta^n| \times [1/2, 1]$.
    To see that these assemble into a continuous map out of the geometric realization $|C_\cd^\gs| \times [1/2,1]$, we need to check that they are compatible with face maps.
    Let $(W, (p^\ga)_{\ga \in A})$ be an $n$-simplex in $C_n^\gs$ and $\bbt \in |\Delta^{n-1}|$, then $[d_i(W, p), \bbt] = [(W,p), \delta^i\bbt]$ in $|C_\cd^\gs|$, so we need to check that they are send to the same point.
    On surgery paths we defined the face map by post-composing with the map $|W| \to |d_iW|$ that collapses the segment corresponding to the $i$th component.
    Because of this $\gs(d_iW,\bbt,A,p^{-}(r))$ does not change as we vary the surgery paths within the region $\pr_\IR^{-1}([i,i+1])\subset |W|$, and on the other hand $\gs(W,\delta^i\bbt,A,p^{-}(r))$ does not change because $(\delta^i \bbt)_i = (\delta^i \bbt)_{i+1}$.
    This allows us (for fixed $r$) to move all values of $p^\ga(r)$ that would be in the interior of $[i,i+1]$ to the boundary, and then we can re-choose the auxiliary maps $\varepsilon$ and $\go$ in definition \ref{defn:K-new} in such that way that the are no points between $i$ and $i+1$, i.e.~writing $j:= \varepsilon^{-1}(i)$ we have $\varepsilon(j+1) = i+1$.
    (This again uses that there are no closed components in $W$ as we need to move each element $\gamma \in W(i)$ to one in either $W(i-1)$ or $W(i+1)$ that maps to the same component of $W(i-1 \le i+1)$ as $\gamma$.)
    Once this is done, there are no other morphisms between $W(i-1 \le i)$ and $W(i \le i+1)$ in the $N$-simplex $W'$ and thus $d_j(W') = (d_iW)'$, which together with $\delta^j(\bbt') = (\delta^i \bbt)'$ gives
    \[
        \gs(d_iW, \bbt, A, p^{-}(r)) = [(d_iW)', \bbt'] = [d_j(W'), \bbt'] = [W', \delta^j (\bbt')] = [W', (\delta^i\bbt)'] = \gs(W, \delta^i\bbt, A, p^{-}(r)).
    \]

    We have so far argued that the map is continuous on $|C_\cd^\gs| \times [1/2,1]$.
    The homotopy from definition \ref{defn:rho} gives a continuous map on $|C_\cd^\gs| \times [0,1/2]$ whose restriction to $\{1/2\}$ is by construction the map $\mcS_{1/2}$ that takes the disjoint union with $A^{\rm in}$.
    Inspecting definition \ref{defn:K-new} we see that this is also what $\gs(W, \bbt, A, \rho^{-}(0))$ does, so the two homotopies fit together giving the desired map on $|C_\cd^\gs| \times [0,1]$.
    
    Claim (i) follows from the construction of $\rho$ in definition \ref{defn:rho}.
    Claim (iii) is similarly straightforward:
    if the surgery data is empty then neither $\rho$ nor $\gs(W,\bbt, \emptyset, \emptyset)$ change anything about $(W,\bbt)$.
    
    Claim (ii) is the most interesting one: it says that the result of the surgery is always in the positive-boundary subspace.
    We need to check that for every $n$-simplex $W \in C_n^\nc$ with surgery data $(p^\ga)_{\ga \in A}$ the endpoint of the surgery move $\gs(W,\bbt,A,p^{(-)}(1)) = [W',\bbt']$ is in $|C_\cd^\pb|$.
    The $N$-simplex $W' \in C_N^\nc$ will not be in $C_N^\pb$ (unless $W$ was), but the coordinates $\bbt'$ have many repetitions and by rewriting $[W',\bbt']$ using the simplicial relations we will be able to see that it is in $|C_\cd^\pb|$.
    As we are evaluating the surgery paths at $1$, the map $x:= p^{(-)}(1): A \to |W|$ lands entirely over $\{0,\dots,n+1\}$, and therefore every entry in $\bbt'$ is a repetition of one in $\bbt$.
    After using all simplicial relations to simplify $[W',\bbt'] = [W'',\bbt]$ we have an $n$-simplex $W''$ with
    \begin{align*}
        W''(j) &= W(j) \amalg \{ \ga \in A \;|\; p_\IR^\ga(1) \le j-1\}\\
        W''(j-1\le j) &= 
        \left(\left(P_j^+ \circ W(j-1 \le j) \right) 
            \amalg \id_{\{\ga \in A \;|\; \varepsilon(\ga)\le j\}}\right)
         \circ P_j^-
            \amalg \id_{\{ \ga \in A \;|\; p_\IR^\ga(1) \le j-1\}}
    \end{align*}
    where $P_j^+$ is obtained by composing all $P_{\go(\ga),\ga}$ for which $p^\ga(1) = j$ and $\varepsilon(\ga) > j$,
    and similarly $P_j^-$ is obtained from those with $p^{\ga}(1) = j$ and $\varepsilon(\ga) < j$.
    Note that each component of $P_j^+$ and $P_j^-$ is positive boundary and if we compose them with $W(j-1\le j)$ then every connected component where we compose them becomes positive boundary.
    Therefore, we need that for every $y \in W(j-1 \le j)$ either it already has outgoing boundary, or there is some $\ga \in A$ with $p^\ga(1) = y \in |W|$.
    This is exactly what we required in the definition \ref{defn:surgery-datum}.
\end{proof}

\subsection{Contractible surgery data}
One problem we have not addressed yet is how to choose the surgery data 
        $(p^\ga)_{\ga \in A}$
for a given point $(W,t)$ in $|C_\cd^\nc|$. 
Concretely, the problem is that the forgetful map $C_\cd^\gs \to C_\cd^\nc$ 
will not induce an equivalence on realizations.
To resolve this we need to interpolate between different choices of surgery data by only traversing some surgery paths partially.

The following definition allows for multiple pieces of surgery data.
\begin{defn}
    The bi-semi-simplicial set $C_{\cd,\cd}^\gs$ has as $(n,m)$-simplices 
    tuples $(W, ((p_0^\ga)_{\ga \in A_0},\dots, (p_m^\ga)_{\ga \in A_m}))$
    where $W \in C_n^\nc$ and each $(p_i^\ga)_{\ga \in A_i}$ 
    is a surgery datum for $W$, such that $A_i$ and $A_j$ are disjoint for $i \neq j$.
    The $i$-th face operator in the second direction is defined by forgetting
    $(p_i^\ga)_{\ga \in A_i}$.
\end{defn}

We can always find surgery data that is disjoint to all previous surgery,
which makes it very easy to show contractibility.
\begin{lem}\label{lem:surgery-contractible}
    For any $n$ the augmentation of $C_{n,\cd}^\gs$ induces an equivalence
    $|C_{n,\cd}^\gs| \simeq C_n^\nc$.
\end{lem}
\begin{proof}
    Since $C_n^\nc$ is a discrete set we need to show that for all $W \in C_n^\nc$ the fiber $C_{n,\cd}^\gs(W)$ of the map $C_{n,\cd}^\gs \to C_n^\nc$ at $W$ is a contractible semi-simplicial set.
    The realisation of the semisimplicial set $C_{n,\cd}^\gs(W)$ is exactly the simplical complex whose vertices are surgery data $(p^{\alpha})_{\alpha\in A}$ and where $(n+1)$ vertices span a simplex if their indexing sets are disjoint.

    First, note that this simplicial complex is non-empty.
    We will show this by constructing a surgery datum for $W$.
    Let $A \subset \gO_O$ be a subset disjoint to $W$ and big enough 
    so that we may choose a surjection 
    $f:A \twoheadrightarrow \coprod_{i=1}^n W_\pi(i-1\le i) \subset |W|$.
    Since $W \in C_\cd^\nc$ every connected component of $|W|$ is connected 
    to either $\bot \sim [W(0) \times \{0\}]$ or $\top \sim [W(n) \times \{n+1\}]$.
    We can therefore find piece-wise linear paths $p^\ga:[0,1] \to |W|$ 
    such that $p^\ga(0) \in \{\bot, \top\}$ and $p^\ga(1) = f(\ga)$.
    This trivially satisfies the connectivity condition because 
    it hits all components of morphisms, in particular those that are not
    positive boundary. Therefore $(p^\ga)_{\ga \in A}$ is the desired 
    surgery datum for $W$.

    Now to show that $C_{n,\cd}^\gs(W)$ is contractible, we check that for every finite sequence of surgery data 
    $(p_0^\ga)_{\ga \in A_0},\dots, (p_m^\ga)_{\ga \in A_m}$ 
    we can define a new surgery datum that is disjoint to all of them by simply copying the first one.
    Concretely, we choose a subset $A' \subset \gO$ that is disjoint to $W$ and all of the $A_i$ and pick a bijection $\gp:A' \cong A_0$.
    Then $p_{n+1}^\ga := p_{0}^{\gp(\ga)}$ defines a surgery datum indexed by $\ga \in A'$
    and is disjoint to all the others by construction.
    This shows that any finite subcomplex of $C^\gs_{n,\cd}$ can be coned off and hence $|C^\gs_{n,\cd}| \simeq *$.
\end{proof}

We now need to construct the surgery homotopy $\mcS_\gl$ dependent 
on a weighted collection of surgery data.
This will require us to surgery along multiple pieces of surgery data at once, which we can do by progressing through the surgery paths at different speeds.

To do so, let $\|C_{\cd,\cd}^\gs\|$ denote the geometric realization 
of the semisimplicial space $[n] \mapsto |C_{n,\cd}^\gs|$.
\begin{defn}\label{defn:mcS}
    We define a continuous map
    \[
        \mcS: \|C_{\cd,\cd}^\gs\| \times [0,1] \longrightarrow |C_\cd^\nc|.
    \]
    A point on the left is represented by a triple 
    $((W, \bbt), ((A_0,\dots,A_m),(s_0,\dots,s_m)), r)$
    where the $s_i \in [0,1]$ are such that $\sum_i s_i = 1$.%
    \footnote{
        Note that here it is convenient to use a different parametrisation of the topological $n$-simplex than we did for $\bbt$.
    }
    We set $r_i := r \cdot s_i/\max(s_0,\dots,s_m)$
    and choose a partition $\{0,\dots,m\} = I \sqcup J$ such that $r_i \le 1/2$ for $i \in I$ and $r_j \ge 1/2$ for $j \in J$.
    Then we define
    \[
        \mcS^r((W,\bbt),(A_0,\dots,A_m),(s_0,\dots,s_m)) 
        := (\circ_{j \in J}\; \rho_{A_j^{\rm in}}^{1-2r_j})\Bigg(\gs\Big(W, \bbt, \coprod_{i \in I} A_i, x\Big)\Bigg)
    \]
    where $x$ is the map that sends $\ga \in A_i$ (with $r_i \ge 1/2$) to $p^\ga(2r_i-1) \in |W|$.
\end{defn}

\begin{rem}
    The order of operations in definition \ref{defn:mcS} might be slightly counter-intuitive, as we first do the surgery at each $p^\ga(2r_i - 1) \in |W|$ for $i \in I$ (so $r_i \ge 1/2$) and only afterwards do the homotopies $\rho_{A_j^{\rm in}}^{1-2r_j}$ that introduces the arcs $A_j^{\rm in}$ for for $j \in J$ (so $r_j < 1/2$).
    The reason for this is that this way around it is clearer that the operations are well-defined.
    The result of the surgeries is still disjoint to $\coprod_{j \in J} A_j$, so we can indeed apply each of the $\rho_{A_j^{\rm in}}^{1-2r_j}$ to them.
    If we did this the other way around, we would have to argue that after applying $\rho_{A_j^{\rm in}}^{1-2r_j}$ we still have the surgery data in terms of the collections of surgery paths for the $(A_i)_{i \in I}$.
\end{rem}

\begin{proof}[Proof of theorem \ref{thm:Csp-surgery}]
    We begin by arguing that the map $\mcS$ defined above is indeed a well-defined continuous map.
    It follows from lemma \ref{lem:gs-continuous} in the previous section, that this map is continuous if it is well-defined.
    For this, we need to check that it is independent of the choice of partition $I \sqcup J$ and that it is compatible with face maps in the first and second coordinate.
    When choosing the partition of $\{0,\dots,m\}$ into $I$ and $J$, the only choice is where we put those $k$ such that $r_k=1/2$.
    Here the choice does not matter because the surgery at the start points $\gs(W,\bbt, A_k, p^{(-)}(0))$ takes the disjoint union of $(W,\bbt)$ with the set $A_k^{\rm in} = \{ \ga \;|\; p_\IR^\ga(0) = 0\}$, which is exactly what $\rho_{A_k^{\rm in}}^0(-)$ does as well.
    This already implies that $\mcS$ induces a well-defined continuous map $|\Delta^n| \times |\Delta^m| \to |C_\cd^\nc|$ for each element of $C_{n,m}^\gs$.
    This glues under face maps in the first coordinate by the same argument as in the proof of proposition \ref{prop:mcS}.
    Finally, to get gluing in the second (semi-simplicial) coordinate, we need to consider what happens when $s_k = 0$. 
    In this case we have $r_k = 0$ as well, and as $\rho_{A_k^{\rm in}}^1(-)$ is the identity, this does not do anything and we can equivalently forget the $k$th entry.

    Now that we know that $(\mcS^r)_{r \in [0,1]}$ is a continuous family of continuous maps
    $\|C_{\cd,\cd}^\gs\| \to |C_\cd^\nc|$ we can fit it into the following diagram:
    \[
    \begin{tikzcd}[column sep = 4pc, row sep = 2pc]
        {|C_\cd^\pb|} \ar[d, equal] \ar[r, "\iota"]     &
        {\|C_{\cd,\cd}^\gs\|} \ar[d, "\mcS^0"] \ar[ld, "\mcS^1"', dashed]   \\
        {|C_\cd^\pb|} \ar[r]     &
        {|C_\cd^\nc|} 
    \end{tikzcd}
    \]
    The map $\iota$ is defined by equipping $(W,t)$ with the empty surgery $A_0 = \{\emptyset\}$.
    This is possible because $(W,t)$ is already positive boundary in this case.
    We need to check that $\mcS^1$ indeed lands in the subspace
    $|C_\cd^\pb| \subset |C_\cd^\nc|$.
    Indeed, for $r = 1$ we know that there is a $k$ such that $s_k=1$, meaning that at least in that coordinate we are fully performing the surgery. 
    The result of that surgery $\gs(W,t,A_k, p^{(-)}(1))$ then is positive boundary by part (ii) of \ref{prop:mcS}.
    Additionally doing the surgeries for the other $\coprod_{i \in I \setminus \{k\}} A_i$ and adding the disjoint copies of $T_\ga$ for the $\ga \in \coprod_{j \in J} A_j$ does not break this positive boundary property, so $\mcS^1$ always lands in $|C_\cd^\pb|$ and hence the dashed map exists.
    For the empty surgery the map $\mcS^1$ does not do anything 
    by part (iii) of \ref{prop:mcS} and so the top-left triangle commutes.
    By construction $\mcS^\gl$ is a homotopy for the bottom-right triangle in the diagram.
    Moreover, part (i) of \ref{prop:mcS} tells us that the map $\mcS^0$ 
    simply forgets the surgery data and we saw in lemma \ref{lem:surgery-contractible}
    that this map is an equivalence.
    
    In summary, the above diagram is homotopy commutative and the right-hand 
    vertical map is a homotopy equivalence. From this it follows formally from 
    the $2$-out-of-$6$ property that all other maps in the diagram are weak equivalences.
    We can check this by hand:
    Pick any base-points and apply $\pi_k$ to the diagram. 
    Considering the top-left triangle we see that the diagonal map $\pi_k(\mcS^1)$
    has to be surjective because the identity on $\pi_k|C_\cd^\pb|$ factors through it.
    Similarly, the homotopy commutativity of the bottom-right triangle 
    and the fact that $\pi_k(\mcS^0)$ is an isomorphism implies that $\pi_k(\mcS^1)$
    is injective. Hence we have shown that $\pi_k(\mcS^1)$ is an isomorphism
    and it follows quickly that all the other maps are, too.
\end{proof}

\subsection{The positive boundary subcategory}

In this section we compute the classifying space of the positive boundary subcategory
$\mcC^\pb \subset \mcC$ for all weighted cospan categories.
This includes the three cases $\Csp$, $\Cob_2$, and $\Cobn$ that appear in our main theorems.

\begin{prop}\label{prop:computing-pb}
    For any weighting monoid $(A, A_1, \ga)$ the classifying space of the positive boundary
    subcategory $\Csp(A, A_1, \ga)^\pb \subset \Csp(A, A_1, \ga)$ is
    \[
        B(\Csp(A, A_1, \ga)^\pb) \simeq BA
    \]
    and the inclusion $BA \simeq B(\Csp(A, A_1, \ga)^\pb) \to B(\Csp(A, A_1, \ga))$
    admits a splitting as an infinite loop space map.
\end{prop}
\begin{proof}
    In the case of $\Cob_2$ the equivalence $B\Cob_2^\pb \simeq S^1$
    was shown in \cite[Proposition 6]{Til96} and the splitting constructed
    in \cite[Theorem 10]{Til96}.
    We begin by recalling the constructions made there so we can modify 
    them for our purposes.

    We would like to construct a homotopy equivalence $B\Cob_2^\pb \to B\IZ$.
    Sending a surface to its Euler characteristic defines a functor $\chi:\Cob_2^\pb \to \IZ$, where we think of $\IZ$ as a category with one object.
    This is functorial because the Euler characteristic is additive under gluing along circles.
    However, the map $B\Cob_2^\pb \to B\IZ$ turns out not to be an equivalence, but rather a degree $2$ map, so we would like to ``divide $\chi$ by $2$''.
    This division is not always possible as surfaces with boundary can have odd Euler characteristic, but it is possible if we account for the boundary.
    The functor
    \begin{align*}
        \chi': \Cob_2^\pb & \to \IZ, \\
        (W:M \to N) & \mapsto \chi(W) + |\pi_0(M)| - |\pi_0(N)|
    \end{align*}
    is naturally isomorphic to $\chi$ but lands in $2\IZ$.
    In fact, using that the Euler characteristic of a \emph{connected} surface of genus $g$ with $n$ boundary components is $2-2g-n$, we can rewrite this number as
    \[
        \chi'(W) = \chi(W) + |\pi_0(M)| - |\pi_0(N)| 
        = 2\left(|\pi_0(W)| - |\pi_0(N)| - g(W) \right),
    \]
    where $g(W)$ is the sum of the genera of each of the components of $W$.
    The number $\chi'(W)$ is always even and non-positive as we assumed that $\pi_0(N) \to \pi_0(W)$ is a surjection in the positive boundary category.
    Therefore $-\oh\chi'$ lands in $\IN$ and we can define a functor $\Phi:\Cob_2^\pb \to \Cob_2^\pb$ that sends any object to $S^1$
    and that sends a cobordism $W:M \to N$ to the unique cobordism $\Phi(W):S^1 \to S^1$
    that is connected and has genus:
    \[
        g(\Phi(W)) 
        = -\oh\chi'(W) =  g(W) + |\pi_0(N)| - |\pi_0(W)|.
    \]
    There is a natural transformation $\rho:\Id_{\Cob_2^\pb} \Rightarrow \Phi$
    defined by letting $\rho_M:M \to S^1$ be the unique connected genus $0$ cobordism.
    One checks naturality by noting that $\rho_N \circ W$ and $\Phi(W) \circ \rho_M$ are both connected bordisms of genus $g(W) + |\pi_0(N)| - |\pi_0(W)|$.
    
    Now we can consider the full subcategory $\mcC_1 \subset \Cob_2^\pb$ 
    on the object $S^1$, which is equivalent to $\IN$ by counting the genus.
    The functor $\Phi$ lands in $\mcC_1$ and defines a retraction 
    $B\Phi: B\Cob_2^\pb \to B\mcC_1$ to the inclusion.
    The natural transformation $\rho$ gives a homotopy between the identity and
    $B\Phi: B\Cob_2^\pb \to B\mcC_1 \subset B\Cob_2^\pb$ and hence the inclusion and
    $B\Phi$ form a homotopy equivalence.
    
    We will now generalize this argument to the weighted cospan category
    $\mcC := \Csp(A, A_1, \ga)$. Let $\mcC_1 \subset \mcC^\pb$ denote
    the full subcategory on the object $*$. This category has a single object
    and morphisms $A$, so $B\mcC_1 = BA$.
    The functor $\Phi:\mcC^\pb \to \mcC_1$ again sends all objects to the single object $*$,
    and on morphisms it is defined by sending $(M \to W \leftarrow N, a:W \to A)$ 
    to the cospan $(* \to * \leftarrow *)$ weighted by:
    \[
        \Phi(W) = \sum_{x \in W} a(x) + \ga \cdot (|N| - |W|).
    \]
    The number $|N|-|W|$ is again non-negative as $N \to W$ is surjective 
    in the positive boundary category. One checks that this is functorial
    by using the definition of the composite weighting in definition \ref{defn:weighted-cospans}.
    
    To define the natural transformation $\rho:\Id \Rightarrow \Phi$
    we let $\rho_M: (M \to * \leftarrow *, a:* \to A)$ where the labelling is $a(*) = \ga$.
    (We cannot set the labelling to be $0$ like we did for the surface category case,
    as this might not be an allowed cospan for $M = \emptyset$.)
    Naturality is checked by inserting the definitions.
    By the same argument as above it now follows that $BA = B\mcC_1 \to B\mcC^\pb$
    is a homotopy equivalence.
    
    To construct the splitting of $B\mcC^\pb \to B\mcC$ 
    let $A^{\rm gp}$ be the group completion of $A$.
    Define a functor $\Phi': \mcC \to A^{\rm gp}$ by sending a morphism $W:M \to N$ to 
    \[
        \Phi'(W) = \sum_{x \in W} a(W) + \ga \cdot (|N| - |W|).
    \]
    The number $|N| - |W|$ can now be negative, but this is fine because we have 
    a formal inverse $-\ga$ for $\ga$ in $A^{\rm gp}$.
    $\Phi'$ clearly extends $\Phi$ and it is functorial for the same reason that $\Phi$ was.
    This yields the desired splitting: 
    the composite $BA= B\mcC_1 \to B\mcC^\pb \to B\mcC \xrightarrow{\Phi'} B(A^{\rm gp})$
    is the standard map $BA \to B(A^{\rm gp})$, which is an equivalence.
\end{proof}

\begin{cor}
    The classifying space of a weighted cospan category is:
    \[
        B(\ICsp(A, A_1, \ga)) \simeq 
        BA \times Q\left(\bigvee\nolimits_{a \in A} S^2 B(\mcF_a(\Csp(A, A_1, \ga)))\right).
    \]
\end{cor}
\begin{proof}
    The surgery theorem applies and by proposition \ref{prop:computing-pb} the left-hand map 
    in theorem \ref{theorem:decomposition-and-surgery} (corollary \ref{cor:decomposition-and-surgery})
    admits a splitting as infinite loop space map.
    Therefore the middle term decomposes as a product of $B(\Csp(A, A_1, \ga)^\pb)$ 
    and the free infinite loop spaces on the $S^2 (B \mcF_a)$.
\end{proof}

\section{Factorisation categories and filtered graphs}\label{sec:mcF-and-mcJ}
Our work in the previous sections shows that computing the classifying space
$B\Cobn$ of the $(\chi\le 0)$-surface category essentially reduces to understanding 
its factorisation categories.
In this section we show for $g \ge 2$ the classifying space of $\mcF_{\gS_g}(\Cobn)$
is equivalent to the classifying space a certain finite category $\J_g$ of stable graphs of genus $g$,
which plays an important role in studying the moduli spaces $\gD_g$ of tropical curves of genus $g$ and volume $1$, see \cite{CGP16}.%
\footnote{
    Note that as discussed in remark \ref{rem:our-J_g} our version $\J_g$
    differs from their version in that we remove the terminal object $\cd_g$.
}
In fact, we will see in section \ref{sec:J=gD} that $B\J_g$ is rationally equivalent to $\gD_g$.

\begin{thm}\label{thm:computing-mcF}
    There are equivalences
    \[
        B\mcF_{\gS_g}(\Cobn) \simeq \begin{cases}
            BO(2) & \text{ for } g=1,\\
            B\J_g & \text{ for } g\ge 2.
        \end{cases}
    \]
\end{thm}

\subsection{The category of graphs}
We now recall the category of stable graphs from \cite[section 2.2]{CGP16}, and discuss their geometric realizations.

\begin{defn}
    A \emph{graph} $G$ consists of finite sets of vertices $V_G$ and half-edges $H_G$, together with a fixed-point free involution $i:H_G \to H_G$ and a root map $r:H_G \to V_G$.
    We define the set of edges as $E_G = \{\{h, i(h)\} \;|\; h \in H_G\}$.
    
    A graph $G$ is \emph{connected} if the equivalence relation on $V_G \amalg H_G$ generated by $h \sim i(h)$ and $h \sim r(h)$ has a single equivalence class.
    (In particular connected graphs are non-empty.)
    We define the \emph{valence} of a vertex $v$ in a graph to be the number of incident half-edges: $\val(v) := |r^{-1}(v)|$.
\end{defn}

\begin{ex}
    The graph with one vertex and one loop $\oneloop$ can be defined in the above context
    as $V_G = \{v\}$, $H_G =\{h, h'\}$ with $i(h) = h'$, $i(h') = h$, and $r(h) = r(h') = v$.
    This has one vertex $v$, two half-edges $h$ and $h'$, and a single edge $\{h, h'\}$.
\end{ex}

\begin{defn}\label{defn:stable-graph}
    A \emph{stable graph} is a connected graph $G$ together with a function $w:V_G \to \IN$ such that for any vertex $v \in V_G$ we have $\val(v) + 2w(v) \ge 3$.
    We define the genus of a connected weighted graph $(G,w)$ to be 
    \[
        g(G,w) = (|E_G| - |V_G| + 1) + \sum_{v \in V_G} w(v) \in \IN.
    \]
\end{defn}

\begin{rem}
    Every stable graph $(G,w)$ has genus at least $2$.
    Indeed, we have
    \[
        \sum_{v \in V_G} w(v)
        \ge \sum_{v \in V_G} (\tfrac{3}{2}-\tfrac{1}{2}\val(v))
        = \tfrac{3}{2}|V_G| - |E_G|
    \]
    as the sum of all vertex valences is twice the number of edges,
    so $g(G,w) \ge \tfrac{1}{2}|V_G| + 1 > 1$.
\end{rem}

We define the relevant notion of morphisms between weighted graphs:
\begin{defn}
    Let $(G,w) = ((V_G,H_G,i,r),w)$ and $(G',w')=((V_{G'}, H_{G'},i',r'),w')$ be two stable graphs.
    A \emph{morphism} of stable graphs $f:(G,w) \to (G',w')$ is a map 
    $f:V_G \amalg H_G \to V_{G'} \amalg H_{G'}$ satisfying 
    \begin{enumerate}
        \item $ f \circ (\id_{V_{G}} \amalg i) = (\id_{V_{G'}} \amalg i') \circ f$ (so in particular $f(V_G) \subset V_{G'}$),
        \item $f \circ (\id_{V_{G}} \cup r) = (\id_{V_{G'}} \cup r') \circ f$,
        \item for every $h \in H_{G'}$ the preimage $f^{-1}(h)$ contains exactly one element, and 
        \item for every $v \in V_{G'}$ the preimage $f^{-1}(v)$ (which is a graph) is a connected graph of genus $w'(v)$.
    \end{enumerate}
    We let $\J$ denote the category of the stable graphs that have at least one edge, with morphisms defined as above.
    For any $g \in \IN$ and we let $\J_g \subset \J$ denote the full subcategory on the graphs of genus $g$.
\end{defn}

Note that any stable graph with at least one edge must have at least genus $2$.
The total genus of graphs is preserved by morphisms and hence we have a disjoint decomposition
\[
    \J = \coprod_{g \ge 2} \J_g.
\]

\begin{rem}\label{rem:our-J_g}
    In \cite[section 2.2]{CGP16} the authors define a category called $J_g$,
    on which our definition of the category $\J_g$ is based.
    One key difference is that $J_g$ contains the graph $\cd_g$ with one vertex and no edges. 
    This is a terminal object, which makes the classifying
    space $BJ_g \simeq *$ contractible. 
    If we remove this terminal object then $J_g \setminus \{\cd_g\}$
    is equivalent to our $\J_g$.
    We will argue in remark \ref{rem:our-gD_g} that our version of $J_g$ still leads to the same tropical moduli space $\gD_g$ as theirs.
\end{rem}

We will also need to think about graphs as topological spaces.
\begin{defn}
    For a stable graph $G = (V_G, H_G, i, r) \in \J$ we define its geometric realization to be the topological space
    \[
        |G| := \left( V_G \amalg H_G \times [0,1] \right)/\sim
    \]
    where we identify $(h,0) \sim r(h)$ and $(h,t) \sim (i(h),1-t)$ for all $h \in H_G$.
    For every graph morphism $f: G \to G$, we define the induced map $|f|: |G| \to |G'|$ by $[v] \mapsto [f(v)]$ and $[h,t] \mapsto [f(h),t]$.
    Moreover, we let $\Homeo(G)$ denote the group of homeomorphisms of $|G|$ that send vertices to vertices and we let $\Homeo_0(G) \subset \Homeo(G)$ be its identity component.
\end{defn}

Note that because $\varphi \in \Homeo_0(G)$ is isotopic to the identity in $\Homeo(G)$, it has to fix all vertices and it has to send each edge to itself.
Therefore we get an isomorphism of topological groups
\[
    \Homeo_0(G) \cong \prod_{e \in E_G} \Homeo_\partial([0,1])
\]
and in particular this group is contractible.
Moreover, $\Homeo(G) \cong \Aut(G) \ltimes \Homeo_0(G)$.

\subsection{From factorisations to graphs}

The purpose of this section is to prove theorem \ref{thm:computing-mcF} for $g\ge 2$.
Translating this through the equivalence $\Cobn \simeq \Csp(\IN, \IN_{\ge 1}, 1)$ from lemma \ref{lem:Cob2=labelled-csp}, we need to show that the classifying space of
\[
    \mcF_{\ge 2} := \coprod_{g \ge 2} \mcF_{g}(\Csp(\IN, \IN_{\ge 1}, 1))
\]
is equivalent to the classifying space of $\J$.
We will deal with the special case of $g=1$ in the next subsection.
The equivalence we construct does not come from a functor in either direction.
Rather, we define a third category $\SubJ$ where objects are triples $(G,w,U)$ consisting of a stable graph $(G,w)$ and an ``admissible'' subset $U \subset |G|$ of the geometric realization of $G$.
There are two functors
\[
    \J \longleftarrow \SubJ \longrightarrow \mcF_{\ge 2}.
\]
The one to the left simply forgets the subset $U$, 
and the one to the right takes $(G,w,U)$ and thinks of it as a factorisation
$|G| = U \cup_{\partial U} (|G| \setminus U^\circ)$.
We will show that both functors induce equivalences on classifying spaces.

\begin{defn}
    For a graph $G \in \J$, a subset $U \subset |G|$ is called \emph{admissible} if 
    \begin{itemize}
        \item $U$ is closed, has finitely many components, and no isolated points,
        \item $U$ is neither empty nor all of $X$, and 
        \item $V_G \cap \partial U = \emptyset$, i.e.~$U$ is a neighbourhood of all vertices it contains.
    \end{itemize}
    We let $\Subpi(G)$ denote the category where
    objects are admissible subsets $U \subset |G|$ and a morphism $U \to U'$ is equivalence class of homeomorphisms $\varphi\in \Homeo_0(G)$ satisfying $\varphi(U) \subset U'$.
    Two homeomorphisms $\varphi_0$ and $\varphi_1$ are equivalent if they are isotopic through homeomorphisms $\varphi_t$ satisfying $\varphi_t(U) \subset U'$.
\end{defn}

\begin{ex}
    If the graph $G = \oneloop$ consists of a single edge attached to a vertex, then the full subcategory of $\Subpi(G)$ on those admissible subsets $U \subset |G| \cong S^1$ that do not contain the vertex, is equivalent $\Delta$.
    (The equivalence sends $U \subset S^1 \setminus \{0\} \cong (0,1)$ to the totally ordered set $\pi_0(U)$.)
    One can similarly give combinatorial descriptions of other parts of this category, but the author is not aware of a concise description of $\Subpi(G)$ as a whole.
\end{ex}

\begin{lem}\label{lem:Subpi-iso}
    A morphism $[\varphi]: U \to U'$ in $\Subpi(G)$ is an isomorphism if and only if we can find a representative $\psi \in [\varphi]$ such that $\psi(U) = U'$.
\end{lem}
\begin{proof}
    First, consider the case that $G$ consists of two vertices connected by an edge.
    Then $U, U' \subset |G| \cong [0,1]$ are closed subsets with finitely many components and no isolated points.
    Since $[U] \cong [U']$ in $\Sub(G)$, we know that they have the same number of components, that $0 \in U \Leftrightarrow 0 \in U'$, and that $1 \in U \Leftrightarrow 1 \in U'$.
    By matching up the boundary points $\partial U \cong \partial U'$ and linearly interpolating, we can construct $\psi \in \Homeo_\partial([0,1])$ with $\psi(U) = U'$.
    The case of a general graph now follows by building $\psi$ on each edge as described above.
\end{proof}

If $f:H \to G$ is a morphism in $\J$ and $U \subset |G|$ is admissible, then $f^{-1}(U) \subset |H|$ is also admissible.
    It is closed and has finitely many components because $f\colon |H| \to |G|$ is continuous and has connected fibers.
    It is non-empty and not equal to $|G|$ because $f$ is surjective.
    It is a neighbourhood of all vertices it contains because $f\colon |H| \to |G|$ sends vertices to vertices and $\mrm{int}(f^{-1}(U)) \supset f^{-1}(\mrm{int}(U))$.

\begin{defn}
    For $f: H \to G$ we define a functor
    \[
        f^*: \Subpi(G) \to \Subpi(H)
    \]
    that on objects sends $U$ to $f^{-1}(U)$.
    For a morphism $[\varphi]:U \to U'$ we set $f^*[\varphi] = [\psi]$ where $\psi: H \to H$ is the unique homeomorphism that satisfies $f\circ \psi = \varphi \circ f$ and that is the identity on $f^{-1}(V(G))$.
    This induces a functor
    \[
        \Subpi(-): \J^{op} \to \Cat
    \]
    and we let $\SubJ \to \J$ be its cartesian unstraightening.
\end{defn}

We want to show that $B\Subpi(G)$ is contractible for all $G$.
First we recall a result about the topological poset of non-empty configurations $C(X)$.
For a space $X$, we let $C(X)$ be the topological poset $\{ \emptyset \neq A \subseteq X \;|\; A \text{ finite}\}$ ordered by inclusion and topologised as $\coprod_{n\ge 1} X^{\times n} /\Sigma_n$.

\begin{lem}\label{lem:C(X)-contractible}
    If $X$ is a Hausdorff space with infinitely many points, then the classifying space $BC(X)$ is weakly contractible.
\end{lem}
\begin{proof}
    This follows from the discretisation technique \cite[Propositions 2.6-2.8]{GRW17}.
    The exact version we need here can be found in \cite[Proposition 2.18]{BBS-finite}, and we take the open, symmetric, anti-reflexive relation $\bot$ mentioned there to be $\neq$.
    (To apply this proposition we need to check that for any finite subset $A \subset X$ the following discrete simplicial complex is contractible: vertices are points in $X \setminus A$ and $(n+1)$ vertices form an $n$-simplex if they are disjoint.
    Indeed, this is contractible as its poset of simplices is the discrete version of $C(X\setminus A)$, which is a filtered poset because $X\setminus A$ is non-empty.)
\end{proof}

\begin{lem}\label{lem:Subpi-contractible}
    For every graph $G$ the category $\Subpi(G)$ has a contractible classifying space.
\end{lem}
\begin{proof}
    For the purpose of this proof all geometric realizations are fat.
    Let $C(G)$ be the topological poset $C(|G| \setminus V)$ of finite non-empty subsets that avoid the vertices. 
    We know from lemma \ref{lem:C(X)-contractible} that $BC(G) \simeq *$ and this remains true when we take the fat geometric realization.
    (The thin geometric realization of $N_\cd C(G)$ is the fat geometric realization of $N_\cd (C(G), \subsetneq)$ and the comparison map is an equivalence by \cite[Proposition 3.8]{ERW19}, as $(C(G), \subset)$ is obtained from the non-unital $(C(G), \subsetneq)$ by disjointly adding units.)

    In preparation of the proof we note that the topological group $\Homeo_0(G)$ acts on $C(G)$ and this action is locally retractile in the sense of \cite[Definition 2.3]{CRW17} (the local retractions can be constructed directly as piecewise linear maps).
    In particular, any $\Homeo_0(G)$-equivariant map with target $C(G)$ will be a fibration by \cite[Lemma 2.5]{CRW17}.
    For example, the maps $d_i: N_1 C(G) \to N_0 C(G)$ are fibrations and by basechange it follows that all outer face maps in the nerve are fibrations.

    Throughout, we fix a geodesic metric on $|G|$.
    We let $\Subtop(G)$ denote the topological poset of admissible subsets of $|G|$.
    We topologise $\Subtop(G)$ with the following variant of the Hausdorff metric:
    \[
        d(U, U') = \left| |\pi_0(U)| - |\pi_0(U')| \right| + \inf \{ \varepsilon \in [0,\infty) \;|\; U  \subseteq B_\varepsilon(U') \text{ and } U' \subseteq B_\varepsilon(U) \}.
    \]
    The first term ensures that $\Subtop(G)$ decomposes as a disjoint union over the number of connected components, and we see that taking an admissible subset to its boundary defines a continuous map $\partial\colon \Subtop(G) \to C(G) = C(|G| \setminus V_G)$.
    Each fiber of this map is either empty or has exactly two elements: $U$ and $|G| \setminus U^\circ$.
    Because $\Homeo_0(G)$ still acts on $\Subtop(G)$ and $\partial$ is equivariant, the local retractile condition implies that $\partial$ is a finite covering map.
    This implies that $\Subtop(G)$ is also $\Homeo_0(G)$-locally retractile, and by the same argument as before we see that the outer face maps in the nerve $N\Subtop(G)$ are fibrations.
    
    We will now show that $B\Subtop(G)$ is contractible.
    To deduce this from contractibility of $BC(G)$, we proceed similar to the proof of the topological version of Quillen's theorem A in \cite[Theorem 4.7]{ERW19} applied to a hypothetical functor $C(G) \to \Subtop(G)$ -- but we do not construct such a functor.
    Let $X_{\cd, \cd} \subset N_\cd C(G) \times N_\cd\Subtop(G)$ be bi-semi-simplicial space whose $(n,m)$-simplices are those tuples 
    \[
        (A_0\subset \dots \subset A_n, U_0 \subset \dots \subset U_m)
    \]
    where $A_n \subset U_0$.
    We will argue that for all $n$ and $m$ the geometric realizations of projection maps
    \[
        p_n: X_{n,\cd} \to N_n C(G)
        \qquad
        q_m: X_{\cd,m} \to N_m \Subtop(G)
    \]
    are equivalences.
    First, note that $(p_n)_m$ and $(q_m)_n$ are fibrations for all $m$ and $n$.
    (For $n=0=m$ this holds because they are $\Homeo_0(G)$-equivariant maps into a locally $\Homeo_0(G)$-retractile space.
    The general case follows by using basechange and the fact that the outer face maps of $N_\cd \Subtop(G)$ and $N_\cd C(G)$ are fibrations as we established above.)
    Now \cite[Lemma 2.14]{ERW19} allows us to show that $p_n$ and $q_m$ realize to equivalences by proving that their fibers have a contractible realizations.
    The fiber of $q_n: X_{\cd, m} \to N_m\Subtop(G)$ 
    over some $U_0 \subset \dots \subset U_m$ in $N_m\Subtop(G)$
    is the nerve of the topological poset $C(U_0 \setminus (U_0 \cap V))$ of finite non-empty subsets of $U_0$ that avoid vertices.
    This has a contractible classifying space by lemma \ref{lem:C(X)-contractible}, and hence $|q_n|$ is an equivalence.
    The fiber of $p_n$ at some $(A_0 \dots A_n)$ is the nerve of the topological poset $\Subtop_{A_n\subset}(G)$ of admissible subsets containing $A_n$.
    Let $\delta > 0$ be sufficiently small (such that any two points in $A_n$ are more than $2\delta$ apart from each other and from all vertices) and let $A_n^\delta$ be the union of closed $\delta$-balls around the points of $A_n$ -- this is an admissible subset of $G$.
    Then the inclusion $\Subtop(G)_{A_n^\delta/} \to \Subtop_{A_n \subset}(G)$ induces a levelwise equivalence on nerves (we can push $\partial U$ away from $A_n$) and thus an equivalence on classifying spaces, but $\Subtop(G)_{A_n^\delta/}$ has an initial object and thus a contractible classifying space.
    This concludes the proof that $|p_n|$ and $|q_m|$ are equivalences for all $n$ and $m$.
    Now \cite[Theorem 2.2]{ERW19} implies that as we vary $n$ and $m$ they realize to equivalences $B\Subtop(G) \simeq ||X_{\cd,\cd}|| \simeq B C(G) \simeq *$.

    We now have to relate the topological poset $\Subtop(G)$ to the discrete $1$-category $\Subpi(G)$.
    Let $\mcA$ be the topological action category for the action of $\Homeo_0(G)$ on $\Subtop(G)$.
    This has as space of objects $\Subtop(G)$ and the space of morphisms is the space of triples $(U_0,U_1,\varphi)$ where $U_0,U_1 \in \Subtop(G)$, $\varphi \in \Homeo_0(G)$, and $\varphi(U_0) \subset U_1$.
    There is an inclusion $\Subtop(G) \to \mcA$ and on nerves this is a level-wise homotopy equivalence because one can construct a splitting $N_n\mcA \simeq N_n\Subtop(G) \times \Homeo_0(G)^{\times n}$ and $\Homeo_0(G)$ is contractible.
    In particular, $B\mcA$ is also contractible.
    For $\mcA$, the source-and-target map 
    \[
        (d_1,d_0): N_1(\mcA) \to N_0(\mcA) \times N_0(\mcA)
    \]
    is a fibration because it is a $\Homeo_0(G)^{\times 2}$-equivariant map into a locally $\Homeo_0(G)^{\times 2}$-retractile space.
    Let $\mcA^\delta$ be the version of $\mcA$ where the space of objects has the discrete topology and the space of morphism is topologised such that only the homeomorphism $\varphi$ can vary continuously.
    Then the ``identity'' functor $\mcA^\delta \to \mcA$ induces a homotopy base-change on nerves in the sense of definition \ref{defn:base-change} and thus by theorem \ref{thm:base-change} (or alternatively \cite[Theorem 5.2]{ERW19}) this inclusion induces a homotopy equivalence $B(\mcA^\delta) \simeq B\mcA \simeq *$.

    Finally, we have a functor of topologically enriched categories $\mcA^\delta \to \Subpi(G)$ that is the identity on objects an on morphisms sends $\varphi$ to its isotopy class.
    This induces an equivalence on mapping spaces because if we fix $U_1, U_2 \in \Subtop(G)$, then the space of $\varphi\in \Homeo_0(G)$ that satisfy $\varphi(U_1) \subset U_2$ has contractible components.
    Therefore we get $* \simeq B(\mcA^\delta) \simeq B\Subpi(G)$.
    (Recall that the fat and thin geometric realization of a semisimplicial set are equivalent \cite[Lemma 1.7]{ERW19}, so the difference does not matter for $B\Subpi(G)$.)
\end{proof}

\begin{cor}\label{cor:SubJ}
    The forgetful functor 
    \[
        \SubJ \to \J
    \]
    induces an equivalence on classifying spaces.
\end{cor}
\begin{proof}
    As this is (by definition) a cartesian fibration, Quillen's theorem A reduces to checking that the fibers have contractible classifying spaces, but the fiber at $G \in \J$ is $\Subpi(G)$, so the claim follows from lemma \ref{lem:Subpi-contractible}.
\end{proof}

Now we construct a functor 
$\SubJ \to \mcF_{\ge 2} =\coprod_{g \ge 2} \mcF_{g}(\Csp(\IN, \IN_{\ge 1}, 1)) $, 
which will also be an equivalence on classifying spaces.
(Recall that via the equivalence $\Cobn \simeq \Csp(\IN, \IN_{\ge1}, 1)$ from lemma \ref{lem:Cob2=labelled-csp} we can think of such labelled cospans as representing surface bordisms.)

\begin{defn}\label{defn:S}
    We define a functor $S\colon \SubJ \to \mcF_{\ge 2}$ by sending 
    $(G, U)$ to the factorisation
    \[
        S(G, U) := 
        \left( \emptyset \to [\pi_0(U), w_{U}] \leftarrow 
        \partial U \to [\pi_0(X \setminus U^\circ), w_{X\setminus U}] \leftarrow \emptyset\right).
    \]
    Here the weights $w_U:\pi_0(U) \to \IN$ and 
    $w_{X\setminus U}:\pi_0(X \setminus U^\circ) \to \IN$
    are defined by sending a component $U_0 \subset U$ to its total genus, i.e.~the number $b_1(U_0) + \sum_{v \in V_G \cap U_0} w(v)$.
    
    To a morphism $(f,[\varphi]): (G_1, U_1) \to (G_2,U_2)$ a we assign the cospan 
    \[
        S(f) = \left(\partial U \to
        [\pi_0((f\circ \varphi)^{-1}(U') \setminus U^\circ), w_{V, U}] 
        \xleftarrow{(f\circ \varphi)^{-1}} \partial U' \right)
    \]
    where $w_{V, U}$ sends each component of $V \setminus f(U)^\circ$ 
    to its total genus.
    See figure \ref{fig:functor-S} for an illustration of how this functor 
    is applied to two objects and a morphism.
\end{defn}

\begin{figure}[ht]
    \centering
    \small
    \def\svgwidth{\figurerescalefactor\linewidth}
    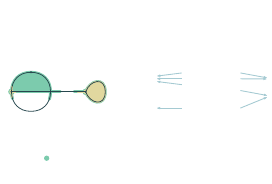
    \caption{The functor $S$ evaluated on a morphism $(f,[\varphi]):(X, U) \to (Y, V)$, where for simplicity we consider $\varphi = \id$.}
    \label{fig:functor-S}
\end{figure}

\begin{lem}
    The above construction indeed describes a well-defined functor
    $S\colon \SubJ \to \mcF_{\ge 2}$.
\end{lem}
\begin{proof}
    First we note that $S(f)$ does not depend on the representative $\varphi$ in the isotopy class $[\varphi]$ because the homotopy type of $(\varphi \circ f)^{-1}(U') \setminus U^\circ$ and the set of vertices it contains does not change as we vary $\varphi$ among homeomorphisms satisfying $\varphi(f(U)) \subset U'$.

    Because $\SubJ$ only involves admissible subsets of \emph{stable} graphs, the resulting factorisation $S(G,U)$ will always be in the subcategory $\Csp(\IN, \IN_{\le 1}, 1) \subset \Csp(\IN, 1)$ and the total composite will always have genus at least $2$.

    We need to check that $S$ is functorial and for this it suffices to check that the composite $\SubJ \to \mcF_{\ge 2} \to \Cob_2^{\chi\le0}$ is functorial.
    Consider morphisms $f: G_1 \to G_2$, $g: G_2 \to G_3$ and admissible subsets $U_i \in \Subpi(G_i)$ with $[\varphi]:U_1 \to f^{-1}(U_2)$ and $[\psi]:U_2 \to g^{-1}(U_3)$.
    We would like to show that the weighted cospan $S(g \circ f, [f^*(\psi) \circ \varphi])$ 
    agrees with the composite $ S(g,[\psi]) \circ S(f,[\varphi]) = S(f,[\varphi])\cup_{\partial U_2} S(g, [\psi])$.
    On underlying sets one can describe this pushout as
    \begin{align*}
        \pi_0(f_\varphi^{-1}(U_2) \setminus U_1^\circ) 
        \cup_{\partial U_2} \pi_0(g_\psi^{-1}(U_3) \setminus U_2^\circ)
        &\cong \pi_0(f_\varphi^{-1}(U_2) \setminus U_1^\circ) \cup_{\partial U_2} \pi_0((g_\psi \circ f_\varphi)^{-1}(U_3) \setminus f_\varphi^{-1}(U_2^\circ)) \\
        &\cong \pi_0((g_\psi \circ f_\varphi)^{-1}(U_3) \setminus U_1^\circ) 
    \end{align*}
    where we write $f_\varphi := f \circ \varphi: |G_1| \to |G_2|$, $g_\psi := g \circ \psi: |G_2| \to |G_3|$, so that $g_\psi \circ f_\varphi = (g \circ f) \circ (f^*\psi \circ \varphi)$.
    (The first bijection in the display uses that $f_\varphi$ has connected preimages and the second bijection uses that $\pi_0$ preserves pushouts.)
    To conclude $S(g,[\psi]) \circ S(f,[\varphi]) = S(g \circ f, [f^*\psi \circ \varphi])$ we need to
    check that both sides are weighted in $\IN$ in the same way.
    To show this, let $C \subset (g_\psi \circ f_\varphi)^{-1}(U_3) \setminus U_1^\circ$ some component.
    As a point in $S(g \circ f, [f^*\psi\circ \varphi])$ this is labelled by its total genus.
    We have a decomposition of $C$ into $C_0 = C \cap (f_\varphi^{-1}(U_2) \setminus U_1^\circ)$ 
    and $C_1 =  C \cap ((g_\psi \circ f_\varphi)^{-1}(U_3) \setminus f_\varphi^{-1}(U_2)^\circ)$.
    The total genus of each component of $C_1$ agrees with that of the relevant
    component in $f_\varphi(C_1)$, which is a union of components of $g_\psi^{-1}(U_3)\setminus U_2^\circ$.
    The total genus of $C$ can be computed by adding up the genus of each component of $C_0$ and $C_1$ and then adding the first Betti number of the graph with vertices $\pi_0(C_0) \amalg \pi_0(C_1)$ and edges $ C_0 \cap C_1 = \partial U_2$.
    This is exactly the bipartite graph that we used in definition \ref{defn:weighted-cospans} to define the composition in a weighted cospan category.
    Hence $S(g \circ f, [f^*\psi \circ \varphi])$ indeed agrees with $S(g, [\psi]) \circ S(f,[\varphi])$ as a weighted cospan.
\end{proof}

\begin{lem}\label{lem:S_n-adjoint}
    For all $n$ the composite functor
    \[
        S_n: \int_{G \in \J} N_n^{\rm R} \Subpi(G) 
        \to N_n^{\rm R} \int_{G \in \J} \Subpi(G) 
        \xrightarrow{S} N_n^{\rm R} \mcF_{\ge 2}
    \]
    is an equivalence on classifying spaces.
    (In fact, it admits a fully faithful right adjoint.)
\end{lem}
\begin{proof}
    For a graph $G$ let $\Fil_n(G)$ denote the full subcategory of $N_n^{\rm R}(\Subpi(G))$ on those objects $(U_0 \to \dots \to U_n)$ where every edge of $|G|$ intersects $\bigcup_{i=0}^n \partial U_i$ non-trivially.
    This is functorial in graph isomorphisms and we have an inclusion functor 
    \[
        I_n: \int_{G \in \J^{\cong}} \Fil_n(G) \to \int_{G \in \J} N_n^{\rm R}(\Subpi(G)).
    \]
    This functor is fully faithful because if $U_\cd \in \Fil_n(G)$ and $f:G' \to G$ is a morphism such that $f^*U_\cd \in \Fil_n(G')$, then $f$ must be an isomorphism.
    (As $\bigcup_i \partial f^{-1}(U_i)$ intersects every edge of $G'$ the morphism $f$ cannot collapse any edges and hence must be an isomorphism.)

    Next we show that $I_n$ has a left adjoint.
    For an object $(G, (U_\cd, [\varphi_\cd])) \in \int_{G \in \J} N_n^{\rm R}\Subpi(G)$ we can construct its localization onto the image of $I_n$ by collapsing all edges of $G$ that do not intersect $\bigcup_i \partial U_i$.
    Let $G'$ denote the graph resulting from this collapse (with quotient map $f:G \to G'$) and let $U_i'$ be the quotients of the $U_i$ -- these are still admissible because their boundary did not intersect an edge we collapsed.
    Similarly, the $[\varphi_i]:U_{i-1} \to U_i$ descend to the quotient and yield $[\varphi_i']: U_{i-1}' \to U_i'$.
    Then, because $U_i = f^{-1}(U_i')$, the quotient map $f:G \to G'$ defines a map $f:(G, (U_\cd,[\varphi_\cd])) \to (G', (U_\cd', [\varphi_\cd']))$ in $\int_{G \in \J} N_n^{\rm R}\Subpi(G)$, which by construction is initial among maps from $(G, (U_\cd, [\varphi_\cd]))$ to objects in the essential image of $I_n$.
    Therefore, $I_n$ has a left adjoint $L_n$, and hence induces an equivalence on classifying spaces.

    To complete the proof it will now suffice to check that for all $n$ the functor
    \[
        \int_{G \in \J^{\cong}} \Fil_n(G) 
        \xrightarrow{I_n} \int_{G \in \J} N_n^{\rm R}\Subpi(G)
        \xrightarrow{S_n} N_n^{\rm R}\mcF_{\ge 2}
    \]
    is an equivalence of groupoids. 
    (In fact, one can check that $I_n \circ (S_n \circ I_n)^{-1}$ provides a fully faithful right adjoint to $S_n$ -- this uses that $S_n \cong S_n \circ I_n \circ L_n$.)
    For the remainder of the proof it will be convenient to write the objects of $\Fil_n(G)$ as $U_\cd = (U_0 \subset \dots \subset U_n)$, i.e.~to assume that the $\varphi_i$ are the identity.
    We can do this without loss of generality as every object in $\Fil_n(G)$ is isomorphic to one of these.
    Recall that an object in the groupoid $N_n^{\rm R}(\mcF_{\ge 2})$ can be written as a sequence of $n+2$ composable cospans
    \[
        \emptyset \to X_{-1} \leftarrow A_0 \to X_0 \leftarrow A_1 \to \dots \leftarrow A_n \to X_{n} \leftarrow \emptyset
    \]
    where the $X_i$ come equipped with labellings $c_i\colon X_i \to \IN$. 
    These diagrams are subject to the conditions:
    \begin{enumerate}
        \item all the $A_i$ are non-empty, 
        \item the total composite $X_{-1} \amalg_{A_0} \dots \amalg_{A_n} X_{n}$ is a point and its label is $\ge 2$, and
        \item the labellings $c_i$ respect the ``$\chi \le 0$'' condition, i.e.~if $c_i(x) = 0$ for some $x \in X_i$, then $x$ has at least two preimages in $A_i \amalg A_{i+1}$.
    \end{enumerate}
    The morphisms in this groupoid are the natural isomorphism of such diagrams.%
    \footnote{
    To be entirely precise, and we should identify any two objects between which there is an isomorphism that is the identity on the $A_i$. 
    However, as the other conditions imply that each $A_{i-1} \amalg A_i \to X_i$ is surjective, this does not remove any non-trivial automorphisms, and therefore does not change the groupoid up to equivalence.
    Hence, we can safely ignore it here.
    }
    The functor $S_n$ sends the filtered graph $(G, U_0 \subset \dots \subset U_n)$ to the diagram
    \[
        \emptyset \to \pi_0(U_0) \leftarrow  \partial U_0 \to \pi_0(U_1 \setminus U_0^\circ) \leftarrow \partial U_1 \to \dots \leftarrow \partial U_n \to \pi_0(|G| \setminus U_n^\circ) \leftarrow \emptyset
    \]
    where the labellings $c_i: \pi_0(U_i \setminus U_{i-1}^\circ) \to \IN$ are defined by sending a component $W \subset U_i \setminus U_{i-1}^\circ$ to the sum of all the vertex weights in $W$ plus the first Betti number of $W$.
    Note that because of the condition imposed on $U_\cd \in \Fil_n(G)$ we actually know that no such component $W$ can contain an entire edge.
    This means that $W$ is either a contractible neighbourhood of some vertex or a contractible subset of an edge.

    To construct an inverse functor to $S_n \circ I_n$ we first build from the diagram $(X_\cd, A_\cd, c) \in N_n^{\rm R}\mcF_{\ge2}$ the topological graph
    \[
        G^{\rm top} = X_{-1} \amalglim_{A_0} (A_0 \times [-1,0]) \amalglim_{A_0} X_0 \amalglim_{A_1} (A_1 \times [0,1]) \amalglim_{A_1} X_1  \dots 
        \amalglim_{A_{n}} (A_n \times [n,n+1]) \amalglim_{A_n} X_{n+1}
    \]
    as in \ref{defn:|W|} and \ref{rem:|W|}.
    This topological space comes with a projection $p$ to the interval $[-1,n+1]$ and it has a canonical filtration given by $U_i := p^{-1}([-1,i-\oh])$.
    We can write $G^{\rm top} = |G_1|$ for $G_1$ a graph with vertices $\coprod_{i=-1}^{n+1} X_i$ and edges $\coprod_{i=0}^n A_i$.
    With respect to the above filtration each component of $U_i \setminus U_{i-1}^\circ$ would then contain exactly one vertex.
    The $c_i\colon X_i \to \IN$ yield a labelling of the vertices of $G_1$ by natural numbers in such a way that univalent vertex are not labelled by $0$, and the total genus $b_1(G_1) + \sum_i \sum_{x \in X_i} c_i(x)$ is at least $2$.
    However, it is still possible for bivalent vertices to be labelled by $0$ as we allowed this in definition \ref{defn:weighted-cospans}.
    Let $G$ be the graph obtained by forgetting all the bivalent vertices of $G_1$ that are labelled by $0$ and joining their adjacent edges.
    (Because we assumed that the total genus is at least $2$, this will not remove all vertices.)
    Then $|G| \cong |G_1| \cong G^{\rm top}$, so $U_\bullet$ yields a well-defined element of $\Fil_n(G)$.
    This construction is functorial with respect to isomorphisms and it is inverse (up to natural isomorphism) to the functor $S_n \circ I_n$ we started with.
\end{proof}

We now have all the tools to show that $\SubJ \to \mcF_{\ge2}$ induces an equivalence on classifying spaces.
Rather than proving this directly, we first deduce the following stronger statement, as it clarifies the relation between $\J$, $\mcF_{\ge2}$ and $\Subpi$.

\begin{cor}\label{cor:colim-of-infinity-cats}
    The functor $S$ exhibits $\mcF_{\ge2}$ as the colimit of the diagram
    \[
        \J^\op \xrightarrow{\Subpi(-)} \Cat \to \Cat_\infty
    \]
    in the $\infty$-category of $\infty$-categories.
\end{cor}
\begin{proof}
    The Rezk nerve functor extends to a functor $N_\cd^{\mcR}:\Cat_\infty \to \Fun(\Delta^\op, \mcS)$ defined by sending an $\infty$-category $\mcC$ to the simplicial space $\Map_{\Cat_\infty}([\cd], \mcC)$.
    This is compatible with the Rezk nerve we have used so far in the sense that there is an equivalence $B N_n^{\rm R}(\mcD) \simeq N_n^{\mcR}(\mcD)$ natural in $[n] \in \Delta^\op$ and $\mcD \in \Cat_1$.
    Using this translation, lemma \ref{lem:S_n-adjoint} states that the functor
    \[
        \int_{G \in \J} N_n^\mcR(\Sub(G))
        \to N_n^{\mcR}(\SubJ)
        \xrightarrow{N_n^\mcR(S)} N_n^{\mcR}\mcF_{\ge 2}
    \]
    induces an equivalence on classifying spaces for all $n$.
    By Thomason's theorem \cite{Tho79}, the classifying space of the Grothendieck construction $\int_\J N_n^\mcR(\Sub(G))$ is the homotopy colimit of the functor $N_n^\mcR(\Sub(-))$.
    (In $\infty$-categorical language, the classifying space of (the source of) a right fibration computes the colimit of the corresponding presheaf \cite[Corollary 3.3.4.6]{LurHTT}.)
    Therefore we have a natural equivalence of spaces
    \[
        \colim_{G \in \J} N_\cd^\mcR \Subpi(G)
        \simeq N_\cd^\mcR \mcF_{\ge 2}.
    \]
    Because the Rezk nerve $N_\cd^{\mcR}$ is fully faithful \cite{Rez01,Hebestreit2025} it reflects all colimits (even though it does not preserve them), and hence we get an equivalence $\colim_{G \in \J} \Subpi(G) \simeq \mcF_{\ge2}$ in $\Cat_\infty$
\end{proof}

\begin{cor}\label{cor:Sub=mcF}
    The functor $\SubJ \to \mcF_{\ge 2}$ induces an equivalence on classifying spaces.
\end{cor}
\begin{proof}
    By \cite[Corollary 3.3.4.3]{LurHTT} the colimit of a diagram $\J^\op \to \Cat_\infty$ can equivalently be described as the localization of the unstraightening $\SubJ$ at the cartesian edges.
    Therefore corollary \ref{cor:colim-of-infinity-cats} implies that the functor $S: \SubJ \to \mcF_{\ge2}$ becomes an equivalence after inverting all cartesian edges in $\SubJ$.
    If we further invert all other morphisms on sides we get that $S$ is an equivalence on classifying spaces.
\end{proof}

\subsection{The genus $1$ case}

We still have to deal with the genus $1$ case of theorem \ref{thm:computing-mcF}.
That is, we have to show that there is a homotopy equivalence
\[
    |\mcF_1(\Cobn)| \simeq BO(2).
\]
An object in $\mcF_1(\Cobn)$ is a factorisation $(W\colon \emptyset \to M, V \colon M \to \emptyset)$ such that $W \cup_M N$ is a torus and $M \neq \emptyset$.
As all components must have non-positive Euler characteristic, this is only possible if both $W$ and $V$ are disjoint unions of cylinders.

One can show that $\mcF_1 = \mcF_1(\Cob_2^{\chi \le 0})$ is equivalent to $\mcF_{S^1}(\Cob_1^{\rm unor})$.
(Collapsing the cylinders to arcs yields a factorisation of a circle into arcs.)
Just like we showed in \cite[Section 8]{Stb21} that $\mcF_{S^1}(\Cob_1)$ 
is equivalent to Connes' cyclic category $\gL$,
it should be possible to show that $\mcF_{S^1}(\Cob_1^{\rm unor})$
is the dihedral version of Connes' cyclic category.
This dihedral category is known to have classifying space $BO(2)$, see \cite[Proposition 3.11]{Lod87}.

However, making this equivalence precise is rather tedious as we need to compare our geometric category to Connes' combinatorially defined category.
Therefore, we will instead explain how to adjust the proof from the previous section to this case.

\begin{defn}
    Let $\mcC$ be the category where objects are pairs $(C, U)$ of a circle $C$ and an admissible subset $\emptyset \neq U \subsetneq S^1$ (i.e.~a finite disjoint unions of closed intervals of positive length).
    A morphism $(C, U) \to (C', U')$ in this category is represented by a homeomorphism $\varphi: C \cong C'$ with $\varphi(U) \subset U'$ and two such homeomorphism represent the same morphism if they are isotopic through a family $\varphi_t$ with $\varphi_t(U) \subset V$.
\end{defn}

\begin{proof}[Proof of theorem \ref{thm:computing-mcF} for $g=1$]
As in definition \ref{defn:S}, we still have a functor 
\[
    S: \mcC \to \mcF_1.
\]
In this case, we can in fact show that it is an equivalence of categories, which we will do by showing that for all $n$ the induced map
\[
    S_n: N_n^{\rm R}(\mcC) \to N_n^{\rm R}(\mcF_1)
\]
is an equivalence of groupoids.
An inverse $R_n$ can be constructed in essentially the same way as in the second half of lemma \ref{lem:S_n-adjoint}, but with $N_n^{\rm R}\mcC$ playing the role of $\int_{G \in \J^{\cong}} \Fil_n(G)$.
Indeed, given an object $(X_\cd, A_\cd, c_\cd)$ of $N_n^{\rm R}(\mcF_1)$ we again construct the topological graph $G^{\rm top}$ with its canonical filtration $U_\bullet$.
The analysis of factorisations of the torus in $\Cobn$ at the start of the subsection shows that $G^{\rm top}$ must be a circle with labels $c_i \equiv 0$.
We can thus define $R_n(X_\cd, A_\cd, c_\cd) := (G^{\rm top}, U_\bullet)$ and this is functorial with respect to isomorphisms and it is inverse (up to natural isomorphism) to $S_n$.
Therefore, we in fact have that $\mcC \to \mcF_1$ is an equivalence of categories and in particular it induces an equivalence on classifying spaces.

The same argument as in lemma \ref{lem:Subpi-contractible} shows that the topological poset $\Subtop(S^1)$ of admissible subsets has a contractible classifying space.
The argument also shows that $B\mcC$ is equivalent to the realization of the action groupoid of $\Homeo(S^1)$ on $\Subtop(S^1)$.
Here the difference is that we do not restrict to the identity component of $\Homeo(S^1)$ and that this identity component is not contractible.
(However, it is still true that the space of homeomorphisms $\varphi\in \Homeo(S^1)$ satisfying $\varphi(U) \subset V$ has contractible components.)
Therefore we get that
\[
    B\mcC \simeq B\Subtop(S^1) /\!\!/ |\Homeo(S^1)| \simeq B\Homeo(S^1) \simeq BO(2). \qedhere
\]
\end{proof}

\section{Comparison to \texorpdfstring{$\Delta_g$}{Delta g}}\label{sec:J=gD}

In the previous sections we computed the classifying space $B(\ICobn)$
up to weak equivalence of infinite loop spaces and showed that it 
splits off free infinite loop spaces on the double-suspension of $B\J_g$.
Below we recall the moduli space $\Delta_g$ from tropical geometry and show that 
there is a rational homology equivalence $B\J_g \to \Delta_g$.
In fact, one can argue that $B\J_g$ is the homotopy type of a moduli \emph{stack} whose coarse space is $\Delta_g$.

One shortcoming of our results so far is that the splitting map $B(\ICobn) \to Q(\Sigma^2 B\J_g)$ is obtained though a rather long zig-zag of equivalences and maps,
making it difficult to understand how exactly the much-studied rational cohomology
of $\Delta_g$ embeds into the cohomology of $B(\ICobn)$.
The purpose of this section is to show that a very concrete model of this map can be constructed if we ``ignore finite automorphism groups''. 
Concretely, we will work up to rational homotopy equivalence, 
which leads to several simplifications:
we can work with the $1$-category $\Cobn$ instead of the $(2,1)$-category $\ICobn$,
we can map to the coarse tropical moduli space $\Delta_g$ instead of $B\J_g$,
and we can work with the free commutative topological monoid 
$\SP^\infty(X) = (\coprod_{n\ge 0} X^n/\Sigma_n)/\sim$ on a based space $X$
instead of the free infinite loop space $Q(X)$.

Recall from remark \ref{rem:BC-PCM}, that the partial commutative monoidal structure on $\Cobn$ (as in example \ref{ex:Cob-PCM}) induces a partially defined commutative monoid structure on the space $B(\Cobn)$ by
\begin{align*}
    &[(M_0 \xrightarrow{[W_1]} \dots \xrightarrow{[W_n]} M_n), (t_0, \dots, t_n)]
     +
    [(N_0 \xrightarrow{[V_1]} \dots \xrightarrow{[V_n]} N_n], (t_0, \dots, t_n)] \\
    &=[(M_0 \cup N_0 \xrightarrow{[W_1 \amalg V_1]} \dots \xrightarrow{[W_n \amalg V_n]} M_n \cup  N_n], (t_0, \dots, t_n)] 
\end{align*}
whenever $M_i \cap N_i = \emptyset$, and undefined otherwise.
We will construct a continuous map of partially defined 
commutative topological monoids as described in figure \ref{fig:mu} in the introduction:
\[
    \mu: B\left(\Cobn\right) \longrightarrow
    \SP^\infty\left( S^2(\Delta_2) \vee S^2(\Delta_3) \vee \dots) \right) .
\]
Here $S^2(-)$ denotes the unreduced double-suspension.
The goal of this section is to show that $\mu$ corresponds,
up to rational equivalence, to the projection onto the third factor
in theorem \ref{theorem:BICobn}; in the following sense:

\begin{thm}\label{thm:mu-fits}
    The map $\mu$ fits into a commutative diagram 
    of partially defined commutative topological monoids:
\[
\begin{tikzcd}[column sep = 0pc]
    B\left(\ICobn\right) \ar[d, "\simeq_\IQ"] \ar[rr, "\simeq", "\text{Thm \ref{theorem:BICobn}}"'] &{\hphantom{\dots}}&
    {B\IZ \times |B\Cut_\cd^{(g=1)}| \times |B\Cut_\cd^{(g\ge2)}|} \ar[r, "\text{project}"] &
    {|B\Cut_\cd^{(g \ge 2)}|} \ar[d, "\simeq_\IQ"] \\
    B\left(\Cobn\right) \ar[rrr, "\mu"] &&&
    \SP^\infty\left( S^2(\Delta_2) \vee S^2(\Delta_3) \vee \dots) \right) 
\end{tikzcd}
\]
    where the maps labelled by $\simeq$ and $\simeq_\IQ$ 
    are (rational) homotopy equivalences.
\end{thm}

\begin{rem}
    To compare this to the formulation of theorem \ref{theorem:BICobn}
    given in the introduction, 
    note that by combining corollary \ref{cor:Cut=Q(SBF)} with theorem \ref{thm:computing-mcF} we get:
    \[
        |B\Cut_\cd^{(g=1)}| \simeq Q(\gS^2 BO(2)) 
        \qand
        |B\Cut_\cd^{(g\ge2)}| \simeq Q\big({\bigvee}_{g \ge 2} \gS^2 B\J_g\big) .
    \]
    The specific map $B(\ICobn) \to B\IZ \times |B\Cut_\cd^{(g=1)}| \times |B\Cut_\cd^{(g \ge 2)}|$ 
    in theorem \ref{thm:mu-fits} is obtained as follows.
    The maps to the second and third factor come from the simplicial map $N_\cd\ICobn \to D_\cd \to \Cut_\cd$ 
    (denoted $R_\cd$ in lemma \ref{lem:Cut-is-simplicial})
    that discards all non-closed components of the $n$-simplex. 
    The map to the first factor comes from the functor $F:\ICobn \to \IZ$ 
    that sends a cobordism $W:M \to N$ to the integer 
    $\oh(|\pi_0(N)| - |\pi_0(M)| - \chi(W))$.
    This is the homotopy splitting of $\Cobnpb \subset \Cobn$ constructed in 
    proposition \ref{prop:computing-pb}.
    Note that $B\IZ$ is a topological abelian group and $BF:B(\ICobn) \to B\IZ$
    sends disjoint unions to sums.
\end{rem}

The Dold-Thom theorem \cite{DT58} states that there is a natural isomorphism
$\pi_*\SP^\infty(X) \cong \tilde{H}_*(X)$ between the homotopy groups 
of the infinite symmetric power of $X$ and the reduced homology of $X$. 
Hence $\mu$ induces a map from $\pi_n B(\Cobn)$ to 
$\bigoplus_{g \ge 2} H_{n-2}(\Delta_g; \IZ)$.
\begin{cor}
    On rational homotopy groups $\mu$ induces a surjection:
    \[
        \pi_*^\IQ(\mu): \pi_*^\IQ B(\Cobn) \twoheadrightarrow 
        \bigoplus_{g \ge 2} H_{*-2}(\Delta_g; \IQ)
    \]
    the kernel of which is spanned by a single class $\alpha$ in degree one
    and a class $\rho_i$ in each degree $4i+2$.
\end{cor}
\begin{proof}
    We have $\pi_{*>0}^\IQ(S^1) \cong \IQ\langle \alpha \rangle$, $|\alpha| = 1$ and $\pi_{*>0}^\IQ(BO(2)) \cong \IQ\langle \rho_1, \rho_2, \dots,\rangle$, $|\rho_i|=4i$.
\end{proof}

\subsection{Constructing the maps}
We now recall the definitions used above, starting with 
the moduli space $\gD_g$, which we define
as a colimit indexed by $\J_g$ following \cite{CGP16}

\begin{defn}
    Consider the functor $\gD^E: \J^\op \to \Top$ that sends $(G,w)$ to the space
    \[
        \gD^{E_G} = \{ d: E_G \to [0,1] \;|\; \sum_{e \in E_G} d(e) = 1\}
    \]
    and where the functoriality is by extension by $0$.
    For $g \ge 2$ we define the 
    \emph{moduli space of tropical curves of genus $g$ and volume $1$}
    $\gD_g$ as the colimit $\colim_{(G,w) \in \J_g^\op} \gD^E(G,w)$.
\end{defn}

A point in $\gD_g$ is represented by a stable graph $G \in \J$ of genus $g$
together with a function $d:E_G \to [0,1]$, which we think of as recording 
the lengths of the edges. Every point in $\gD_g$ has a representative where 
the edge lengths are strictly positive, and in this case there is a unique metric 
on the topological space representing $G$ such that any edge $e \subset G$
is isometric to $[0, d(e)] \subset \IR$.
The underlying set of $\gD_g$ can hence be identified with the set of
of stable metric graphs up to weight-preserving isometry.

\begin{rem}\label{rem:our-gD_g}
    In \cite{CGP16} the authors first define a moduli space $M_g^{\rm trop}$
    as the colimit over the functor $\gs:J_g \to \Top$ 
    that sends $G$ to $\gs(G) = \IR_{\ge 0}^{E_G} = \{l:E_G \to \IR_{\ge 0}\}$.
    This colimit has a natural map to $\IR_{\ge 0}$ defined by taking the
    sum of all edge-lengths: this is the volume of the tropical curve.
    The moduli space $\gD_g$ is then defined as the subspace of $M_g^{\rm trop}$
    containing the tropical curves of volume $1$.
    This is canonically homeomorphic to our definition via:
    \begin{align*}
        \gD_g 
        & \stackrel{\text{\cite{CGP16}}}{=} 
        \{1\} \times_{\IR_{\ge0}} \left(\colim_{G \in J_g} \IR_{\ge 0}^{E_G}\right)
        \cong \colim_{G \in J_g} 
        \left(\{1\} \times_{\IR_{\ge0}} \IR_{\ge 0}^{E_G}\right)\\
        & \cong \colim_{G \in J_g} \gD^{E_G}
        \cong \colim_{G \in \J_g} \gD^{E_G}.
    \end{align*}
    Here the last homeomorphism uses that $\J_g \simeq J_g \setminus \{\cd_g\}$ (see remark \ref{rem:our-J_g})
    and $\gD^E(\cd_g) = \emptyset$.
\end{rem}

\begin{defn}\label{defn:Phi}
    The map $\Phi: B\mcF_{\Sigma_g}(\Cobn) \to \gD_g$ is defined as follows;
    compare figure \ref{fig:Phi}.
    A point in $B\mcF_{\Sigma_g}(\Cobn)$ is represented by a tuple 
    \[
        [(\emptyset \xrightarrow{[W_0]} M_0 \xrightarrow{[W_1]} \dots M_n \xrightarrow{[W_{n+1}]} \emptyset), (t_0, \dots, t_n)]
    \]
    whose first entry is an element of $N_n\mcF_{\Sigma_g}(\Cobn) \subset N_{n+2}\Cobn$
    and whose second entry is a point in the topological $n$-simplex $\gD^n$.
    Define a graph $G_0$ with vertices $\coprod_{i=0}^{n+1} \pi_0(W_i)$
    and half-edges $\{\uparrow, \downarrow\} \times \coprod_{i=0}^n \pi_0(M_i)$.
    The involution $i$ swaps $\uparrow$ and $\downarrow$.
    For $a \in \pi_0(M_i)$ the root of $(\uparrow, a)$ is the image of $a$ under $\pi_0(M_i) \to \pi_0(W_{i+1})$ and the root of $(\downarrow, a)$ is the image of $a$ under $\pi_0(M_i) \to \pi_0(W_{i})$.
    Define $w:V_{G_0} \to \IN$ by sending $[U] \in \pi_0(W_i)$ to the genus of the subsurface $U \subset W_i$.
    Moreover, we define an edge length function $d:E_{G_0} \to \IR_{\ge0}$
    by $d(a) = t_i \cdot |\pi_0(M_i)|^{-1}$ for $a \in \pi_0(M_i)$.
    The $\chilez$ condition ensures that every univalent vertex in $G_0$ has label at least $1$, but there can still be bivalent vertices labelled by $0$.
    We let $(G,w,d)$ denote the stable metric graph obtained deleting the bivalent vertices of weight $0$ from $(G_0,w)$, glueing the adjacent edges, and adding their lengths in the process.
    The graph $(G,w)$ lies in $\J_g$ because it is stable by construction and it has total genus $g$ because $G_0$ did.
    Now we set $\Phi((M_\cd, [W_\cd]), t_\cd) := [G,w,d]$ in $\gD_g$.
\end{defn}

\begin{figure}[ht]
    \centering
    \def\svgwidth{\figurerescalefactor\linewidth}
    \small
    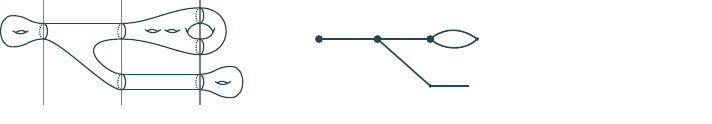
    \caption{An example of how the map $\Phi: B(\mcF_{\Sigma_5}(\Cobn)) \to \gD_5$ 
    from definition \ref{defn:Phi} can be evaluated on a $2$-simplex.
    The $2$-simplex is parametrised by $(t_0, t_1, t_2) \in [0,1]^3$ with $t_0+t_1+t_2=1$.
    $\Phi$ sends this to the weighted metric graph with one vertex per component 
    in each morphism $W_i$, weighted by the genus of this component,
    and one edge of length $\tfrac{1}{|\pi_0M_i|}t_i$ per circle in each object $M_i$.
    Afterwards, all valence $2$ and genus $0$ vertices are deleted 
    and the length of their adjacent edges is added.}
    \label{fig:Phi}
\end{figure}

We briefly check that this map is well-defined.
For each $n$ simplex in the nerve $N_\cd\mcF_{\Sigma_g}(\Cobn)$ the process in definition \ref{defn:Phi} yields a continous map $|\Delta^n| \to \Delta_g$, but we need to check that these glue under face and degeneracy maps.
For degeneracy maps, we need to show that if we write $t_i = t_i^0 + t_i^1$, then
    \[
        [(\emptyset \xrightarrow{[W_0]} M_0  \to \dots \to M_i \xrightarrow{\id_{M_i}} M_i \to  \dots \to M_n \xrightarrow{[W_{n+1}]} \emptyset), (t_0, \dots, t_i^0, t_i^1, \dots, t_n)]
    \]
is still sent to the same point $[G,w,d] \in \Delta_g$.
This is indeed the case, because inserting the identity morphism only results in several bivalent vertices of weight $0$ in $G_0$, which we delete in passing to $G$.
When deleting these vertices we also add adjacent edge-lengths which recovers the same function $d$ since $t_i^0 + t_i^1=t_i$.
For face maps, we need to check that if we have $t_i = 0$ in $[(M_\cd,[W_\cd]),t_\cd)]$, then
    \[
        [(\emptyset \xrightarrow{[W_0]} M_0  \to \dots \to M_{i-1} \xrightarrow{[W_i \cup W_{i+1}]} M_{i+1} \to  \dots \to M_n \xrightarrow{[W_{n+1}]} \emptyset), (t_0, \dots, t_{i-1}, t_{i+1}, \dots, t_n)]
    \]
is still sent to $[G,w,d] \in \Delta_g$.
The graph $G_0'$ that we get here is obtained from the original $G_0$ by collapsing all edges corresponding to $\pi_0(M_i)$, and we get a map $f:G \to G'$ in $\J_g$.
Because $t_i=0$ we also have that that some edges in $G_0$ are of length $0$, and the resulting edge length function on $G'$ satisfies $d' = f_*d$.
Therefore, $[G',w',d']$ and $[G,w,d]$ represent the same point in $\Delta_g = \colim_{G \in \J} \Delta^{E_G}$.

\begin{defn}
    For a space $X$ let $S^2(X)$ denote the unreduced double-suspension of $X$.
    As we saw in example \ref{ex:cone-and-unreduced-suspension}, this is homeomorphic to $\Sigma\mathrm{Cone}(X \to *)$, so if $X$ is non-empty, we can parametrise the unreduced double-suspension as 
    \[
        S^2(X) = \{ (x, a, b) \in X \times [0,1]^2 \;|\; a+ b \le 1\} / \sim
    \]
    where $(x,a,0) \sim (x,0,0) \sim (x,0,a)$ and $(x,a,1-a) \sim (y,a,1-a)$ for all $a \in [0,1]$ and $x,y \in X$.
\end{defn}

\begin{defn}
    For a based space $(Y, y_0)$ let $\SP^\infty(Y, y_0)$ denote
    the infinite symmetric power of $(Y, y_0)$.
    The underlying set of $\SP^\infty(Y, y_0)$ is the free abelian monoid on the set $Y$
    modulo the relation $[y_0] = 0$.
    This is topologised with the quotient topology with respect to the map
    \[
        \coprod_{n \ge 0} Y^n \to \SP^\infty(Y, y_0), \quad
        (y_1, \dots, y_n) \mapsto \sum_{i=1}^n [y_i].
    \]
\end{defn}

\begin{defn}\label{defn:mu}
    We construct a continuous map:
    \begin{align*}
        &\mu: B(\Cobn) \longrightarrow
        \SP^\infty\left( S^2( \Delta_2 \amalg \Delta_3 \amalg \dots) \right) \\
        &[(M_0 \xrightarrow{[W_1]} \dots \xrightarrow{[W_{n}]} M_n), (t_0, \dots t_n)]  \\
        &\longmapsto
        \sum_{U \subset W_1 \cup \dots \cup W_n \atop \text{ closed of genus }\ge 2}
        [[G_U,w_U,d_U], t_0 + \dots + t_{a_U-1}, t_{b_U+1} + \dots + t_n ] 
    \end{align*}
    The sum runs over all connected components $U \subset W_1 \cup \dots \cup W_n$
    that are closed surfaces of genus at least $2$.
    Here $1 \le a_U \le b_U \le n$ are the smallest and largest number, respectively, 
    such that $U \cap M_{a_U} \neq \emptyset \neq U \cap M_{b_U}$
    and $[G_U, w_U, d_U] \in \Delta_{g(U)}$ is defined using the map $\Phi$
    from definition \ref{defn:Phi}:
    \begin{align*}
        &[G_U,w_U,d_U] := \\
        &\Phi\left(
        \big(\emptyset \xrightarrow{[W_{a_U} \cap U]} M_{a_U}\cap U
        \xrightarrow{[W_{a_U+1} \cap U]} 
        \dots \to M_{b_U} \cap U \xrightarrow{[W_{b_U+1}\cap U]} \emptyset\big),
        (t_{a_U}, \dots, t_{b_U})
        \right) \in \Delta_{g(U)}
    \end{align*}
\end{defn}

One can check that this yields a well-defined continous map by combining the arguments below definition \ref{defn:Phi} with the arguments in lemma \ref{lem:Cut-con=gS}.
Alternatively, when concluding the proof of theorem \ref{thm:mu-fits} at the end of this section, we will obtain this map as the composite of other maps, so it suffices to check the well-definedness of those.

\subsection{The map \texorpdfstring{$B\mcF_{\Sigma_g}(\Cobn) \to \gD_g$}{BF_g to Delta_g} is an isomorphism on rational homology}

In this subsection we will define maps $\Phi_1$, $\Phi_2$ and $\Phi_3$ 
making the following diagram commute up to homotopy:
\begin{equation}\label{eqn:Phi-maps}
\begin{tikzcd}
    &{B\left( \int_{\J_g} \Sub\right)} \ar[d, "\simeq"] \ar[r, "\simeq"] &
    B\J_g \ar[d, "{\Phi_2}"] &
    {B\left(\int_{\J_g} \Delta_{\rm top}^E\right)} \ar[l, "\simeq"'] \ar[dl, "{\Phi_3}"] \\
    B\mcF_{\Sigma_g}(\Cobn) \ar[r, "\simeq"] \ar[rr, "\Phi", bend right = 20] &
    B\mcF_g\left(\Csp(\IN, \IN_{>0}, 1)\right) \ar[r, "\Phi_1"]&
    \Delta_g &
\end{tikzcd}
\end{equation}
The maps labelled by $\simeq$ are equivalences. 
For the bottom left map this follows from the equivalence $\Cobn \simeq \Csp(\IN, \IN_{>1},1)$
constructed in lemma \ref{lem:Cob2=labelled-csp}.
For the two maps in the middle this was shown in our proof of theorem \ref{thm:computing-mcF},
which stated that $B\mcF_g\left(\Csp(\IN, \IN_{>0}, 1)\right) \simeq B\J_g$ via this zig-zag.
For the horizontal map on the right it follows from the observation that each
$\Delta_{\rm top}^{E_G}$ is contractible and that hence $\int_{\J_g} \Delta_{\rm top}^E \to \J_g$
is a level-wise equivalence on nerves.
Once the diagram is established we show that $\Phi_3$ is an isomorphism on rational homology by applying arguing that it compares a homotopy colimit with a strict colimit in a setting where the difference can be described in terms of classifying spaces of finite groups.
It then it follows that $\Phi_1$ and $\Phi_2$ are also isomorphisms on rational homology.

\begin{defn}
    The maps $\Phi_i$ in diagram \ref{eqn:Phi-maps} are defined as follows.
\begin{itemize}
    \item 
    The map $\Phi_3: B\left(\int_{\J_g} \Delta_{\rm top}^E\right) \to \Delta_g$ is defined by 
    \[
        [(G_0 \to G_1 \to \dots \to G_n), d \in \Delta_{\rm top}^{E_{G_n}}, (t_0, \dots, t_n)]
        \longmapsto
        [G_n, d].
    \]
    This is the canonical comparison map from the homotopy colimit
    to the strict colimit.
    
    \item
    The map $\Phi_2:B\J_g \to \Delta_g$ is defined by
    \[
        [(G_0 \xrightarrow{f_1} G_1 \to \dots \xrightarrow{f_n} G_n), (t_0, \dots, t_n)]
        \longmapsto
        [G_0, d^{f}]
    \]
    where the edge length function $d^{f} \in \Delta_{\rm top}^{E_{G_0}}$ is given as
    \[
        d^{f}(e) = \sum_{i=0}^{m_e} \frac{t_i}{|E_{G_i}|}
        \quad \text{ where } 
        m_e = \max\{i \le n \;|\; G_0 \to G_i \text{ does not collapse } e\}.
    \]
    
    \item
    The map $\Phi_1: B\mcF_g\left(\Csp(\IN, \IN_{>0}, 1)\right) \to \gD_g$ is defined 
    in analogy with definition \ref{defn:Phi} as
    \[
        [([X_0] \leftarrow A_0 \to [X_1] \leftarrow \dots \leftarrow A_n \to [X_{n+1}]), (t_0, \dots t_n)]
        \longmapsto [G, d].
    \]
    Here we first build a graph $G_0$ with vertices $\coprod_{i=0}^{n+1} X_i$ and half-edges $\{\uparrow,\downarrow\} \times \coprod_{i=0}^n A_i$, so that for each $a \in A_i$ there is an edge connected to the two vertices corresponding to the images of $a$ under $X_i \leftarrow A_i \to X_{i+1}$.
    The edge length function $d:E_G \to \IR_{\ge0}$ is 
    $d(a) = t_i \cdot |A_i|^{-1}$ for $a \in A_i$.
    Then $[G,d]$ is obtained by deleting bivalent vertices and adding the adjacent edge-lengths.
\end{itemize}
\end{defn}

We briefly check that $\Phi_2$ is well-defined.
Firstly, the total volume of $d^f$ is
\[
    \sum_{e \in E_{G_0}} d^f(e)
    = \sum_{e \in E_{G_0}} \sum_{i=0}^{m_e} \frac{t_i}{|E_{G_i}|}
    = \sum_{i=0}^{n} \sum_{e \in E_{G_i}} \frac{t_i}{|E_{G_i}|}
    = \sum_{i=0}^{n} t_i = 1.
\]
Secondly, if one of the $t_j$ is $0$, then the $j$th summand in $\sum_{i=0}^{m_e} \tfrac{t_i}{|E_{G_i}|}$ never contributes to the sum, so it does not change the end result if we forget about $G_i$ and compose $f_{i+1}$ with $f_i$.
Finally, if one of the $f_j$ is the identity, then $\tfrac{t_i}{|E_{G_i}|}$ contributes to $d^f(e)$ if and only if $\tfrac{t_{i-1}}{|E_{G_{i-1}}|}$ does, so it does not change the end result if we forget about $f_i$ and add $t_i+t_{i-1}$.

\begin{lem}
    The maps $\Phi_1$, $\Phi_2$, and $\Phi_3$ make the diagram on page \pageref{eqn:Phi-maps} commute up to homotopy.
\end{lem}
\begin{proof}
    First, consider the bigon involving $\Phi$ and $\Phi_1$.
    This commutes on-the-nose, basically because $\Phi$ was defined by combining $\Phi_1$ 
    with the functor $\Cobn \to \Csp(\IN, \IN_{>0}, 1)$ from lemma \ref{lem:Cob2=labelled-csp}.

    Next, consider the triangle on the right. The two maps 
    $B\left(\int_{\J_g} \Delta_{\rm top}^E\right) \to \Delta_g$ are given by
    \[
        [(G_0 \to G_1 \to \dots \to G_n), d \in \Delta_{\rm top}^{E_{G_n}}, (t_0, \dots, t_n)]
        \longmapsto
        [G_n, d] 
        \text{ or }
        [G_0, d^f]
        \text{, respectively.}
    \]
    We can rewrite $[G_n, d] = [G_0, f^*d]$ where $f^*d \in \Delta_{\rm top}^{E_{G_0}}$ 
    assigns $0$ to each edge that is collapsed by $f:G_0 \to G_n$
    and $d(f(e))$ to any other edge.
    The homotopy can then be defined as 
    \[
        H_\gl: [(G_0 \to G_1 \to \dots \to G_n), d \in \Delta_{\rm top}^{E_{G_n}}, (t_0, \dots, t_n)]
        \longmapsto
        [G_0, \gl f^*d + (1-\gl) d^f],
    \]
    which yields a family of maps
    $H_\gl: B\left(\int_{\J_g} \Delta_{\rm top}^E\right) \to \Delta_g$ 
    depending continuously on $\gl \in [0,1]$.
    
    Finally, we consider the square on the left. 
    The map 
    $B\left(\int_{\J_g} \Sub\right) \to B\mcF_g\left(\Csp(\IN, \IN_{>0}, 1)\right)$
    induced by the functor $S$ from definition \ref{defn:S} is given by (here we are suppressing the $\varphi_i$)
    \begin{align*}
        &[(G_0 \xrightarrow{f_1} G_1 \to \dots \xrightarrow{f_n} G_n), (U_i \in \Subpi(G_i)), (t_0, \dots, t_n)]\\
        & \longmapsto [(\emptyset \xrightarrow{[\pi_0(U_0)]} \partial U_0 \xrightarrow{[\pi_0(f_1^{-1}(U_1) \setminus U_0)]} \dots \to \partial U_n \xrightarrow{[\pi_0(|G_n|\setminus U_n)]} \emptyset), (t_0, \dots, t_n)] .
    \end{align*}
    If we further apply $\Phi$, then we first get a graph with vertices $\coprod_{i=-1}^n \pi_0(f_i^{-1}(U_i) \setminus U_0)$ and edges $\coprod_{i=0}^n \partial U_i$, where an edge in $\partial U_i$ has length $\tfrac{t_i}{|\partial U_i|}$.
    (We can think of this graph as being obtained by subdividing the topological graph $|G_0|$.)
    This graph might still have some bivalent genus $0$ vertices and after deleting those and gluing the adjacent edges, we in fact recover a graph canonically isomorphic to $G_0$.
    The resulting edge length function $d^U$ on $G_0$ is then obtained by taking a weighted sum of how often a given edge (though of as an open interval $e \subset |G_0|$) intersects each $\partial U_i$:
    \[
        d^U(e) = \sum_{i=0}^n t_i \frac{|e \cap (f_i \circ \dots \circ f_1)^{-1}(\partial U_i)|}{|\partial U_i|} .
    \]
    The composite $B\left(\int_{\J_g} \Sub\right) \to B\mcF_g\left(\Csp(\IN, \IN_{>0}, 1)\right) \to \Delta_g$ hence sends $[G_\cd, U_\cd, t_\cd]$ to $[G_0, d^U]$.
    The other composite in the square simply sends $[(G_\cd), (U_\cd), (t_\cd)]$
    to $[G_0, d^f]$ with $d^f$ as in the definition of $\Phi_2$.
    Now that we have written the value of both maps as $G_0$ with a certain metric,
    we can again use affine linear interpolation of metrics to construct the desired homotopy:
    $[G_0, \lambda d^U + (1-\lambda) d^f]$.
\end{proof}

The map $\Phi_3$ can be thought of as the comparison between the homotopy colimit and the colimit of a diagram.
We now establish a general criterion for when such a comparison map is an isomorphism on rational homology.

\begin{lem}\label{lem:rational-colim}
    Let $\mcI$ be a category with a conservative functor $d:\mcI \to (\IN, \le)$ and such that the automorphism group $\Aut(i_0)$ is finite for all $i_0 \in \mcI$.
    Suppose that $F: \mcI \to \sSet$ is a diagram such that for all $i_0 \in \mcI$ the map
    \[
        \colim_{\substack{i \in \mcI_{/i_0},  i \not\cong i_0}} F(i) \to F(i_0)
    \]
    is a monomorphism, i.e.~it is level-wise injective.
    Then the comparison map 
    \[
        \hocolim_{i \in \mcI} |F(i)| \to \colim_{i \in \mcI} |F(i)|
    \]
    of topological spaces is an isomorphism on rational homology.
\end{lem}
\begin{proof}
    We can equip $\mcI$ with the structure of a generalized Reedy category in the sense of \cite{Berger2010-generalized-Reedy}:
    its degree function is $d$, we set $\mcI^- = \mcI^{\cong}$ and $\mcI^+ = \mcI$.
    (This satisfies the axioms of a generalized Reedy category \cite[Definition 1.1]{Berger2010-generalized-Reedy}, but it is not necessary ``dualizable''.)
    Let $\sSet$ be the category of simplicial sets with the Quillen model structure, $\Top$ the category of compactly generated weak Hausdorff spaces with its Serre model structure, and $\Ch_\IQ$ the category of rational chain complexes with the model structure where the weak equivalences are the quasi-isomorphisms and where a morphism is a cofibration if it is level-wise injective.
    Taking rational chains defines a left Quillen functor $C_*:\sSet \to \Ch_\IQ$ and taking geometric realization defines a left Quillen functor $|-|:\sSet \to \Top$.
    In fact, we have a commutative square of left adjoints
    \[\begin{tikzcd}
        {\Top^\mcI}\dar["\colim"'] & {\sSet^{\mcI}} \lar["|-|"'] \rar["C_*"] \dar["\colim"']& \Ch_{\IQ}^\mcI \dar["\colim"] \\
        {\Top} & {\sSet} \lar["|-|"'] \rar["C_*"] \rar & \Ch_\IQ
    \end{tikzcd}\]
    where we use the Reedy model structures in the top row.
    All of these functors are Quillen left adjoints; for the vertical functors this is the content of \cite[Corollary 1.7]{Berger2010-generalized-Reedy}, and for the top horizontal functors it follows from definition of the model structure \cite[above Theorem 1.6]{Berger2010-generalized-Reedy} that they in fact preserve all weak equivalences and cofibrations.

    The diagram $F$ defines an object in $\sSet^{\mcI}$, and we would like to show that for any cofibrant resolution $w:F' \to F$ the resulting map $\colim |F'| \to \colim |F|$ is an isomorphism on rational homology.
    By the above square it suffices to check that $\colim C_*(F') \to \colim C_*(F)$ is a weak equivalence in $\Ch_\IQ$.
    The functor $C_*(-):\sSet^\mcI \to \Ch_\IQ^\mcI$ preserves all weak equivalences, so $C_*(w)$ is a weak equivalence and $C_*(F')$ is cofibrant because $F'$ was.
    By Ken Brown's lemma the left Quillen functor $\colim$ preserves weak equivalences between cofibrant objects, so we just have to argue that $C_*(F)$ is cofibrant. (Recall that $F$ itself was not necessarily cofibrant!)

    The cofibrancy condition for $C_*(F)$ requires that for all $i \in \mcI$
    \[
        \colim_{\substack{i \in \mcI_{/i_0} , i \not\cong i_0}} C_*F(i) \to C_*F(i_0)
    \]
    is a cofibration in the projective model structure on $\Ch_\IQ^{\Aut(i_0)}$.
    This projective model structure is characterized by its weak equivalences (quasi-isomorphisms after forgetting the group action) and fibrations (surjections) and hence we see that it agrees with the projective model structure on $\Ch_{\IQ[\Aut(i_0)]}$ of chain complexes over the group ring.
    Recall that a morphism $A \to B$ of bounded-below $\IQ[\Aut(i_0)]$ chain complexes is a cofibration if each $A_n \to B_n$ is a split injection with projective kernel \cite[Lemma 2.3.6 and Proposition 2.3.9]{Hovey2007-model-categories}.
    By Maschke's theorem this is true for any injective map (of bounded below chain complexes).
    Therefore, it suffices to show that the map
    \[
        C_*\big(\colim_{\substack{i \in \mcI_{/i_0} , i \not\cong i_0}} F(i)\big)
        \cong \colim_{\substack{i \in \mcI_{/i_0} , i \not\cong i_0}} C_*F(i) 
        \to C_*F(i_0)
    \]
    is degree-wise injective, which follows from our hypothesis.
\end{proof}

\begin{lem}\label{lem:Phi3}
    The map $\Phi_3:B\left(\int_{\J_g} \Delta_{\rm top}^E\right) \to \Delta_g$ is an isomorphism on rational homology.
\end{lem}
\begin{proof}
    The classifying space $B(\int_{\J_g^\op} \Delta_{\rm top}^E)$ is exactly the Bousfield--Kan formula for the homotopy colimit in $\Top$ and the map
    \[
        \Phi_3: B\left(\int_{\J_g} \Delta_{\rm top}^E\right) \cong \hocolim_{G \in \J_g^\op} \Delta_{\rm top}^{E_G} \to \colim_{G \in \J_g^\op} \Delta_{\rm top}^{E_G} = \Delta_g
    \]
    is the canonical comparison map.

    We would like to apply lemma \ref{lem:rational-colim}.
    For this, we can equip $\J_g^\op$ with the conservative functor $\J_g^\op \to \IN$ that sends $G$ to its number of edges.
    This is functorial because graph maps reduce the number of edges, and it is conservative because if $f:G \to G'$ is a morphism between two graphs with the same number of edges, then $f$ must be a bijection on edges and thus an isomorphism.
    Consider the functor
    \begin{align*}
        F\colon \J_g^\op &\to \sSet\\
        G& \mapsto N_\cd((\J_g^\op)_{/G}).
    \end{align*}
    For each $G \in \J_g$ the category $(\J_g^\op)_{/G}$ can be identified with the poset of non-empty subsets of $E_G$ via the functor that sends $f:G\to G'$ to the subset of edges that are not collapsed by $f$ -- equivalently this is the poset of simplices of the topological simplex $\Delta_{\rm top}^{E_G}$.
    Its nerve $F(G) = N_\cd((\J_g)_{G/})$ is exactly the edgewise subdivision $\mrm{sd}(\Delta_{\rm top}^{E_G})$ of this simplex and in fact $|F(G)| \cong \Delta_{\rm top}^{E_G}$ as functors of $\J_g^\op \to \Top$.

    To see that the functor $F$ satisfies the condition of lemma \ref{lem:rational-colim} we need to check that for all $G \in \J_g$ the map
    \[
        \colim_{G' \in (\J_g^\op)_{/G}^{\not\cong G}} N_\cd((\J_g^\op)_{/G'})
        \to N_\cd((\J_g^\op)_{/G})
    \]
    is level-wise injective.
    The left side is the colimit of $\mrm{sd}(\Delta_{\rm top}^{I})$ over all non-empty proper subsets $I \subset E_G$.
    This exactly yields the edge-wise subdivision of the boundary $\partial \Delta_{\rm top}^{E_G}$, which indeed injects into the edge-wise subdivision of the simplex $\Delta_{\rm top}^{E_G}$.
\end{proof}

\begin{rem}
    Alternatively, one prove lemma \ref{lem:Phi3} using the Vietoris--Begle theorem \cite{Vietoris1927,Begle1950}.
    The key idea here is that it is not so difficult to see that the fibers of $\Phi_3$ are classifying spaces of finite groups.
    The Vietoris--Begle theorem says that if the fibers of a closed map between paracompact Hausdorff spaces have the rational Alexander-cohomology of a point in a range, then the map induces an isomorphism on rational Alexander-cohomology in that range, see \cite[p. 344, Theorem 15]{Spanier1981}.
    (While $\Phi_3$ is not closed, we can argue that for all $n \gg 0$ the map $|N_\cd^{(n)} \int_{\J_g^\op} \Delta_{\rm top}^E| \to \Delta_g$ from the $n$-skeleton is an isomorphism on rational homology groups up to degree $n-1$.)
    To obtain the statement about singular (co)homology that we want, we can for example check that all the spaces involved (including the fibers) are locally contractible, see \cite[p. 339f, Theorem 1 and Example 2]{Spanier1981}.
\end{rem}

\subsection{The map into the symmetric power}

In this section we prove theorem \ref{thm:mu-fits} by constructing 
the following commutative diagram of partially defined commutative 
topological monoids:
\begin{equation}\label{eqn:mu-diagram}
\begin{tikzcd}
    {B\ICobn} \ar[r] \ar[d, "\simeq_\IQ"] & 
    {|B\Cut_\cd^{(g \ge 2)}|} \ar[r, "\simeq_\IQ"] \ar[d] & 
    \SP^\infty(|B\Cut_\cd^{\con, (g \ge 2)}|) \ar[d] \ar[rd, bend left, "\simeq_\IQ"] & \\
    {B\Cobn} \ar[r] & 
    {|\pi_0\Cut_\cd^{(g \ge 2)}|} \ar[r] & 
    \SP^\infty(|\pi_0\Cut_\cd^{\con, (g \ge 2)}|) \ar[r, dashed] &
    \SP^\infty(\bigvee_{g \ge 2} S^2(\Delta_g))
\end{tikzcd}
\end{equation}
To clarify the notation used in the diagram, recall that 
$\Cut_\cd$ is a simplicial object in partial commutative monoidal groupoids.
This means that each $B\Cut_n$ is a partially defined commutative topological
monoid and the set of isomorphism classes $\pi_0\Cut_n := \pi_0(B\Cut_n)$ is a commutative monoid.
We therefore have a canonical map $|B\Cut_\cd| \to |\pi_0\Cut_\cd|$,
which is compatible with the partially defined multiplication.
This describes the three vertical maps in the diagram.

In the left-hand square the top vertical map is the composite 
\[
    B\ICobn \hookrightarrow |BD_\cd| \xrightarrow{R_\cd} |B\Cut_\cd^{(g\ge2)}|
\]
(where $R_\cd$ is the quotient map from definition \ref{defn:Cut})
and the bottom map is defined by passing to the quotient where isomorphic
$n$-simplices are identified. 
One checks that this is well-defined, and the square commutes automatically.

For the square in the middle the top vertical map is defined by summing over connected 
components, as described in the following lemma, where we also show that it
is a rational equivalence. The bottom vertical map is again defined by 
passing to the quotient.
\begin{lem}\label{lem:map-that-sums}
    The map
    \begin{align*}
        &\qquad \qquad \qquad |B\Cut_\cd^{(g \ge 2)}|  \longrightarrow 
        \SP^\infty(|B\Cut_\cd^{\con, (g \ge 2)}|) \\
    & [(\emptyset \xrightarrow{[W_0]} M_0 \to \dots \to M_n\xrightarrow{[W_{n+1}]} \emptyset),
    (t_0,\dots,t_n)]  \\
    & \longmapsto 
    \sum_{U \subset W_0 \cup \dots \cup W_{n+1}} 
    [(\emptyset \xrightarrow{[W_0\cap U]} M_0\cap U \to \dots \to M_n\cap U\xrightarrow{[W_{n+1}\cap U]} \emptyset),
    (t_0,\dots,t_n)] 
    \end{align*}
    is a rational homotopy equivalence.
    Here the sum runs over all connected components $U$
    of the closed surface $W_0\cup\dots\cup W_{n+1}$.
\end{lem}
\begin{proof}
    The map described is compatible with the partially defined commutative monoid 
    structure and hence induces a map of $\gC$-spaces.
    It suffices to check that the map is an isomorphism on rational homology
    (because both sides are equivalent to infinite loop spaces, 
    or alternatively because both sides are simply connected).
    By corollary \ref{cor:HBCut-is-free} the rational homology of 
    $|B\Cut_\cd^{(g\ge2)}|$ is freely generated by the rational 
    homology of $|B\Cut_\cd^{\con,(g\ge2)}|$.
    (While corollary \ref{cor:HBCut-is-free} does not have the $g \ge 2$ restriction, imposing this restriction defines a simplicial subspace of both $B\Cut_\cd$ and $B\Cut_\cd^\con$, and the same arguments apply to these subspaces.)
    The same is true for $\SP^\infty(|B\Cut_\cd^{\con,(g\ge2)}|)$,
    so all we need to check is that the following composite
    \[
        {|B\Cut_\cd^{\con,(g \ge 2)}|}  \longrightarrow
        {|B\Cut_\cd^{(g \ge 2)}|}  \longrightarrow
        \SP^\infty(|B\Cut_\cd^{\con, (g \ge 2)}|) 
    \]
    is the canonical inclusion.
    This follows from inspection: the sum $\sum_{U\subset W_0 \cup \dots}$ only has a single summand
    as $W_0 \cup \dots \cup W_{n+1}$ is connected.
\end{proof}

To complete diagram \ref{eqn:mu-diagram} we only have to construct
the commutative triangle on the right in such a way 
that the diagonal map is a rational equivalence.
This will be obtained by applying $\SP^\infty$ to the following diagram.
\[
\begin{tikzcd}
    {|B\Cut_\cd^{\con, g\ge 2}|} \ar[r, "\cong"] \ar[d] & 
    {\bigvee_{g \ge 2} S^2|B\Cut_\cd^{F, (g)}|} \ar[r, "\simeq"] & 
    {\bigvee_{g \ge 2} S^2|BN_\cd^{\rm R} \mcF_{\Sigma_g}(\Cobn)|} \ar[dl, dashed, "\simeq_\IQ"'] \\
    {|\pi_0\Cut_\cd^{\con, (g \ge 2)}|} \ar[r, dashed]
    & \bigvee_{g \ge 2} S^2(\Delta_g) &
    \bigvee_{g \ge 2} S^2B\mcF_{\Sigma_g}(\Cobn) \ar[u, "\simeq"'] \ar[l, "S^2(\Phi)", "\simeq_\IQ"'] 
\end{tikzcd}
\]
The horizontal maps in the top row and the right vertical map are the equivalences constructed 
in lemma \ref{lem:Cut-con=gS} and \ref{lem:CutF=F}.
The map $\Phi$ is taken from definition \ref{defn:Phi}.
The diagonal dashed map exists because $\Phi:|N_\cd \mcF_{\Sigma_g}(\Cobn)| \to \Delta_g$
already maps any two isomorphic $n$-simplices in $N_\cd^{\rm R} \mcF_{\Sigma_g}(\Cobn)$
to $\Delta_g$ in the same way.
The horizontal dashed map exists for the same reason.

To conclude the proof of theorem \ref{thm:mu-fits} we combine the formula
of lemma \ref{lem:map-that-sums} and lemma \ref{lem:Cut-con=gS} observe that the composite
\[
    {B\Cobn} \longrightarrow
    {|\pi_0\Cut_\cd^{(g \ge 2)}|} \longrightarrow
    \SP^\infty(|\pi_0\Cut_\cd^{\con, (g \ge 2)}|) \longrightarrow
    \SP^\infty(\bigvee_{g \ge 2} S^2(\Delta_g))
\]
is exactly the $\mu$ described in definition \ref{defn:mu}.

\subsection*{Conflicts of Interest and Data Availability statements}
There are no conflicts of interest.
This manuscript has no associated data.

\printbibliography[heading=bibintoc]

  \bigskip
  \footnotesize

  Jan Steinebrunner, 
  \textsc{
    Gonville \& Caius College,
    University of Cambridge,
    UK.
    }\par\nopagebreak
\textit{E-mail address}: \texttt{js2675@cam.ac.uk}

\end{document}